\documentclass[english]{article}
\usepackage[T1]{fontenc}
\usepackage[latin9]{inputenc}
\usepackage{amsmath,amsthm} 
\usepackage{graphicx}
\usepackage{amssymb}
\usepackage{enumitem}
\usepackage{natbib} 
\usepackage{bussproofs}
\usepackage[letterpaper]{geometry} 
\usepackage{appendix}

\usepackage[makeroom]{cancel}

 \usepackage{float}

\input{fitch.sty}
\usepackage{bussproofs}

\usepackage{mathrsfs} 

\newcommand{\lat}{\mathfrak{L}}

\usepackage{xcolor}

\usepackage{thmtools,thm-restate}

\usepackage{tikz}
\usepackage{pgf} 
\usetikzlibrary{patterns,automata,arrows,shapes,snakes,topaths,trees,backgrounds,positioning,through,calc,arrows.meta}

\usepackage{setspace}
\usepackage{multicol}
\usepackage{graphicx}

\setcitestyle{round}

\definecolor{darkgreen}{rgb}{0.0, 0.7, 0.0}

\usepackage{changes}
\colorlet{Changes@Color}{darkgreen}

\theoremstyle{definition}
\newtheorem{theorem}{Theorem}[section]

\newtheorem{proposition}[theorem]{Proposition}

\newtheorem{definition}[theorem]{Definition}

\newtheorem{lemma}[theorem]{Lemma}

\newtheorem{example}[theorem]{Example}

\newtheorem{remark}[theorem]{Remark}

\makeatletter  

\usepackage{stmaryrd}  
\usepackage{comment}
\usepackage{hyperref}
\hypersetup{colorlinks=true, citecolor=blue, linkcolor=blue}  

\newcommand{\comp}{\vartriangleleft}
\newcommand{\compflip}{\vartriangleright}

\makeatother

\usepackage{babel} 
\onehalfspace

\begin{document} 

\title{A fundamental non-classical logic}

\author{Wesley H. Holliday \\ University of California, Berkeley}

\date{{\normalsize Published in \textit{Logics}, Vol.~1, No.~1, 36-79, 2023.}}

\maketitle

\begin{abstract} We give a proof-theoretic as well as a semantic characterization of a logic in the signature with conjunction, disjunction, negation, and the universal and existential quantifiers that we suggest has a certain fundamental status. We present a Fitch-style natural deduction system for the logic that contains only the introduction and elimination rules for the logical constants. From this starting point, if one adds the rule that Fitch called Reiteration, one obtains a proof system for intuitionistic logic in the given signature; if instead of adding Reiteration, one adds the rule of Reductio ad Absurdum, one obtains a proof system for orthologic; by adding both Reiteration and Reductio, one obtains a proof system for classical logic. Arguably neither Reiteration nor Reductio is as intimately related to the meaning of the connectives as the introduction and elimination rules are, so the base logic we identify serves as a more fundamental starting point and common ground between proponents of intuitionistic logic, orthologic, and classical logic. The algebraic semantics for the logic we motivate proof-theoretically is based on bounded lattices equipped with what has been called a weak pseudocomplementation. We show that such lattice expansions are representable using a set together with a reflexive binary relation satisfying a simple first-order condition, which yields an elegant relational semantics for the logic. This builds on our previous study of representations of lattices with negations, which we extend and specialize for several types of negation in addition to weak pseudocomplementation. Finally, we discuss ways of extending these representations to lattices with a conditional or implication operation.\end{abstract}

\hspace{.14in}\textbf{Keywords:} natural deduction, introduction and elimination rules, lattices with negation, 

\hspace{.886in} lattices with implication, representation of lattices, intuitionistic logic, orthologic

\hspace{.14in}\textbf{MSC:} 03B20, 03G10, 06B15, 06B23, 06C15, 06D15, 06D20

\section{Introduction}\label{Intro}

According to an influential strand of proof theory and philosophy of language, the meaning of the logical connectives is given by their introduction and elimination rules (or just by the introduction rules, from which the elimination rules are thought to follow; see, e.g., \citealt[\S~5.13]{Gentzen1935}, \citealt[\S~4]{Prawitz1973},  \citealt[Chs.~11-13]{Dummett1991}, \citealt{Schroeder-Heister2018}). Prior \citeyearpar{Prior1960} explains a version of the view as follows:
\begin{quote}
[I]f we are asked what is the meaning of the word `and', at least in the purely conjunctive sense (as opposed to, e.g., its colloquial use to mean `and then'), the answer is said to be \textit{completely} given by saying that (i) from any pair of statements P and Q, we can infer the statement formed by joining P to Q with `and' (which statement we hereafter describe as `the statement P-and-Q'), that (ii) for any conjunctive statement P-and-Q we can infer P, and (iii) from P-and-Q we can always infer Q. Anyone who has learnt to perform these inferences knows the meaning of `and', for there is simply nothing more \textit{to} knowing the meaning of `and' than being able to perform these inferences. (p.~38)
\end{quote}
Without going nearly so far as to claim that the ability to follow the introduction and elimination rules is all there is to grasping the meaning of `and', one can still appreciate that the validity of the introduction and elimination rules is a central semantic fact about `and'.

Logicians motivated by proof-theoretic accounts of the meaning of the connectives have tended to favor intuitionistic logic over classical logic on the grounds that the classical rule of Reductio ad Absurdum (if the assumption of $\neg\varphi$ leads to a contradiction, conclude $\varphi$) allegedly cannot be justified on the basis of the meaning of negation in the way that the introduction and elimination rules for negation can be (see \citealt[\S~5.3]{Gentzen1935}, \citealt[pp.~291-300]{Dummett1991}, \citealt[\S~1.2]{Dummett2000}). In fact, one can go further and argue that even intuitionistic logic goes beyond what can be justified on the basis of the meaning of the connectives. For example, in recent work in the formal semantics of natural language (\citealt{Mandelkern:2018a}, \citealt{HM2022}), it has been argued that the distributive law of classical and intuitionsitic logic, according to which $\varphi\wedge (\psi\vee \chi)$ entails $(\varphi\wedge \psi)\vee (\varphi\wedge \chi)$, is invalid for fragments of language that include the epistemic modals `might' ($\Diamond$) and `must' ($\Box$).  First, there is extensive evidence that sentences of the form 
\begin{itemize}
\item[(1)] It's raining but it might not be raining ($p\wedge\Diamond\neg p$) 
\end{itemize}
are contradictory  (see, e.g., \citealt{GSV:1996}, \citealt{Aloni:2000}, \citealt{Yalcin2007}, \citealt{Mandelkern:2018a}, \citealt{HM2022}), not merely pragmatically infelicitous to assert.\footnote{This is in contrast to `It's raining but I don't know it', which is infelicitous to assert but does not embed like a contradiction; e.g., it is fine in the antecedent of a conditional such as `If it's raining but I don't know it, I'll be surprised when I get wet'. For  a review of evidence that the badness of (1) is not merely pragmatic, see \citealt[\S~2.1]{HM2022}.} As discussed in \citealt{HM2022}, if we accepted the distributive law, then from the banal expression of ignorance that
\begin{itemize}
\item[(2)] either it's raining or it's not, and it might be raining and it might not be raining ($(p\vee\neg p)\wedge \Diamond p\wedge\Diamond\neg p$)
\end{itemize}
we could draw the absurd conclusion that
 \begin{itemize}
 \item[(3)] it's raining and it might not be, or it's not raining and it might be ($(p\wedge \Diamond\neg p)\vee (\neg p\wedge\Diamond p)$),
 \end{itemize}
 which is a disjunction of two contradictions and therefore a contradiction. 
 
 One might think that the distributive law can be justified using the introduction and elimination rules for conjunction and disjunction, but this depends on the precise formulation of those rules. In particular, one must be careful to distinguish between what could be called Proof by Cases, the principle that
 \begin{itemize}
 \item if $\varphi\vdash\chi$ and $\psi\vdash\chi$, then $\varphi\vee\psi\vdash \chi$,
 \end{itemize}
  and what could be called Proof by Cases with Side Assumptions, the principle that 
  \begin{itemize}
  \item if $\alpha\wedge\varphi\vdash\chi$ and $\alpha\wedge\psi\vdash\chi$, then $\alpha\wedge (\varphi\vee\psi)\vdash\chi$, or
  \item if $\alpha,\varphi\vdash\chi$ and $\alpha,\psi\vdash\chi$, then $\alpha, (\varphi\vee\psi)\vdash\chi$.
  \end{itemize} 
 If one takes the elimination rule for disjunction to be Proof by Cases with Side Assumptions, then the distributive law is derivable using the introduction and elimination rules for the connectives. But if one takes the elimination rule for disjunction to be Proof by Cases, it is not.\footnote{On the importance of this distinction concerning side assumptions in relation to the idea that the introduction and elimination rules for $\vee$ should be in ``harmony'' with each other, see \citealt[p.~229]{Rumfitt2017}.}
 
 The point can be made in an illuminating way in a Fitch-style natural deduction system (\citealt{Fitch1952,Fitch1966}).  Figure~\ref{FirstFigure} shows a Fitch-style natural deduction of the absurd (3) above from the banal (2). The ``mistake'' in the proof lies in the Reiteration steps on lines 7 and 11: we should not be allowed to reiterate the assumption that \textit{might~$\neg p$} into a subproof where we have just assumed $p$ or reiterate the assumption that \textit{might $p$} into a subproof where we have just assumed $\neg p$! From this perspective, the problematic principle of a Fitch-style natural deduction system when the language contains `might' is the rule of Reiteration, not the rule of $\vee$ elimination. Reiteration also leads to the \textit{pseudocomplementation} principle that if $\varphi\wedge\psi\vdash\bot$, then $\psi\vdash\neg\varphi$. But this principle is unacceptable for a language containing `might', since $p\wedge\Diamond \neg p$ is contradictory and yet $\Diamond \neg p$ (`it might not be raining') plainly does not entail $\neg p$ (`it's not raining') (\citealt{Yalcin2007}). For a battery of further arguments against distributivity, pseudocomplementation, and other laws to which Reiteration leads, in the context of a language with epistemic modals, see \citealt{HM2022}. 
 
  \begin{figure}
 \[\begin{nd}
\hypo [1] {1} {(p\vee\neg p)\wedge(\Diamond p\wedge\Diamond\neg p)} 
\have [2] {2} {p\vee\neg p}\ae{1} 
\have [3] {3} {\Diamond p\wedge\Diamond \neg p}\ae{1} 
\have [4] {4} {\Diamond p} \ae{3}
\have [5] {5} {\Diamond \neg p} \ae{3}
\open
\hypo [6] {6} {p}
\have [7] {7} {\Diamond \neg p} \r{5}
\have [8] {8} {p\wedge \Diamond \neg p} \ai{6,7}
\have [9] {9} {(p\wedge \Diamond \neg p)\vee (\neg p \wedge \Diamond p)} \oi{8}
\close
\open
\hypo [10] {10} {\neg p}
\have [11] {11} {\Diamond p} \r{4}
\have [12] {12} {\neg p \wedge \Diamond p} \ai{10,11}
\have [13] {13} {(p\wedge \Diamond \neg p)\vee (\neg p \wedge \Diamond p)} \oi{12}
\close
\have [14] {14} {(p\wedge \Diamond \neg p)\vee (\neg p\wedge\Diamond p)} \oe{2,6-9, 10-13}
\end{nd}\]
\caption{An illustration of the problem with Reiteration in a language with epistemic modals.}\label{FirstFigure}
\end{figure}
 
 For the purposes of the present paper, it is enough for the reader to find the project of going to a weaker logic without distributivity or pseudocomplementation to be an interesting one. Denying these principles is familiar from \textit{quantum logic} (see \citealt{Chiara2002}), but the orthomodularity principle of quantum logic also appears to be invalid for fragments of natural language containing `might' and `must' (\citealt{HM2022}). Thus, we are interested in the weaker system of \textit{orthologic} (\citealt{Goldblatt1974}), though we weaken it even further by following the intuitionists in dropping Reductio ad Absurdum.  In addition to the criticisms of Reductio for enabling nonconstructive proofs (\citealt{Troelstra1988a}), there are arguments to the effect that Reductio and the principle of excluded middle to which it leads should be rejected for a language with \textit{vague} predicates (see, e.g., \citealt{Wright2001}, \citealt{Field2003}, \citealt{Bobzien2020}). In any case, here we drop Reductio not on ideological grounds but rather to find a neutral base logic.

In this paper, we begin in \S~\ref{FitchSection} with a Fitch-style natural deduction system for a propositional logic in the signature with conjunction, disjunction, and negation that contains only the introduction and elimination rules for the connectives. We defer the  addition of the universal and existential quantifiers with their introduction and elimination rules to \S~\ref{QuantSection}. Starting from the system we define,  if one adds Fitch's rule of  Reiteration, one obtains a proof system for intuitionistic logic in the given signature, defined in Appendix~\ref{AppendixA}; if instead of adding Reiteration, one adds the rule of Reductio ad Absurdum, one obtains a proof system for orthologic; by adding both Reiteration and Reductio, one obtains a proof system for classical logic. Arguably neither Reiteration nor Reductio is as intimately related to the meaning of the connectives as the introduction and elimination rules are, so the base logic we identify serves as a more fundamental starting point and common ground between proponents of intuitionistic logic, orthologic, and classical logic. In \S~\ref{AlgSection}, we turn to the algebraic semantics for the logic, which is based on bounded lattices equipped with what has been called a weak pseudocomplementation. In \S~\ref{RelationalSection}, we show that such lattice expansions are representable using a set together with a reflexive binary relation satisfying a simple first-order condition, which yields an elegant relational semantics for the logic. This builds on our previous study of representations of lattices with negations (\citealt{Holliday2022}), which we extend and specialize for several types of negation in addition to weak pseudocomplementation. In \S~\ref{QuantSection}, we use one of our representation theorems to prove completeness with respect to relational semantics of the extension of the logic from \S~\ref{FitchSection} with quantifiers.  In \S~\ref{Conditionals} and Appendix~\ref{AppendixB},  we discuss ways of extending our representational approach to lattices with a conditional or implication operation. Finally, in \S~\ref{ConclusionSection}, we conclude with a brief summary and look ahead. 

Several Jupyter notebooks with code to check proofs and to construct algebras from relational frames and relational representations of algebras are available at \href{https://github.com/wesholliday/fundamental-logic}{github.com/wesholliday/fundamental-logic}.

\begin{remark}Though our argument against distributivity involved modals, we do not include modals in our language in this paper. A  modal version of the fundamental logic defined in \S~\ref{FitchSection} can be studied using ideas from \citealt{HM2022} and \citealt{Holliday2022}, but we will not do so here. As a result, our formal system will not reflect an important point: setting aside issues from quantum mechanics, as far as we can tell from natural language, distributivity is valid for sentences not including modals (or conditionals). However, in this paper, we take atomic sentences $p,q,r,\dots$ to be genuine \textit{propositional variables}, standing in for arbitrary propositions (cf.~\citealt[pp.~147-8]{Burgess2003}); thus, the failure of distributivity for modal propositions implies that we cannot accept $p\wedge (q\vee r)\vdash (p \wedge q)\vee (p\wedge r)$ as a schematically valid principle. By enriching the language, one can define a system in which Reiteration and hence distributivity hold for special non-modal propositions but not for modal propositions (see \citealt{HM2022}). But in this paper, the rules of the fundamental logic are supposed to be schematically valid principles holding for all propositions. 
\end{remark}

\begin{remark} The relational semantics in \S~\ref{RelationalSection} covers logics much weaker than the fundamental logic of \S~\ref{FitchSection}, including paraconsistent logics in the spirit of Battilotti and Sambin's \citeyearpar{Battilotti1999} \textit{basic logic}, which (in a fragment of its language) is a sublogic of fundamental logic without $\varphi\wedge\neg\varphi\vdash\psi$, $\varphi\vdash\neg (\psi\wedge\neg\psi)$, or $\bot\vdash\psi$ (though we do not have a primitive $\bot$ in our language). In fact, we can cover logics as weak as the logic of lattices with an antitone unary operation $\neg$ (Theorem \ref{NegThmAntitone}). Note that ``fundamental'' is not supposed to indicate that the logic of \S~\ref{FitchSection} is as weak as possible but rather that it has a special status based on introduction and elimination rules insofar as the only gap between this logic and intuitionistic logic (resp.~orthologic) in the relevant signature is Reiteration (resp.~Reductio). Of course, Kolmogorov \citeyearpar{Kolmogorov1925} and others have questioned the explosion principle $\varphi\wedge\neg\varphi\vdash\psi$ of intuitionistic logic. However, for inference in natural language, $\psi\vee(\varphi\wedge\neg\varphi)\vdash\psi$ appears acceptable, and this is equivalent to explosion given the rules for $\vee$. In any case, readers interested in weaker logics can focus on our semantics for those logics.\end{remark}

\section{Fitch-style natural deduction}\label{FitchSection}

Given a nonempty set $\mathsf{Prop}$ of propositional variables, our propositional language $\mathcal{L}$ is given by the grammar
\[\varphi::= p\mid \neg\varphi\mid (\varphi\wedge\varphi)\mid (\varphi\vee\varphi)\]
where $p\in\mathsf{Prop}$. As abbreviations, we define $\bot :=(p\wedge\neg p)$ and $\top:=\neg \bot$.

We will define when a formula $\psi$ is provable from a formula $\varphi$, denoted $\varphi\vdash_\mathsf{F}\psi$, using a Fitch-style natural deduction system (\citealt{Fitch1952,Fitch1966}, based on \citealt{Jaskowski1934}). We chose `$\mathsf{F}$'  for \textit{fundamental logic} or rather \textit{fundamental propositional logic}, as we introduce a first-order extension in \S~\ref{QuantSection}. To represent an argument with multiple assumptions, conjoin the assumptions with $\wedge$ into a single formula $\varphi$. We chose Fitch-style natural deduction in part because we agree that it ``corresponds more closely to proofs in ordinary mathematical practice'' (\citealt[p.~134]{Geuvers2004}) and ``is more faithful to the phenomenology of reasoning'' (\citealt[p.~1110]{Hazen2014})  than Gentzen-style natural deduction. Although the idea that the meaning of the connectives is given by introduction and elimination rules is usually formulated in proof theory in terms of Gentzen rules, the view described by Prior in the quotation in \S~\ref{Intro} can certainly be formulated in terms of Fitch rules; indeed, referring to the introduction and elimination rules for negation as in \citealt{Fitch1966}, Hazen and Pelletier \citeyearpar{Hazen2014} write that ``they have as good a claim as any Gentzen-ish pair to specify uniquely the meaning of the connective they govern'' (p.~1114).

We depart from Fitch in dropping his rules of Reiteration and double negation elimination  (\citealt{Fitch1966}).  A proof will be a sequence of formulas and possibly other proofs, defined inductively below. Every proof begins with one formula, considered its assumption (even if this is just $\top$). When diagramming proofs as in Figure \ref{FirstFigure}, we adopt Fitch's convention of drawing a horizontal line under the assumption of a proof. We regard a one formula proof $\langle\varphi\rangle$ as having $\varphi$ as both its assumption and its conclusion, diagrammed as follows:
\[\begin{nd}
\hypo [\,]  {1} {\varphi}
\have [\,] {2} {\varphi}
\end{nd}\]
We allow proofs that do not end with a conclusion formula (which could be called ``partial proofs'') but we define the provability relation $\vdash_\mathsf{F}$ as follows: $\varphi\vdash_\mathsf{F} \psi$ if there exists a proof beginning with $\varphi$ and ending with $\psi$. For those familiar with Fitch-style natural deduction, the rules of our system are shown in Figure~\ref{FitchRules}. 

\begin{figure}[!p]
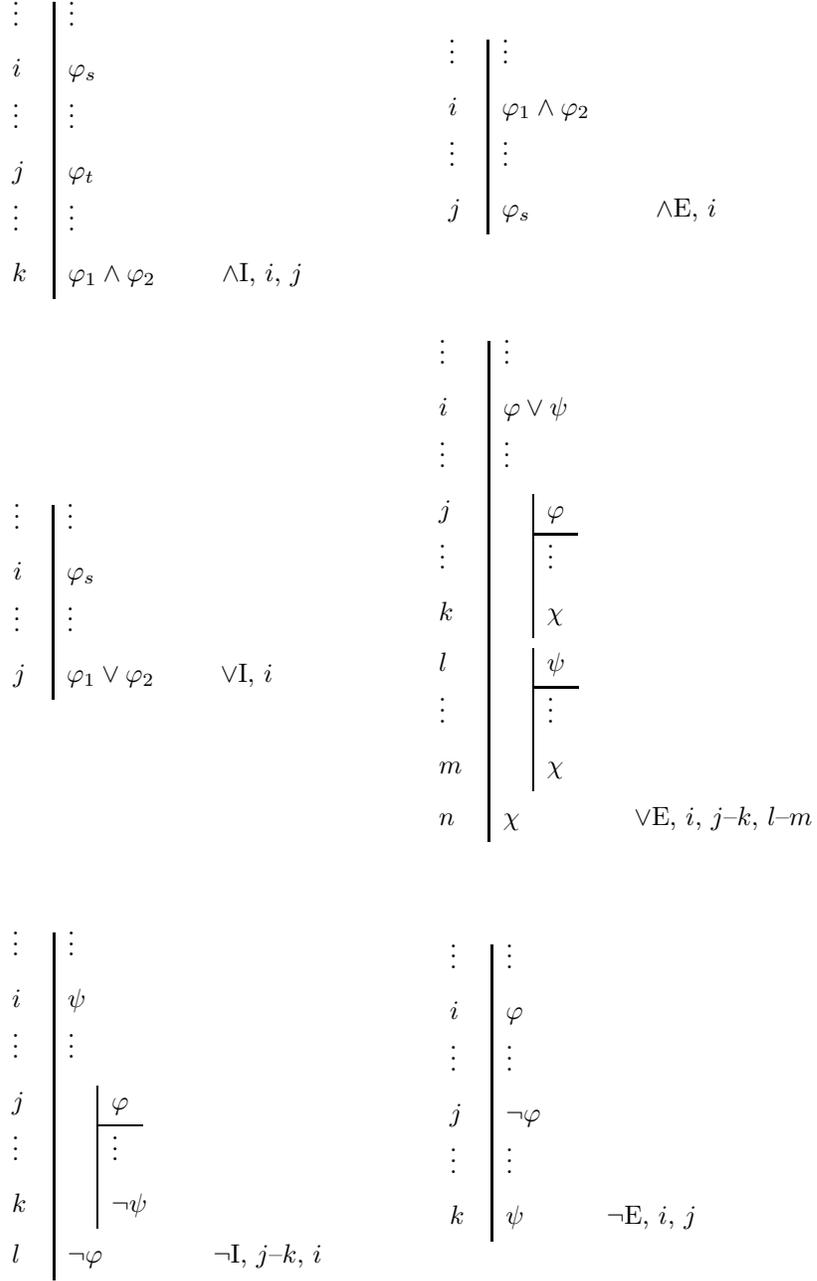

\begin{center}

\begin{minipage}{2in}
\[\begin{nd}
\have [\vdots] {4} {\vdots}
\have [i] {5}  {\varphi_s}
\have [\vdots] {8}   {\vdots}
\have [j] {9} {\varphi_t}
\have [\vdots] {10}   {\vdots}
\have [k] {11}   {\varphi_1\wedge\varphi_2}\ai{5,9}
\end{nd}\]
\end{minipage}\;\;\,\begin{minipage}{2.25in}
\[\begin{nd}
\have [\vdots] {4} {\vdots}
\have [i] {5}  {\varphi_1\wedge\varphi_2}
\have [\vdots] {8}   {\vdots}
\have [j] {9} {\varphi_s} \ae{5}
\end{nd}\]
\end{minipage}
\begin{minipage}{2.45in}
\[\begin{nd}
\have [\vdots] {1} {\vdots} 
\have [i] {2} {\varphi_s} 
\have [\vdots] {3} {\vdots} 
\have [j] {4} {\varphi_1\vee\varphi_2}\oi{2}
\end{nd}\]
\end{minipage}\;\;\,\begin{minipage}{2.4in}

\[\begin{nd}
\have [\vdots] {} {\vdots} 
\have [i] {0} {\varphi\vee\psi}
\have [\vdots] {1} {\vdots} 
\open
\hypo [j] {2} {\varphi}
\have [\vdots] {3} {\vdots}
\have [k] {5} {\chi}
\close
\open
\hypo [l] {7}  {\psi}
\have [\vdots] {8} {\vdots}
\have [m] {9} {\chi}
\close
\have [n] {n} {\chi} \oe{0,2-5,7-9}
\end{nd}\]\vspace{.1in}
\end{minipage}

\begin{minipage}{2in}
\[\begin{nd}
\have [\vdots] {0} {\vdots}
\have [i] {3}   {\psi}
\have [\vdots] {}  {\vdots}
\open
\hypo [j] {1} {\varphi}
\have [\vdots] {4}   {\vdots}
\have [k] {5}   {\neg\psi}
\close
\have [l] {6} {\neg\varphi} \ni{1-5,3}
\end{nd}\]
\end{minipage}\begin{minipage}{2.25in}
\[\begin{nd}
\have [\vdots] {4} {\vdots}
\have [i] {5}  {\varphi}
\have [\vdots] {6} {\vdots}
\have [j] {7}  {\neg\varphi}
\have [\vdots] {8} {\vdots}
\have [k] {9}{\psi} \ne{5, 7}
\end{nd}\]
\end{minipage}
\end{center}
\caption{Rules of a Fitch-style proof system for the logic, where $s,t\in\{1,2\}$.}\label{FitchRules}
\end{figure}

A rigorous inductive definition is as follows.\footnote{To avoid ambiguity, assume formulas are constructed in such a way that no formula is a sequence beginning with a formula.}  The set of proofs is the smallest set containing for each formula $\varphi$ the sequence  $\langle \varphi\rangle$ and satisfying the following closure conditions for $1\leq i,j\leq n$:
\begin{itemize}
\item If $\langle \sigma_1,\dots,\sigma_n\rangle$ is a proof and $\tau$ is a proof, then $\langle \sigma_1,\dots,\sigma_n,\tau\rangle$ is a proof.
\item If $\langle \sigma_1,\dots,\sigma_n\rangle$ is a  proof and $\sigma_i,\sigma_j$  are formulas, then $\langle \sigma_1,\dots,\sigma_n,\sigma_i\wedge\sigma_j\rangle$ is a proof ($\wedge$I).
\item If $\langle \sigma_1,\dots,\sigma_n\rangle$ is a  proof and $\sigma_i$ is a formula of the form $\varphi\wedge\psi$, then $\langle \sigma_1,\dots,\sigma_n,\varphi\rangle$ and $\langle \sigma_1,\dots,\sigma_n,\psi\rangle$ are proofs ($\wedge$E).
\item If $\langle \sigma_1,\dots,\sigma_n\rangle$ is a  proof and  $\sigma_i$ is a formula, then for any formula $\varphi$,  both $\langle \sigma_1,\dots,\sigma_n,\sigma_i\vee\varphi\rangle$ and $\langle \sigma_1,\dots,\sigma_n,\varphi\vee \sigma_i\rangle$ are proofs ($\vee$I).
\item If $\langle \sigma_1,\dots,\sigma_n\rangle$ is a  proof, $\sigma_i$ is a formula of the form $\varphi\vee\psi$, $\sigma_{n-1}$ is a sequence beginning with $\varphi$ and ending with $\chi$, and $\sigma_n$ is a sequence beginning with $\psi$ and ending with $\chi$, then $\langle \sigma_1,\dots,\sigma_n, \chi\rangle$ is a proof ($\vee $E).
\item If $\langle \sigma_1,\dots,\sigma_n\rangle$ is a  proof, $\sigma_i$ is a formula  $\psi$, and $\sigma_n$ is a sequence beginning with $\varphi$ and ending with $\neg\psi$, then $\langle \sigma_1,\dots,\sigma_n,\neg\varphi\rangle$ is a proof ($\neg$I).
\item If $\langle \sigma_1,\dots,\sigma_n\rangle$ is a  proof and $\sigma_i$ and $\sigma_j$ are formulas of the form $\varphi$ and $\neg\varphi$, respectively, then for any formula $\psi$, $\langle \sigma_1,\dots,\sigma_n,\psi\rangle$ is a proof ($\neg$E). 
\end{itemize}
Note that for any proof $\langle \sigma_1,\dots,\sigma_n\rangle$, $\sigma_1$ is a formula and all later $\sigma_i$ are either formulas or proofs. Also note that when diagramming proofs, we follow Fitch and include line numbers that justify a given rule application, but these data are not needed as official parts of a proof, just as they are not needed in Hilbert-style proofs. Whether a sequence is a proof is clearly decidable by an algorithm.

Our introduction and elimination rules for $\wedge$ and $\vee$ and our elimination rule for $\neg$ match those of \citealt{Fitch1966}. However, our introduction rule for $\neg$ is not exactly the same as his. Our $\neg$ introduction rule says that 
\begin{itemize}
\item[] \textit{if from the assumption of $\varphi$, you derive the negation of another formula derived just before the assumption, then conclude $\neg\varphi$}.
\end{itemize}
This formulation of $\neg$ introduction is admissible in Fitch's system, thanks to his Reiteration rule; but Fitch \citeyearpar{Fitch1966} states his $\neg$ introduction rule in a way that requires a pair of contradictory formulas to appear in the subproof that starts with $\varphi$.\footnote{Note that if one does derive a pair of contradictory formulas in a subproof that starts with $\varphi$, then by $\neg$E one can derive the negation of a formula derived just before the assumption of the subproof, so our $\neg$I rule is applicable.} To accomplish what we accomplish with $\neg$I, Fitch would reiterate $\psi$ into the subproof beginning with $\varphi$ to obtain a contradiction between $\psi$ and $\neg\psi$ within the subproof. But we can disassociate Reiteration, which we do not allow (recall the cautionary Figure \ref{FirstFigure}), from $\neg$ introduction. The idea of Reiteration is that if $\psi$ was derived just before a subproof beginning with $\varphi$, then \textit{$\psi$ still holds under the assumption of $\varphi$}. By contrast, when applying our $\neg$I rule, we prove that the \textit{negation} of $\psi$ \textit{holds under the assumption of $\varphi$}, and then since we know that \textit{$\psi$ holds prior to the assumption of $\varphi$}, we deduce $\neg\varphi$.\footnote{Note that our $\neg$I rule produces proofs of the form $\langle\dots, \psi,\dots,\langle \varphi,\dots,\neg\psi\rangle,\neg\varphi \rangle$ but not of the form $\langle \dots,\psi,\dots, \langle \chi,\dots ,\langle \varphi,\dots,\neg\psi\rangle, \neg\varphi\rangle\rangle$ (where $\psi$ is not an element of the subproof beginning with $\chi$). If we were to strengthen $\neg$I so as to allow the intervention of the additional assumption $\chi$ as in the preceding sequence, then we could commit the same mistakes to which Reiteration leads as in \S~\ref{Intro}. Indeed, we could reiterate any negated formula into a subproof: given a formula $\neg\psi$ immediately preceding a subproof $\sigma$ beginning with $\chi$, to reiterate $\neg\psi$ into $\sigma$, create a subproof $\sigma'$ inside $\sigma$ such that $\sigma'$ begins with $\psi$, from which we can prove $\neg\neg\psi$, contradicting the $\neg\psi$ occurring before the assumptions of $\chi$ and $\psi$, in which case the strengthened rule would allow us to conclude $\neg\psi$ after $\sigma'$. Then a restricted version of pseudocomplementation would hold: if $\varphi\wedge\neg\psi\vdash\bot$, then $\neg\psi\vdash\neg\varphi$. But then from the fact that $\Diamond \neg p\wedge\neg\neg p$ (``It might  be that it isn't raining, but it's not the case that it isn't raining'') is contradictory, we could prove using the restricted version of pseudocomplementation and other properties of negation that $\Diamond\neg p\vdash\neg p$.}  

Let us relate our Fitch-style proof system to a \textit{binary logic} in the sense of \citealt{Goldblatt1974}. The following definition differs from Goldblatt's definition of an \textit{orthologic} only in dropping $\neg\neg\varphi\vdash\varphi$ and  adding rules for $\vee$, which for us is not definable in terms of $\wedge$ and $\neg$. Similarly, a sequent calculus presentation can be obtained from Cutland and Gibbins' \citeyearpar[\S~3]{Cutland1982} sequent calculus for orthologic by dropping their rule $\neg\neg\to$.

\begin{definition}\label{BinaryLogic} An \textit{intro-elim logic} is a binary relation $\vdash \,\subseteq\mathcal{L}\times\mathcal{L}$ such that for all $\varphi,\psi,\chi\in\mathcal{L}$:

\begin{center}
\begin{tabular}{ll}
1. $\varphi\vdash\varphi$ & 8. if $\varphi\vdash\psi$ and $\psi\vdash\chi$, then $\varphi\vdash\chi$  \\
2. $\varphi\wedge\psi\vdash\varphi$ & \\
3.  $\varphi\wedge\psi\vdash\psi$ & 9. if $\varphi\vdash \psi$ and $\varphi\vdash\chi$, then $\varphi\vdash \psi\wedge\chi$\\
4. $\varphi\vdash \varphi\vee\psi$ & \\
5. $\varphi\vdash \psi\vee\varphi$ & 10. if $\varphi\vdash\chi$ and $\psi\vdash\chi$, then $\varphi\vee\psi \vdash \chi$ \\
6. $\varphi\vdash \neg\neg\varphi$ &  \\
7. $\varphi\wedge\neg\varphi\vdash\psi$ \qquad\qquad\qquad& 11. if  $\varphi\vdash\psi$, then $\neg\psi\vdash\neg\varphi$. \\
\end{tabular}
\end{center}
\end{definition}

The following is easy to check.

\begin{proposition}\label{BinaryLogicProp} $\vdash_\mathsf{F}$ is an intro-elim logic.
\end{proposition}
\noindent In fact, we will see that $\vdash_\mathsf{F}$ is the smallest intro-elim logic (Proposition \ref{SmallestIntelim}), which justifies the name of such logics: they all have at least the power of the introduction and elimination rules for the connectives from $\vdash_\mathsf{F}$. Let us highlight the most important, even if obvious, cases of the proof of Proposition \ref{BinaryLogicProp} for our purposes. First is $\varphi\vdash_\mathsf{F}\neg\neg\varphi$, which is shown as follows:
\[\begin{nd}
\hypo [1] {1} {\varphi}
\open
\hypo [2] {2} {\neg\varphi}
\have [3] {3} {\neg\varphi} 
\close
\have [4] {4} {\neg\neg\varphi} \ni{2-3,1}
\end{nd}\]
Next is the property that if $\varphi\vdash_\mathsf{F}\psi$, then $\neg\psi\vdash_\mathsf{F}\neg\varphi$. Assuming we have a  proof from $\varphi$ to $\psi$, we construct a proof from $\neg\psi$ to $\neg\varphi$  as follows:
\[\begin{nd}
\hypo [1] {1} {\neg\psi}
\open
\hypo [2] {2} {\varphi}
\have [\vdots]  {} {\vdots}
\have [n] {3} {\psi} 
\open
\hypo [n+1] {4} {\neg\psi}
\have [n+2] {5} {\neg\psi} 
\close
\have [n+3] {6} {\neg\neg\psi} \ni{4-5,3}
\close
\have [n+4] {7} {\neg\varphi} \ni{2-6,1}
\end{nd}\]
Proving 8-10 of Definition \ref{BinaryLogic} for $\vdash_\mathsf{F}$ also involves gluing together proofs. For 8, given proofs $\langle \varphi,\sigma_1,\dots,\sigma_n,\psi\rangle$ and $\langle \psi,\tau_1,\dots,\tau_m,\chi\rangle$, it is easy to see that $\langle \varphi,\sigma_1,\dots\sigma_n,\psi,\tau_1,\dots,\tau_m,\chi\rangle$ is also a proof. For 9, given proofs $\langle \varphi,\sigma_1,\dots,\sigma_n,\psi\rangle$ and $\langle \varphi,\tau_1,\dots,\tau_m,\chi\rangle$, the sequence $\langle \varphi,\sigma_1,\dots,\sigma_n,\psi, \tau_1,\dots,\tau_m,\chi,\psi\wedge\chi\rangle$ is  a proof. For 10, given proofs $\sigma =\langle \varphi,\sigma_1,\dots,\sigma_n,\chi\rangle$ and $\tau=\langle \psi,\tau_1,\dots,\tau_m,\chi\rangle$, the sequence $\langle \varphi\vee\psi,\sigma,\tau, \chi\rangle$ is  a proof.

Let us mention the three most salient extensions of our logic. First, adding Reductio ad Absurdum as in Figure~\ref{RAAFig} produces a Fitch-style proof system for orthologic, also laid out in \citealt{HM2022}. Equivalently, let $\vdash_\mathsf{O}$ be the smallest intro-elim logic containing $\neg\neg\varphi\vdash\varphi$ for all $\varphi\in\mathcal{L}$. As in the negative translation of classical logic into intuitionistic logic (\citealt{Godel1933}, \citealt{Gentzen1936}), the translation $g$ given by
 \[\mbox{$g(p)=\neg\neg p$, $g(\neg\varphi)=\neg g(\varphi)$, $g(\varphi\wedge\psi)=(g(\varphi)\wedge g(\psi))$, and $g(\varphi\vee\psi)=g(\neg(\neg\varphi\wedge\neg\psi))$}\]
 is  a full and faithful embedding of orthologic into $\vdash_\mathsf{F}$.\footnote{By contrast, a Glivenko-style theorem (\citealt{Glivenko1929}) stating that $\varphi\vdash_\mathsf{O}\psi$ iff $\varphi\vdash_\mathsf{F}\neg\neg\psi$ does not hold, because ${\neg\neg p\wedge \neg\neg q \vdash_\mathsf{O} p\wedge q}$ but $\neg\neg p\wedge\neg\neg q\nvdash_\mathsf{F}\neg\neg (p\wedge q)$, as we show semantically in \S~\ref{AlgSection}.}
\begin{proposition} For all $\varphi,\psi\in\mathcal{L}$, we have $\varphi\vdash_\mathsf{O}\psi$ iff $g(\varphi)\vdash_\mathsf{F}g(\psi)$.
\end{proposition}

\begin{proof} First, an easy induction shows that for all $\varphi\in\mathcal{L}$, $\varphi\vdash_\mathsf{O}g(\varphi)$ and $g(\varphi)\vdash_\mathsf{O}\varphi$. Hence if $\varphi\nvdash_\mathsf{O}\psi$, then $g(\varphi)\nvdash_\mathsf{O}g(\psi)$ and so $g(\varphi)\nvdash_\mathsf{F}g(\psi)$, using that $\vdash_\mathsf{F}$ is the smallest intro-elim logic. For the other direction, we claim that the relation $\vdash_g$ defined by $\varphi\vdash_g\psi$ iff $g(\varphi)\vdash_\mathsf{F}g(\psi)$ is an intro-elim logic such that $\neg\neg\varphi\vdash_g\varphi$ for all $\varphi\in\mathcal{L}$. Then since $\vdash_\mathsf{O}$ is the smallest such logic, $\varphi\vdash_\mathsf{O}\psi$ implies $g(\varphi)\vdash_\mathsf{F}g(\psi)$. 

First, we prove by induction on $\varphi\in\mathcal{L}$ that $\neg\neg\varphi\vdash_g\varphi$. For the base case of $\neg\neg p\vdash_g p$, we need that $\neg\neg\neg\neg p\vdash_\mathsf{F}\neg\neg p$, which follows from $\neg p\vdash_\mathsf{F}\neg\neg\neg p$. For the $\neg$ case of $\neg\neg \neg\varphi\vdash_g \neg\varphi$, we need $\neg\neg\neg g(\varphi)\vdash_\mathsf{F}\neg g(\varphi)$, which follows from $g(\varphi)\vdash_\mathsf{F}\neg\neg g(\varphi)$. For the $\wedge$ case of $\neg\neg (\varphi\wedge\psi)\vdash_g \varphi\wedge\psi$, we need $\neg\neg (g(\varphi)\wedge g(\psi))\vdash_\mathsf{F} g(\varphi)\wedge g(\psi)$. From $g(\varphi)\wedge g(\psi)\vdash_\mathsf{F}g(\varphi)$, we have $\neg g(\varphi)\vdash_\mathsf{F} \neg (g(\varphi)\wedge g(\psi))$ and hence $\neg \neg (g(\varphi)\wedge g(\psi))\vdash_\mathsf{F}\neg\neg g(\varphi)$, so $\neg \neg (g(\varphi)\wedge g(\psi))\vdash_\mathsf{F} g(\varphi)$ by the inductive hypothesis. Similarly, $\neg \neg (g(\varphi)\wedge g(\psi))\vdash_\mathsf{F} g(\psi)$, so we obtain $\neg \neg (g(\varphi)\wedge g(\psi))\vdash_\mathsf{F} g(\varphi)\wedge g(\psi)$. Finally, for the $\vee$ case of $\neg\neg (\varphi\vee\psi)\vdash_g \varphi\vee\psi$, we need $\neg\neg \neg (\neg g(\varphi)\wedge\neg g(\psi))\vdash_\mathsf{F} \neg (\neg g(\varphi)\wedge\neg g(\psi))$, which follows from $\neg g(\varphi)\wedge\neg g(\psi)\vdash_\mathsf{F}\neg \neg (\neg g(\varphi)\wedge\neg g(\psi))$.

Now it is easy to verify that $\vdash_g$ is an intro-elim logic. For condition 10 of Definition \ref{BinaryLogic}, given $\varphi\vdash_g\chi$ and $\psi\vdash_g\chi$, so $g(\varphi)\vdash_\mathsf{F}g(\chi)$ and $g(\psi)\vdash_\mathsf{F}g(\chi)$, we have $\neg g(\chi)\vdash_\mathsf{F}\neg g(\varphi)\wedge\neg g(\psi)$ and hence $\neg(\neg g(\varphi)\wedge\neg g(\psi))\vdash_\mathsf{F}\neg\neg g(\chi)$. It follows by the previous paragraph that $\neg(\neg g(\varphi)\wedge\neg g(\psi))\vdash_\mathsf{F} g(\chi)$, so $\varphi\vee\psi\vdash_g\chi$.
\end{proof}

\begin{figure}[h]
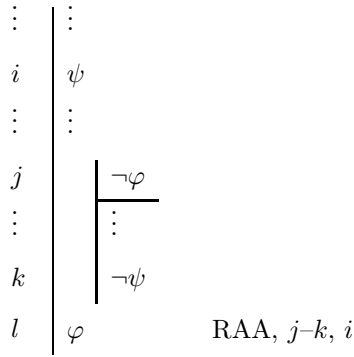

\begin{center}
\begin{minipage}{2.5in}
\[\begin{nd}
\have [\vdots] {0} {\vdots}
\have [i] {3}   {\psi}
\have [\vdots] {4}   {\vdots}
\open
\hypo[ j] {1} {\neg \varphi}
\have [\vdots] {}  {\vdots}
\have [k] {5}   {\neg\psi}
\close
\have [l]{6} {\varphi} \RAA{1-5,3}
\end{nd}
\]\end{minipage}
\end{center}
\caption{The Reductio ad Absurdum rule that turns our proof system into a proof system for orthologic.}\label{RAAFig}
\end{figure}

If instead of Reductio, we add Fitch's rule of Reiteration to $\vdash_\mathsf{F}$, as in Appendix \ref{AppendixA}, then we obtain a Fitch-style proof system for intuitionistic logic in the $\wedge,\vee,\neg$ fragment. Intuitionistic logic in this fragment is the logic of pseudocomplemented distributive lattices (\citealt{Jordi1993}), and using Reiteration we obtain both  pseudocomplementation (see Figure \ref{PseudoFig}) and distributivity (in the style of Figure \ref{FirstFigure}). Finally, adding both Reductio and Reiteration yields a Fitch-style proof system for classical logic (see Appendix~\ref{AppendixA}).

\begin{figure}[h]
\[\begin{nd}
\hypo [1] {1} {\psi}
\open
\hypo [2] {2} {\varphi}
\have [3] {3} {\psi} \r{1}
\have [4] {4} {\varphi\wedge\psi} \ai{2,3} 
\have [\vdots] {}  {\vdots}
\have [n] {5} {\neg\psi} 
\close
\have [n+1] {7} {\neg\varphi} \ni{2-5,1}
\end{nd}\]
\caption{Given a proof from $\varphi\wedge\psi$ to $\bot$, which easily yields a proof from $\varphi\wedge\psi$ to $\neg\psi$, Reiteration would permit the construction of a proof from $\psi$ to $\neg\varphi$.}\label{PseudoFig}
\end{figure}

We briefly note in Figure \ref{GentzenFig} how our points about Reiteration in Fitch-style natural deduction transfer to  Gentzen-style natural deduction (see, e.g., \citealt[\S~3.4]{Chiswell2007}). The introduction and elimination rules for conjunction, the introduction rule for disjunction, and the elimination rule for negation\footnote{We do not have $\bot$ as a primitive in our language, so we formulate $\neg$E as follows: proofs of $\varphi$ and $\neg\varphi$ may be joined with a new root labeled by any formula $\psi$, forming a proof that inherits all the open assumptions of the two proofs.} remain unchanged. We drop RAA from the Gentzen system just as we did from the Fitch system.

\begin{figure}[!ht]
\begin{center}
\begin{minipage}{2in}
\begin{prooftree}
\AxiomC{$\mathcal{D}_0$}
\noLine
\UnaryInfC{$(\varphi\vee\psi)$}

\AxiomC{$[\varphi]$}
\noLine
\UnaryInfC{$\mathcal{D}_1$}
\noLine\
\UnaryInfC{$\chi $}  

\AxiomC{$[\psi]$}
\noLine
\UnaryInfC{$\mathcal{D}_2$}
\noLine
\UnaryInfC{$\chi $}                
\RightLabel{$\vee \mathrm{E}$}
\TrinaryInfC{$\chi$}
\end{prooftree}\end{minipage}\begin{minipage}{2in}\begin{prooftree}
\AxiomC{$\mathcal{D}_0$}
\noLine
\UnaryInfC{$\psi$}       

\AxiomC{$[\varphi]$}  
\noLine
\UnaryInfC{$\mathcal{D}_1$}
\noLine
\UnaryInfC{$\neg\psi$}
\RightLabel{$\neg \mathrm{I}$}
\BinaryInfC{$\neg\varphi$}
\end{prooftree}
\end{minipage}
\end{center}
\caption{To modify Gentzen-style natural deduction rules to match our dropping of Reiteration from Fitch-style natural deduction, for $\vee$E the only open assumptions of $\mathcal{D}_1$ and $\mathcal{D}_2$  may be $\varphi$ and $\psi$, respectively; for $\neg$I the only open assumption of $\mathcal{D}_1$ may be $\varphi$.}\label{GentzenFig}
\end{figure}

In response to a presentation of this paper at the Colloquium Logicum 2022, Aguilera and Byd\u{z}ovsk\'y \citeyearpar{Aguilera2022} observed that a sequent calculus LF for fundamental logic can be obtained from Gentzen's sequent calculus LK for classical logic in the $\{\wedge,\vee,\neg\}$-signature by restricting to sequents $\Gamma\Rightarrow \Delta$ in which $|\Delta|\leq 1$, as for intuitionistic logic, and $|\Gamma| + |\Delta|\leq 2$, as for orthologic (\citealt{Monting1981}). They verified that $\varphi\Rightarrow_\mathrm{LF}\psi$ iff $\varphi\vdash_\mathsf{F}\psi$, that LF admits cut-elimination, and that proof search in the cut-free calculus terminates in polynomial time, following a proof-search strategy of Egly and Tompits \citeyearpar[\S~4.3]{Egly2003} for orthologic.

\begin{theorem}[Aguilera and Byd\u{z}ovsk\'y] It is decidable in polynomial time whether $\varphi\vdash_\mathsf{F}\psi$.
\end{theorem}

\noindent In fact, Aguilera and Byd\u{z}ovsk\'y obtained cut-elimination and decidability for the first-order version of fundamental logic in \S~\ref{QuantSection}.

\section{Algebras}\label{AlgSection}

We now turn to algebraic semantics for the logic presented in \S~\ref{FitchSection}. The relevant algebraic structures are bounded lattices equipped with an appropriate negation. We denote the lattice operations by $\wedge$ and $\vee$ and the negation operation by $\neg$, trusting that no confusion will arise by using the same symbols as in $\mathcal{L}$. 

We first define the operations corresponding to negation in intuitionistic logic, orthologic, and $\vdash_\mathsf{F}$, namely  pseudocomplementation, orthocomplementation, and weak pseudocomplementation, respectively.

\begin{definition}\label{NegDefs} Let $L$ be a bounded lattice and $a\in L$.  An $x\in L$ is the \textit{pseudocomplement} of $a$ if $x$ is the maximum in $L$ of $\{y\in L\mid a\wedge y=0\}$, a \textit{complement} of $a$ if $a\wedge x=0$ and $a\vee x=1$, and a \textit{semicomplement} of $a$ if $a\wedge x=0$.

A \textit{pseudocomplementation} (resp.~\textit{complementation}, \textit{semicomplementation}) is a unary operation $\neg$ on $L$  such that for all $a\in L$, $\neg a$ is the pseudocomplement (resp.~a complement, semicomplement) of $a$. 

An \textit{orthocomplementation} is a complementation that is antitone ($a\leq b$ implies $\neg b\leq\neg a$) and involutive ($\neg\neg a =a $). An \textit{ortholattice} is a bounded lattice equipped with an orthocomplementation.

 A \textit{weak pseudocomplementation} is an antitone semicomplementation satisfying \textit{double negation introduction}: $a\leq\neg\neg a$ for all $a\in L$. \end{definition}
 
The negation operation in a Heyting algebra, defined by $\neg a=a\to 0$, is the pseudocomplementation. Note that if a lattice admits a pseudocomplementation, then it is unique, in contrast to the other kinds of negations above.  The term `weak pseudocomplementation' is taken from \citealt{Dzik2006,Dzik2006b}, \citealt{Almeida2009}.\footnote{Weak psuedocomplementations are also called `Heyting negations' and `Heyting complementations' in \citealt{Dzik2006,Dzik2006b} and \citealt[p.~91]{Dunn2001}, respectively, but this clashes with the fact that the negation in a Heyting algebra is pseudocomplementation.}

The relational semantics of \S~\ref{RelationalSection} will handle other kinds of negations besides those for intuitionistic logic, orthologic, and $\vdash_\mathsf{F}$, so we define some weaker kinds below. For surveys of the large literature on different types of negation, we refer the reader to \citealt{Horn2020} and \citealt[Ch.~8]{Humberstone2011}.

\begin{definition}  A \textit{precomplementation} on a bounded lattice is an antitone unary operation $\neg$  such that $\neg 1=0$.  A \textit{protocomplementation} is an antitone semicomplementation $\neg$ such that $\neg 0=1$. An \textit{ultraweak pseudocomplementation} is an antitone unary operation $\neg$ satisfying double negation introduction and $\neg 1=0$.
\end{definition}

The term `protocomplementation' is from \citealt{Holliday2022}. An ``ultraweak'' pseudocomplementation drops $a\wedge \neg a=0$ from the definition of weak pseudocomplementation in the spirit of paraconsistent logics (\citealt{Priest2022}).\footnote{\label{Quasi-minimal}Ultraweak pseudocomplementations are equivalent to what Dunn and Zhou \citeyearpar{Dunn2005} call \textit{quasi-minimal negations} with the added assumption that $\neg 1=0$ (see Remark \ref{NCremark}).} An example of an ultraweak but not weak pseudocomplementation is the negation operation on the three-element chain with $\neg 1=0$, $\neg 0=1$, and $\neg \frac{1}{2} = \frac{1}{2}$ used for Kleene's \citeyearpar{Kleene1938} three-valued logic.

Properties of and the logical relations between six types of negation are shown in Figures \ref{NegTable} and \ref{NegFig}. For example, to see that any weak pseudocomplementation is a protocomplementation, we show that $1\leq \neg 0$:  given that $1\leq \neg\neg 1$, it suffices to show $\neg 1 =0$; indeed, $\neg 1=1\wedge \neg 1=0$ for any semicomplementation $\neg$. A number of other types of negation could be added to the diagram in Figure \ref{NegFig} (cf.~the ``kites of negations'' in \citealt{Dunn2005}). Each may appear to be based on a rather arbitrary choice of some properties but not others; but what makes weak pseudocomplementations stand out in our view is the connection with the introduction and elimination rules of $\vdash_\mathsf{F}$ established below.

\begin{figure}[h]
\begin{center}
\begin{tabular}{c|c|c|c|c|c|c|}
 		& pre & proto & ultraweak pseudo & weak pseudo & pseudo & ortho \\
		\hline
$a\leq b\Rightarrow\neg b\leq\neg a$ & \checkmark &  \checkmark & \checkmark & \checkmark& \checkmark& \checkmark\\
 $\neg 1=0$ & \checkmark &  \checkmark& \checkmark & \checkmark& \checkmark& \checkmark\\
  $\neg 0 =1$ &  &  \checkmark& \checkmark& \checkmark& \checkmark& \checkmark\\
 $a\wedge\neg a=0$ & &  \checkmark && \checkmark& \checkmark& \checkmark\\
 $a\leq\neg\neg a$ &&&\checkmark & \checkmark & \checkmark & \checkmark \\
  $a\wedge b=0\Rightarrow b \leq\neg a$ &&&&  & \checkmark &  \\
 $\neg\neg a\leq a$ &&&&&& \checkmark 
\end{tabular}
\end{center}
\caption{Properties of six types of negation.}\label{NegTable}
\end{figure}

\begin{figure}[h]

\begin{center}
\begin{tikzpicture}[->,>=stealth',shorten >=1pt,shorten <=1pt, auto,node
distance=2cm,thick,every loop/.style={<-,shorten <=1pt}]

\node at (-3,0) (p) {pseudocomplementation};

\node at (3,0) (o) {orthocomplementation};

\node at (0,-1.5) (wp) {weak pseudocomplementation};

\node at (-3,-3) (pr) {protocomplementation};

\node at (3,-3) (mwp) {ultraweak pseudocomplementation};

\node at (0,-4.5) (prpr) {precomplementation};

\path (p) edge[->] node {{}} (wp);
\path (o) edge[->] node {{}} (wp);
\path (wp) edge[->] node {{}} (pr);
\path (wp) edge[->] node {{}} (mwp);
\path (pr) edge[->] node {{}} (prpr);
\path (mwp) edge[->] node {{}} (prpr);

\end{tikzpicture}
\end{center}
\caption{Logical relations between six types of negation.}\label{NegFig}
\end{figure}
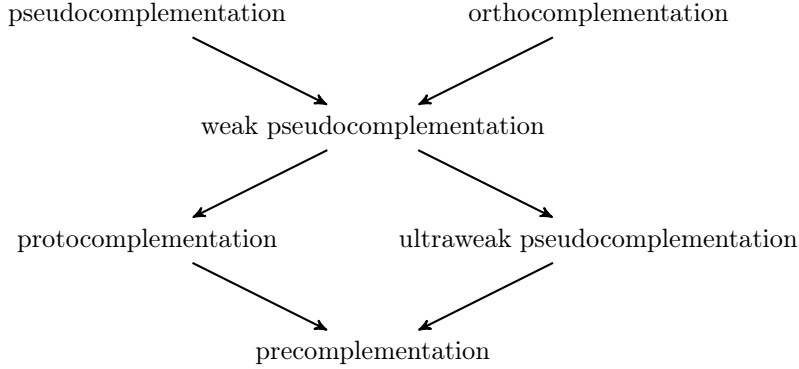

\begin{remark}\label{NCremark} The weakest notion of negation defined above is that of a precomplementation. Yet restricting to precomplementations already forecloses some types of negation studied in the literature. For example, negation in Johansson's \citeyearpar{Johansson1937} \textit{minimal logic} (cf.~\citealt{Kolmogorov1925}) is antitone and satisfies double negation introduction and the principle of non-contradiction in the form  $\neg (a\wedge\neg a)=1$ but not the semicomplementation axiom $a\wedge\neg a=0$; yet any ultraweak pseudocomplementation satisfying non-contradiction is a semicomplementation (and hence a weak pseudocomplementation). To give semantics for negation in minimal logic, we must drop the axiom $\neg 1=0$ of precomplementations. The same applies to the basic logic of Battilotti and Sambin \citeyearpar{Battilotti1999}, whose negation (which is quasi-minimal in the terminology of \citealt{Dunn2005}) satisfies none of $a\wedge\neg a=0$, $\neg (a\wedge\neg a)=1$, or $\neg 1=0$.  Although it is not our focus in this paper, we will explain how to handle negations that do not satisfy $\neg 1=0$ in Remark \ref{MoreGeneralNeg} below.\end{remark}

For later use we note the following facts.

\begin{lemma}\label{UsefulLem} Let $\neg $ be a unary operation on a bounded lattice $L$.
\begin{enumerate}
\item\label{UsefulLem1} If $\neg$ is a semicomplementation, then $\neg$ is \textit{anti-inflationary}: $a\not\leq \neg a$ for all nonzero $a\in L$. If $\neg$ is antitone and anti-inflationary, then $\neg$ is a semicomplementation. 
\item\label{UsefulLem2} $\neg$ satisfies antitonicity and double negation introduction iff for all $a,b\in L$,  $a\leq \neg b$ implies $b\leq\neg a$.
\item\label{UsefulLem3} $\neg$ is an orthocomplementation iff $\neg$ is a weak pseudocomplementation satisfying \textit{double negation elimination}: $\neg\neg a\leq a$ for all $a\in L$.
\end{enumerate}
\end{lemma}

\begin{proof} For part \ref{UsefulLem1}, if for some nonzero $a\in L$, $a\leq \neg a$, then $a\wedge \neg a=a\neq 0$, so $\neg$ is not a semicomplementation. Now suppose $\neg$ is antitone and anti-inflationary. If $a\wedge\neg a\neq 0$, then by anti-inflationarity, $a\wedge\neg a\not\leq \neg (a\wedge\neg a)$, but since $a\wedge\neg a\leq a$, we have $\neg a\leq \neg (a\wedge\neg a)$ by antitonicity and hence $a\wedge\neg a\leq \neg (a\wedge\neg a)$.

For part \ref{UsefulLem2}, if $\neg$ satisfies antitonicity and double negation introduction, them $a\leq \neg b$ implies $\neg \neg b\leq\neg a$ and hence $b\leq\neg a$. Conversely, suppose $\neg$ satisfies the implication in part \ref{UsefulLem2}. Then starting with $\neg b\leq \neg b$ and $a=\neg b$, we have $b\leq \neg\neg b$. For antitonicity, if $a\leq c$, then $a\leq \neg\neg c$, so taking $b=\neg c$, we have  $\neg c\leq \neg a$.

For part \ref{UsefulLem3}, we need only show $1\leq a\vee\neg a$ when $\neg$ is a weak pseudocomplementation satisfying double negation elimination. Since $a\leq a\vee\neg a$ and $\neg a \leq a\vee\neg a$, we have $\neg(a\vee\neg a)\leq \neg a\wedge\neg\neg a$ and hence $\neg (a\vee\neg a)\leq \neg a\wedge a$, so $\neg (\neg a\wedge a)\leq \neg\neg (a\vee\neg a)\leq a\vee\neg a$. Then since a weak pseudocomplementation satisfies $\neg a\wedge a=0$ and $\neg 0=1$, we have $1\leq a\vee\neg a$.
\end{proof}

Figure \ref{N5Fig} shows the $\mathbf{N}_5$ lattice equipped with a pseudocomplementation that is not an orthocomplementation (left), a weak pseudocomplementation that is neither a pseudocomplementation nor an orthocomplementation (middle), and a protocomplementation that is not a weak pseudocomplementation (right). Figure \ref{O6Fig} shows the Benzene ring $\mathbf{O}_6$ equipped with an orthocomplementation that is not a pseudocomplementation (left) and a pseudocomplementation that is not an orthocomplementation (right). 

Note that any bounded lattice can be equipped with a weak pseudocomplementation by setting $\neg 0=1$ and $\neg a=0$ for all $a\neq 0$; and if there are nonzero $a,b\in L$ with $a\wedge b=0$, this $\neg$ is not a pseudocomplementation. Also note that any bounded lattice can be equipped with a precomplementation by setting $\neg 1=0$ and $\neg a=1$ for all $a\neq 1$; and if $L$ has more than one nonzero element, this $\neg$ is not a protocomplementation.

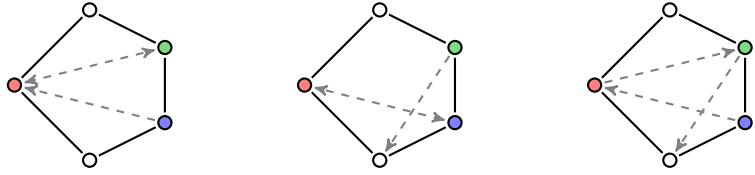
\begin{figure}[ht]
\vspace{.2in}
\begin{center}
\begin{tikzpicture}[->,>=stealth',shorten >=1pt,shorten <=1pt, auto,node
distance=2cm,thick,every loop/.style={<-,shorten <=1pt}]
\tikzstyle{every state}=[fill=gray!20,draw=none,text=black]

\node[circle,draw=black!100, label=right:$$,inner sep=0pt,minimum size=.175cm] (n1) at (0,0) {{}};
\node[circle,fill=darkgreen!50,draw=black!100, label=left:$$,inner sep=0pt,minimum size=.175cm] (nx) at (1,-.5) {{}};
\node[circle,fill=blue!50,draw=black!100, label=right:$$,inner sep=0pt,minimum size=.175cm] (ny) at (1,-1.5) {{}};
\node[circle,fill=red!50,draw=black!100, label=right:$$,inner sep=0pt,minimum size=.175cm] (nz) at (-1,-1) {{}};
\node[circle,draw=black!100, label=right:$$,inner sep=0pt,minimum size=.175cm] (n0) at (0,-2) {{}};
\path (nx) edge[-] node {{}} (n1);
\path (nx) edge[-] node {{}} (ny);
\path (n1) edge[-] node {{}} (nz);
\path (ny) edge[-] node {{}} (n0);
\path (nz) edge[-] node {{}} (n0);

\path (ny) edge[->,dashed,gray] node {{}} (nz);

\path (nz) edge[<->,dashed,gray] node {{}} (nx);

\end{tikzpicture}\qquad\qquad \begin{tikzpicture}[->,>=stealth',shorten >=1pt,shorten <=1pt, auto,node
distance=2cm,thick,every loop/.style={<-,shorten <=1pt}]
\tikzstyle{every state}=[fill=gray!20,draw=none,text=black]

\node[circle,draw=black!100, label=right:$$,inner sep=0pt,minimum size=.175cm] (n1) at (0,0) {{}};
\node[circle,fill=darkgreen!50,draw=black!100, label=left:$$,inner sep=0pt,minimum size=.175cm] (nx) at (1,-.5) {{}};
\node[circle,fill=blue!50,draw=black!100, label=right:$$,inner sep=0pt,minimum size=.175cm] (ny) at (1,-1.5) {{}};
\node[circle,fill=red!50,draw=black!100, label=right:$$,inner sep=0pt,minimum size=.175cm] (nz) at (-1,-1) {{}};
\node[circle,draw=black!100, label=right:$$,inner sep=0pt,minimum size=.175cm] (n0) at (0,-2) {{}};
\path (nx) edge[-] node {{}} (n1);
\path (nx) edge[-] node {{}} (ny);
\path (n1) edge[-] node {{}} (nz);
\path (ny) edge[-] node {{}} (n0);
\path (nz) edge[-] node {{}} (n0);

\path (ny) edge[<->,dashed,gray] node {{}} (nz);

\path (nx) edge[->,dashed,gray] node {{}} (n0);
x);

\end{tikzpicture}\qquad\qquad \begin{tikzpicture}[->,>=stealth',shorten >=1pt,shorten <=1pt, auto,node
distance=2cm,thick,every loop/.style={<-,shorten <=1pt}]
\tikzstyle{every state}=[fill=gray!20,draw=none,text=black]

\node[circle,draw=black!100, label=right:$$,inner sep=0pt,minimum size=.175cm] (n1) at (0,0) {{}};
\node[circle,fill=darkgreen!50,draw=black!100, label=left:$$,inner sep=0pt,minimum size=.175cm] (nx) at (1,-.5) {{}};
\node[circle,fill=blue!50,draw=black!100, label=right:$$,inner sep=0pt,minimum size=.175cm] (ny) at (1,-1.5) {{}};
\node[circle,fill=red!50,draw=black!100, label=right:$$,inner sep=0pt,minimum size=.175cm] (nz) at (-1,-1) {{}};
\node[circle,draw=black!100, label=right:$$,inner sep=0pt,minimum size=.175cm] (n0) at (0,-2) {{}};
\path (nx) edge[-] node {{}} (n1);
\path (nx) edge[-] node {{}} (ny);
\path (n1) edge[-] node {{}} (nz);
\path (ny) edge[-] node {{}} (n0);
\path (nz) edge[-] node {{}} (n0);

\path (ny) edge[->,dashed,gray] node {{}} (nz);

\path (nx) edge[->,dashed,gray] node {{}} (n0);

\path (nz) edge[->,dashed,gray] node {{}} (nx);

\end{tikzpicture}

\end{center}
\caption{$\mathbf{N}_5$ equipped with a pseudocomplementation (left), a weak pseudocomplementation (middle), and a protocomplementation (right), indicated by dashed arrows. Arrows for $\neg 0=1$ and $\neg 1=0$ are omitted.}\label{N5Fig}
\end{figure}

\begin{figure}[!ht]
\begin{center}
\begin{tikzpicture}[->,>=stealth',shorten >=1pt,shorten <=1pt, auto,node
distance=2cm,thick,every loop/.style={<-,shorten <=1pt}]
\tikzstyle{every state}=[fill=gray!20,draw=none,text=black]
\node[circle,draw=black!100, label=below:,inner sep=0pt,minimum size=.175cm] (0) at (0,0) {{}};
\node[circle,draw=black!100,fill=red!50, label=left:,inner sep=0pt,minimum size=.175cm] (a) at (-.75,.75) {{}};
\node[circle,draw=black!100,fill=orange!50, label=right:,inner sep=0pt,minimum size=.175cm] (b) at (.75,.75) {{}};
\node[circle,draw=black!100,fill=darkgreen!50, label=left:,inner sep=0pt,minimum size=.175cm] (1l) at (-.75,1.5) {{}};
\node[circle,draw=black!100,fill=blue!50, label=right:,inner sep=0pt,minimum size=.175cm] (1r) at (.75,1.5) {{}};
\node[circle,draw=black!100, label=above:,inner sep=0pt,minimum size=.175cm] (new1) at (0,2.25) {{}};

\path (new1) edge[-] node {{}} (1l);
\path (new1) edge[-] node {{}} (1r);
\path (1l) edge[-] node {{}} (a);
\path (1r) edge[-] node {{}} (b);
\path (a) edge[-] node {{}} (0);
\path (b) edge[-] node {{}} (0);

\path (1l) edge[<->,dashed,gray] node {{}} (b);
\path (1r) edge[<->,dashed,gray] node {{}} (a);

\end{tikzpicture}\qquad\qquad \begin{tikzpicture}[->,>=stealth',shorten >=1pt,shorten <=1pt, auto,node
distance=2cm,thick,every loop/.style={<-,shorten <=1pt}]
\tikzstyle{every state}=[fill=gray!20,draw=none,text=black]
\node[circle,draw=black!100, label=below:,inner sep=0pt,minimum size=.175cm] (0) at (0,0) {{}};
\node[circle,draw=black!100,fill=red!50, label=left:,inner sep=0pt,minimum size=.175cm] (a) at (-.75,.75) {{}};
\node[circle,draw=black!100,fill=orange!50, label=right:,inner sep=0pt,minimum size=.175cm] (b) at (.75,.75) {{}};
\node[circle,draw=black!100,fill=darkgreen!50, label=left:,inner sep=0pt,minimum size=.175cm] (1l) at (-.75,1.5) {{}};
\node[circle,draw=black!100,fill=blue!50, label=right:,inner sep=0pt,minimum size=.175cm] (1r) at (.75,1.5) {{}};
\node[circle,draw=black!100, label=above:,inner sep=0pt,minimum size=.175cm] (new1) at (0,2.25) {{}};

\path (new1) edge[-] node {{}} (1l);
\path (new1) edge[-] node {{}} (1r);
\path (1l) edge[-] node {{}} (a);
\path (1r) edge[-] node {{}} (b);
\path (a) edge[-] node {{}} (0);
\path (b) edge[-] node {{}} (0);

\path (1l) edge[<-,dashed,gray] node {{}} (b);
\path (1l) edge[<->,dashed,gray] node {{}} (1r);
\path (1r) edge[<-,dashed,gray] node {{}} (a);

\end{tikzpicture}
\end{center}
\caption{The Benzene ring $\mathbf{O}_6$ equipped with an orthocomplementation (left) and pseudocomplementation (right), indicated by dashed arrows. Arrows for $\neg 0=1$ and $\neg 1=0$ are omitted.}\label{O6Fig}
\end{figure}
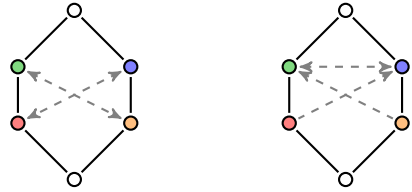

It is noteworthy that all of the intuitionistically acceptable De Morgan inequalities that hold in bounded lattices with pseudocomplementations also hold in bounded lattices with weak pseudocomplementations: $\neg (a\vee b)=\neg a\wedge\neg b$ and $\neg a\vee\neg b\leq \neg (a\wedge b)$. However, there are inequalities that hold in all bounded lattices with pseudocomplementations and all bounded lattices with orthcomplementations but do not hold in all bounded lattices with weak pseudocomplementations. An example is 
\[\neg\neg a \wedge \neg\neg b\leq \neg\neg (a\wedge b).\]
Consider the 4-element Boolean lattice equipped not with Boolean negation but with the weak pseudocomplementation with $\neg 0=1$ and $\neg c=0$ for $c\neq 0$. Where $a$ and $b$ are the side elements of the lattice, we have  $\neg\neg a\wedge\neg\neg b= 1\wedge 1=1$ while $\neg\neg (a\wedge b)=\neg\neg 0=0$.\footnote{This example shows that while in lattices with weak pseudocomplementation, double negation is a closure operator, it is not multiplicative and hence not a \textit{nucleus}, as it is in pseudocomplemented lattices (cf.~\citealt[\S~3]{BH2019}).} This suggests an interesting problem, not pursued here, of axiomatizing the intersection of orthologic and intuitionistic~logic (or  \textit{orthointuitionistic logic}).

As usual, we can interpret the language $\mathcal{L}$ in lattice expansions $(L,\neg)$ as follows.

\begin{definition} A \textit{valuation} on a lattice expansion $(L,\neg)$ is a function $\theta:\mathsf{Prop}\longrightarrow L$ that extends to $\tilde{\theta}:\mathcal{L}\longrightarrow L$ by: $\tilde{\theta}(p)=\theta(p)$, $\tilde{\theta}(\neg\varphi)=\neg\tilde{\theta}(\varphi)$, $\tilde{\theta}(\varphi\wedge\psi)=\tilde{\theta}(\varphi)\wedge \tilde{\theta}(\psi)$, and $\tilde{\theta}(\varphi\vee\psi)=\tilde{\theta}(\varphi)\vee \tilde{\theta}(\psi)$. 

Given a class $\mathscr{C}$ of lattice expansions, we define $\varphi\vDash_\mathscr{C}\psi$ if for every $(L,\neg)\in \mathscr{C}$ and valuation $\theta$ on $(L,\neg)$, we have $\tilde{\theta}(\varphi)\leq\tilde{\theta}(\psi)$.
\end{definition}

Let $\mathscr{W}$ be the class of lattices expanded with a weak pseudocomplementation. Then we have the following soundness result for our Fitch-style proof system.

\begin{proposition}\label{SoundProp} For any $\varphi,\psi\in\mathcal{L}$, if $\varphi\vdash_\mathsf{F}\psi$, then $\varphi\vDash_\mathscr{W}\psi$.
\end{proposition}
\begin{proof} We claim that for any Fitch-style proof $\langle \sigma_1,\dots,\sigma_n\rangle$, if $\sigma_n$ is a formula, then $\sigma_1\vDash_\mathscr{W}\sigma_n$. We proceed by induction on proofs, using the fact that if $\langle \sigma_1,\dots,\sigma_k\rangle$ is a proof, so is $\langle \sigma_1,\dots,\sigma_\ell\rangle$ for $1\leq \ell\leq k$.  Suppose, for example, that $\langle \sigma_1,\dots,\sigma_{n+1}\rangle$ is a proof in which $\sigma_{n+1}=\neg\varphi$ is obtained by the $\neg$I rule: that is,  $\langle \sigma_1,\dots,\sigma_n\rangle$ is a  proof, there is a formula $\sigma_i$ of the form $\psi$, and $\sigma_n$ is a proof beginning with $\varphi$ and ending with $\neg\psi$. Then by the inductive hypothesis applied to the proof $\langle\sigma_1,\dots,\sigma_i\rangle$, we have $\sigma_1\vDash_\mathscr{W}\psi$; and by the inductive hypothesis applied to the proof  $\sigma_n$, we have $\varphi\vDash_\mathscr{W}\neg\psi$, which implies $\psi\vDash_\mathscr{W}\neg\varphi$  by Lemma \ref{UsefulLem}.\ref{UsefulLem2}. Putting the previous two steps together, we have $\sigma_1\vDash_\mathscr{W}\neg\varphi$. The other cases of the proof are similar.\end{proof}

As usual, the Lindenbaum-Tarski algebra of $\vdash_\mathsf{F}$ has as its elements the equivalence classes $[\varphi]$ of formulas of $\mathcal{L}$, where $\varphi$ and $\psi$ are equivalent if $\varphi\vdash_\mathsf{F}\psi$ and $\psi\vdash_\mathsf{F}\varphi$, and the operations are defined by $\neg[\varphi]=[\neg\varphi]$, $[\varphi]\wedge[\psi]=[\varphi\wedge\psi]$, and $[\varphi]\vee[\psi]=[\varphi\vee\psi]$. It is easy to show using Proposition \ref{BinaryLogicProp} that this algebra is a bounded lattice equipped with a weak pseudocomplementation, $(L,\neg)$, whose lattice order we denote by $\leq$. Then the valuation $\theta:\mathsf{Prop}\longrightarrow L$ defined by $\theta(p)=[p]$ is such that for all $\varphi\in L$, $\tilde{\theta}(\varphi)=[\varphi]$. Hence if $\varphi\nvdash_\mathsf{F}\psi$, so $[\varphi]\not\leq [\psi]$, then $\tilde{\theta}(\varphi)\not\leq \tilde{\theta}(\psi)$, so $\varphi\nvDash_\mathscr{W}\psi$. This yields the following completeness result.

\begin{proposition} For any $\varphi,\psi\in\mathcal{L}$, if $\varphi\vDash_\mathscr{W}\psi$, then $\varphi\vdash_\mathsf{F}\psi$.
\end{proposition}

\noindent By similar reasoning, we can show the soundness and completeness with respect to $\mathscr{W}$ of the smallest intro-elim logic, so we obtain the following. 

\begin{proposition}\label{SmallestIntelim} $\vdash_\mathsf{F}$ is the smallest intro-elim logic.
\end{proposition}

Thus, $\vdash_\mathsf{F}$ is the logic of bounded lattices with weak pseudocomplementations. Figure \ref{NumbersOfObjects} shows the numbers of  algebras up to isomorphism of size up to $10$, calculated using Mace4 (\citealt{prover9-mace4}), for  $\vdash_\mathsf{F}$, intuitionistic logic (i.e., finite distributive lattices, each of which can be equipped with a unique pseudocomplementation), and orthologic. For comparison we also include the number of lattices and the number of pseudocomplemented lattices (i.e., lattices in which each element has a pseudocomplement).

\begin{figure}[ht]
\begin{center}
\begin{tabular}{lllllllllll}
 & $f(2)$ & $f(3)$ & $f(4)$ & $f(5)$ & $f(6)$ & $f(7)$ & $f(8)$ & $f(9)$ & $f(10)$ \\
 \hline
lattices with weak pseudocomp. & $1$ & $1$ & $3$ & $9$ & $38$ & $187$ & $1130$ & $7914 $ & $63,782$ \\
lattices & $1$ & $1$ & $2$ & $5$ & $15$ & $53$ & $222$ & $1078 $ & $5994$ \\
pseudocomplemented lattices & $1$ & $1$ & $2$ & $4$ & $10$ & $29$ & $99$ & $391 $ & $1357$ \\
distributive lattices & $1$ & $1$ & $2$ & $3$ & $5$ & $8$ & $15$ & $26 $ & $47$  \\
ortholattices & $1$ & $0$ & $1$ & $0$ & $2$ & $0$ & $5$ & $0 $ & $15$  \\
\end{tabular}
\end{center}
\caption{$f(n)$ is the number of algebras of size $n$ up to isomorphism in the given class.}\label{NumbersOfObjects}
\end{figure}

Finally, we note that the observation above that any bounded lattice can be equipped with a weak pseudocomplementation implies a conservativity fact about $\vdash_\mathsf{F}$: if $\varphi\vdash_\mathsf{F}\psi$ and $\varphi,\psi$ do not contain $\neg$, then $\psi$ is provable from $\varphi$ in the Fitch-style proof system for the $\{\wedge,\vee\}$-fragment of $\mathcal{L}$ defined as for $\vdash_\mathsf{F}$ but without the negation rules. That restricted proof system is easily shown to be sound and complete with respect to the class of all bounded lattices. Hence if $\psi$ is not provable from $\varphi$ in the restricted system, then there is a bounded lattice witnessing that $\psi$ is not a semantic consequence of $\varphi$, which we then expand to a bounded lattice with a weak pseudocomplementation witnessing that $\psi$ is not a semantic consequence of $\varphi$, so $\varphi\nvdash_\mathsf{F}\psi$.

\section{Relational representation and semantics}\label{RelationalSection}

In this section, we give a relational semantics for our logic via a relational representation of bounded lattices equipped with a weak pseudocomplementation. In \S\S~\ref{FrameToLatSection}-\ref{DiscreteRep}, we build on the discrete representation of  bounded lattices equipped with a protocomplementation from \citealt{Holliday2022}, extended and specialized for other kinds of negation from  \S~\ref{AlgSection} (and further extended to bounded lattices with implications in \S~\ref{Conditionals} and Appendix \ref{AppendixB}). In \S~\ref{TopRep}, we cover a topological representation of bounded lattices with negations. It would be natural to extend these representations to categorical dualities between categories of lattices with  negations and categories of relational frames, but we will not pursue such a project here. Finally, in \S~\ref{ModalTranslations}, we discuss translations of propositional logics into modal logics suggested by our relational semantics.

\subsection{From relational frames to lattices with negation}\label{FrameToLatSection}

In \citealt{Ploscica1995}, a representation of bounded lattices is developed using a set together with a reflexive binary relation and a topology. For now we ignore topology (until \S~\ref{TopRep}) and use relational frames for a discrete representation of complete lattices with negations as in \citealt{Holliday2022}. 

Relational representations of lattices with various negations have also been developed on the basis of Urquhart's \citeyearpar{Urquhart1978} doubly ordered sets in \citealt{Allwein1993b} and \citealt{Dzik2006,Dzik2006b} and on the basis of Birkhoff's \citeyearpar{Birkhoff1940} polarities in \citealt{Almeida2009}. Here we use a single relation on a single set  to realize both a lattice and its negation, in contrast to two relations to realize a lattice and a third to realize a negation (\citealt{Dzik2006,Dzik2006b}) or a relation between two sets to realize a lattice and a second relation to realize a negation (\citealt{Almeida2009}).  Using a single relation on a single set to realize a lattice and its negation goes back to Birkhoff and von Neumann (\citealt{Birkhoff1936}, \citealt[p.~25]{Birkhoff1940}), who applied this idea to ortholattices, leading to relational semantics for orthologic  (\citealt{Dishkant1972}, \citealt{Goldblatt1974}). Of course it also appears in relational semantics for intuitionistic logic (\citealt{Dummett1959}, \citealt{Grzegorczyk1964}, \citealt{Kripke1965}), which is a special case of the following approach (see Remark \ref{Downsets}), though using a single relation in this case is not surprising since the relevant negation is uniquely determined by the lattice.  

Inspired by the intuitionistic and orthological cases, Do\v{s}en \citeyearpar{Dosen1984,Dosen1986,Dosen1999}, Vakarelov \citeyearpar{Vakarelov1989}, and Dunn \citeyearpar{Dunn1993,Dunn1996,Dunn1999} (also see \citealt{Dunn2005}) study negation using triples $(X,\comp, \sqsubseteq)$ where $(X,\comp)$ is a relational frame as below, $\sqsubseteq$ is a partial order on $X$, and an interaction condition holds between $\comp$ and $\sqsubseteq$. Their definition of negation is the same as in \citealt{Birkhoff1940} for orthocomplementation, namely that $x\in\neg A$ iff for all $y\comp x$, $y\not\in A$ (or equivalently, for all $y\in A$, $ y\not \comp x$, and possibly writing $x\compflip y$ instead of $y\comp x$), which we will also use; the interaction condition between $\comp$ and $\sqsubseteq$ then ensures that the negation operation sends upsets (or downsets, depending on one's preference) to upsets (or downsets) of $\sqsubseteq$. Berto \citeyearpar{Berto2015} (also see \citealt{Berto2019}) uses their setup to argue that $\neg$ should satisfy at least antitonicity and $a\leq\neg\neg a$, a  congenial conclusion given our interest in $\vdash_\mathsf{F}$. However, the cited authors do not generate the underlying lattice of propositions using the closure operator $c_\comp$ as in Propositions  \ref{IsClosure}.\ref{IsClosure1} and \ref{FrameToLat}.\ref{FrameToLat1} below (Do\v{s}en and Vakarelov take the lattice of upsets/downsets, and Dunn sometimes takes the lattice of upsets/downsets and sometimes does not, e.g., when he wants to represent ortholattices), and their correspondences between conditions on $\comp$ and axioms for negation are not the same as in our setting (see Remark \ref{DunnDifference}). 

The single relation approach has recently been applied to a sublogic of orthologic and intuitionistic logic in  \citealt{Zhong2021}, which axiomatizes the logic of the reflexive frames below in the $\{\neg,\wedge\}$-fragment of $\mathcal{L}$ (see Theorem \ref{MainCompleteness}.\ref{MainCompleteness2} below for the axiomatization in the full language with $\vee$). Zhong \citeyearpar{Zhong2021} takes inspiration from  Dalla Chiara and Giuntini \citeyearpar[pp.~139-140]{Chiara2002}, who observe that there is a closure operator definable from a reflexive relation---the same closure operator used in \citealt{Ploscica1995}---whose fixpoints are propositions for orthologic if the relation is symmetric or  for intuitionistic logic if the relation is transitive.

Finally, the approach of representing a lattice using a binary relation on a set $X$ contrasts with the approach of representing a lattice using a binary relation between $X$ and $\wp(X)$, or equivalently, a function $N:X\to \wp(\wp(X))$, as in neighborhood semantics for modal logic (\citealt{Scott1970}, \citealt{Montague1970}, \citealt{Pacuit2017}). In the neighborhood approach, one imposes conditions on $N$ such that the operation $c:\wp(X)\longrightarrow\wp(X)$ defined by $c(A)=\{x\in X\mid A\in N(x)\}$ is a closure operator,\footnote{Another definition of $c$, building in monotonicity, is $c(A)=\{x\in X\mid \exists B\in N(x):B\subseteq A\}$.} whose fixpoints give us a complete lattice via Proposition \ref{ClosureLattice} below. Conversely, any complete lattice is representable as the lattice of fixpoints of a closure operator on a powerset (see, e.g., \citealt[Thm.~5.3]{Burris1981}), and any closure operator  $c$ on $\wp(X)$ is representable using a function $N$ as above, defined by $N(x)=\{A\subseteq X\mid x\in c(A)\}$. By contrast, in the approach with a binary relation on $X$, matching relational semantics for modal logic (see \S~\ref{ModalTranslations}) instead of neighborhood semantics, the representability of complete lattices is less immediate. Versions of the neighborhood approach have been used by van Fraassen \citeyearpar[\S~II]{Fraassen1986}, who applies it to Heyting algebras, ortholattices, and Boolean algebras, and Goldblatt \citeyearpar{Goldblatt2011}, who applies it to Heyting algebras. Dragalin \citeyearpar{Dragalin1979,Dragalin1988} also uses functions $N:X\to \wp(\wp(X))$ to represent Heyting algebras, but he defines his closure operator from $N$ in a kind of dual way (also see \citealt{BH2016}).

Our basic objects are simply the following frames.

\begin{definition} A \textit{relational frame} is a pair $(X,\comp)$ of a nonempty set $X$ and a binary relation $\comp$ on $X$. We say the frame is \textit{reflexive} if $\comp$ is reflexive.\end{definition}
\noindent We call elements of $X$ \textit{states} and read $x\comp y$ as \textit{$x$ is open to $y$} in the sense of the following remark.\footnote{In previous work (\citealt{Holliday2021,Holliday2022}), I read $x\comp y$ as \textit{$x$ is compatible with $y$}, but many readers have the intuition that ``compatibility'' is necessarily symmetric.} When convenient, we write $y\compflip x$ for $x\comp y$.

\begin{remark}\label{IntuitivePicture} For an intuitive picture to pair with the mathematical development to follow, start with the distinction between \textit{accepting} a proposition and \textit{rejecting} it. We want to allow for \textit{partial states} that are completely noncommittal about a proposition, so non-acceptance of a proposition should not entail rejection of it. Moreover, we want to allow for states that reject a proposition without accepting the negation of it; for example, an intuitionist might \textit{reject} a certain instance of the law of excluded middle, $A\vee\neg A$, but will certainly not accept its negation, which is an intuitionistic contradiction (cf.~Field's \citeyearpar{Field2003} separation of rejection, non-acceptance, and acceptance of the negation). These notions can be linked with our notion of \textit{openness} as follows: $x$ is open to $y$ iff $x$ does not reject any proposition that $y$ accepts. If this is consistent with $y$ rejecting some proposition that $x$ accepts, then openness in our sense is not necessarily symmetric. Now if we start with $(X,\comp)$ and a proposition $A\subseteq X$, say that $x$ accepts $A$ if $x\in A$; $x$ rejects $A$ if for all $y\compflip x$, $y\not\in A$; and $x$ accepts $\neg A$ if for all $y\comp x$, $y\not\in A$.\footnote{It follows that accepting $A$ entails rejecting $\neg A$. The ideas that accepting $A$ is inconsistent with rejecting $A$ and that accepting $\neg A$ entails rejecting $A$ will follow from the key conditions on frames for fundamental logic.} Then we will indeed have that $x\comp y$ iff  $x$ does not reject any proposition that $y$ accepts.\footnote{If $x\comp y$ and $y$ accepts $A$, so $y\in A$, then $x$ does not reject $A$ by definition. Conversely, if $x\not\comp y$, then using Proposition \ref{IsClosure} below, $y$ accepts the proposition $c_\comp(\{y\})$ but $x$ rejects it given $x\not\comp y$.}  Finally, another result of the partiality of states is that accepting a disjunction does not require accepting either disjunct. Instead, $x$ accepting $A\vee B$ will amount to the following: no state open to $x$ rejects both disjuncts.\end{remark}

Rather than moving from a relational frame to an associated Boolean algebra with an operator, as in modal logic, here we move   to an associated lattice equipped with a negation. See \citealt{Holliday2021} for comparison with the realization of  complete lattices using doubly ordered structures and polarities. 

First recall that a unary operation on a lattice  is a \textit{closure operator} if $c$ is inflationary ($a\leq c(a)$), idempotent ($c(c(a))=c(a)$), and monotone ($a\leq b$ implies $c(a)\leq c(b)$). We will use the relation $\comp$ to define a closure operator on $\wp(X)$, whose fixpoints give us a complete lattice as in the following classic result (see, e.g., \citealt[Thm.~5.2]{Burris1981}).

\begin{proposition}\label{ClosureLattice} Let $X$ be a nonempty set and $c$ a closure operator on $\wp(X)$. Then the fixpoints of $c$, i.e., those $A\subseteq X$ with $c(A)=A$, ordered by $\subseteq$ form a complete lattice with
\[\underset{i\in I}{\bigwedge}{A_i} =\underset{i\in I}{\bigcap}{A_i}\mbox{ and } \underset{i\in I}{\bigvee}{A_i} =c(\underset{i\in I}{\bigcup}{A_i}).\]
\end{proposition}

In our case, the relevant closure operator is given in part \ref{IsClosure1} of the following, while the relevant negation operation on the fixpoints of the closure operator is given in part \ref{IsClosure2}. The proof is straighforward.

\begin{proposition}\label{IsClosure} For any relational frame $(X,\comp)$:
\begin{enumerate}
\item\label{IsClosure1} the operation $c_\comp: \wp(X)\longrightarrow\wp(X)$ defined by \[c_\comp(A)=\{x\in X\mid \forall x'\comp x\; \exists x''\compflip x':\, x''\in A\}\]
is a closure operator on $\wp(X)$; 
\item\label{IsClosure2} the operation $\neg_\comp: \wp(X)\longrightarrow\wp(X)$ defined by
\[\neg_\comp A = \{x\in X\mid \forall y\comp x\;\; y\not\in A \}\]
sends $c_\comp$-fixpoints to $c_\comp$-fixpoints.
\end{enumerate}
\end{proposition}
\noindent Thus, $x$ is in the closure of $A$ iff \textit{every state open to $x$ is open to some state in $A$};\footnote{\label{MorphismsNote}Given this definition of the closure operation, a candidate definition of morphism between $(X,\comp)$ and $(X',\comp')$ is a map $f:X\longrightarrow X'$ such that (i) $y\comp x$ implies $f(y)\comp' f(x)$, and (ii) if $y'\comp' f(x)$, then $\exists y\comp x$ $\forall z\compflip y$ $f(z)\compflip' y'$. Condition (ii) guarantees that if $A'$ is a fixpoint of $c_{\comp'}$, then $f^{-1}[A']$ is a fixpoint of $c_{\comp}$. For suppose $x'\not\in f^{-1}[A']$, so $f(x')\not\in A'$. Then since $A'$ is a fixpoint of $c_{\comp'}$, there is a $y'\comp' f(x')$ such that for all $z'\compflip' y'$, we have $z'\not\in A'$. By (ii), $\exists y\comp x$ $\forall z\compflip y$ $f(z)\compflip' y'$, which by the previous sentence implies $\exists y\comp x$ $\forall z\compflip y$ $f(z)\not\in A'$ and hence $z\not\in f^{-1}[A']$. This shows that $f^{-1}[A']$ is a fixpoint of $c_\comp$. If we want morphisms that also preserve negation, then $f^{-1}[\neg_{\comp'}A']\subseteq \neg_\comp f^{-1}[A'] $ follows from (i), and  $\neg_\comp f^{-1}[A']\subseteq f^{-1}[\neg_{\comp'}A']$ follows from the additional condition (iii) that if $y'\comp 'f(x)$, then $\exists y\comp x$ $\forall z'\comp' f(y)$ $z'\comp' y'$. For if $x\not\in f^{-1}[\neg_{\comp'}A']$, so $f(x)\not\in \neg_{\comp'}A'$, then there is a $y'\comp 'f(x)$ with $y'\in A'$. Then we claim for the $y\comp x$ given by (iii) that $f(y)\in A'$; for by (iii), $f(y)\in c_{\comp'}(\{y'\})$, and since $y'\in A'$, we have $c_{\comp'}(\{y'\})\subseteq c_{\comp'}(A')=A'$. Hence $x\not\in \neg_\comp f^{-1}[A']$.}  and $x$ is in the negation of $A$ iff \textit{no state open to $x$ is in $A$}. We call the fixpoints of the $c_\comp$ operation, those $A$ such that $c_\comp(A)=A$, the \textit{$c_\comp$-fixpoints}, rather than closed sets, since later (\S~\ref{TopRep}) we will add a topology in which the $c_\comp$-fixpoints are open but not necessarily closed, so our terminology avoids any possible confusion. We will assume that \textit{propositions} are $c_\comp$-fixpoints, which amounts to the following in the terms of Remark \ref{IntuitivePicture}: $A$ is a proposition ($c_\comp$-fixpoint) iff whenever a state $x$ does not accept $A$, then there is a state open to $x$ that rejects $A$.

In Section \ref{Conditionals} and Appendix \ref{AppendixB}, we also define binary implication operations from the $\comp$ relation, and from these implication operations, both $c_\comp$ and $\neg_\comp$ are in turn definable.

Proposition \ref{ClosureLattice} together with Proposition \ref{IsClosure}.\ref{IsClosure1} yields part \ref{FrameToLat1} of the following, while Proposition \ref{IsClosure}.\ref{IsClosure2} together with some easy additional reasoning yields parts \ref{FrameToLat2} and \ref{FrameToLat3}.

\begin{proposition}\label{FrameToLat} For any relational frame $(X,\comp)$:
\begin{enumerate}
\item\label{FrameToLat1} the $c_\comp$-fixpoints  ordered by $\subseteq$ form a complete lattice $\lat(X,\comp)$ with meet and join calculated as in Proposition \ref{ClosureLattice};
\item\label{FrameToLat2} $\neg_\comp$ is a precomplementation on $\lat(X,\comp)$;
\item\label{FrameToLat3} if $\comp$ is reflexive, then $\neg_\comp$ is a protocomplementation on $\lat(X,\comp)$.
\end{enumerate}
\end{proposition}

\noindent One subtlety to note is that the $0$ of the lattice $\lat(X,\comp)$ is $c_\comp(\varnothing)$, which is equal to $\varnothing$ in reflexive frames but not in arbitrary relational frames, where the situation with $0$ is as follows.

\begin{definition}\label{AbsurdDef} For a relational frame $(X,\comp)$ and $x\in X$,  $x$ is \textit{absurd} if there is no $y$ with $y\comp x$.
\end{definition}

\begin{lemma}\label{AbsurdLem}  For any relational frame $(X,\comp)$:
\begin{enumerate}
\item\label{AbsurdLem1} the $0$ of $\lat(X,\comp)$ is the set of absurd states, also equal to $\neg_\comp 1$;
\item\label{AbsurdLem2} $\neg_\comp 0=1$ iff there is no $y\in X$ and absurd $x\in X$ with $x\comp y$.
\end{enumerate}
\end{lemma}
\begin{proof} For part \ref{AbsurdLem1}, an absurd state $x$ belongs to every $c_\comp$-fixpoint, since it holds vacuously that $\forall x'\comp x$ $\exists x''\compflip x'$: $x''\in A$, so the set of absurd states is a subset of every $c_\comp$-fixpoint and hence equal to $0$. Moreover, since $1=X$, we have $x\in\neg_\comp 1$ only if $x$ is absurd, so $\neg_\comp 1= 0$. Part \ref{AbsurdLem2} follows immediately from part \ref{AbsurdLem1}.\end{proof}

\begin{remark}\label{MoreGeneralNeg} A more general approach to negation, which would allow $\neg 1\neq 0$, uses triples $(X,\comp,F)$ where $(X,\comp)$ is a relational frame and $F$ is a distinguished $c_\comp$-fixpoint. Then we define the negation operation by
\[\neg_{\comp,F}A=\{x\in X\mid \forall x'\comp x\;(x'\in A\Rightarrow \exists x''\compflip x':x''\in F)\}.\]
Then $\neg_\comp$ is the special case $\neg_{\comp,0}$. The $\neg_{\comp,F}$ operation can in turn be obtained from the implication operation $\to_\comp$ studied in Appendix \ref{AppendixB}, as $\neg_{\comp,F} A= A\to_\comp F$. We will return to $\neg_{\comp,F}$ once more in Theorem \ref{NegThmAntitone}.\end{remark}

\begin{remark}\label{Downsets}It is easy to see that if $\comp$ is a reflexive and transitive relation $\leq$, then the lattice of $c_\comp$-fixpoints is simply the complete Heyting algebra of all downsets of $(X,\leq)$, as observed in \citealt[pp.~139-140]{Chiara2002} (cf.~\citealt[Prop.~4.1.1]{Conradie2020}, \citealt[Prop.~2.9(ii)]{Holliday2022}). Note, however, that this construction can only realize special complete Heyting algebras, namely those in which every element is a join of completely join-prime elements (see \citealt[Prop.~1.1]{Davey1979}). By contrast, the result in Theorem \ref{HeytOrthBoole}.\ref{HeytOrthBoole1} below applies to all complete Heyting algebras (cf.~\citealt[\S~4]{BH2019}).\end{remark}

\begin{example} Figures \ref{FramesForN5Fig} and \ref{FramesForO6Fig} show reflexive relational frames that give rise to the lattices with negations in Figures \ref{N5Fig} and \ref{O6Fig}, respectively. When drawing frames, an arrow with a triangle arrowhead from $y$ to $x$ indicates $y\compflip x$. Thus, we draw the directed graph $(X,\compflip)$ to represent the  frame $(X,\comp)$. Reflexive arrows are not shown but are assumed. The $c_\comp$-fixpoints, excluding $\varnothing$ and $X$, are outlined.  Looking at a diagram of a relational frame, one can check that $A$ is a $c_\comp$-fixpoint by checking that the following holds: 
\begin{itemize}
\item from any $x\in X\setminus A$, you can step forward along an arrow to a state $x'$ that cannot step backward along an arrow into $A$. 
\end{itemize}
Informally, ``from $x$ you can see a state that cannot be seen from $A$.'' 

For instance, in the reflexive frame on the left of Figure \ref{FramesForN5Fig}, $\{x\}$ is a $c_\comp$-fixpoint since  obviously any state outside of $\{x\}$ can see a state that cannot be seen from $\{x\}$; the only close call is $y$, but $y$ can see $z$, which cannot be seen from $\{x\}$. By contrast, $\{y\}$ is not a $c_\comp$-fixpoint, because although $x\not\in \{y\}$, $x$ cannot see a state that cannot be seen from $\{y\}$. For a more interesting calculation, consider the reflexive frame on the right of Figure \ref{FramesForO6Fig}. Here $\{z\}$ is a $c_\comp$-fixpoint; the only close call is $w$, but $w$ can see $u$, which cannot be seen from  $z$ (though $u$ can see $z$, but that is irrelevant). By contrast, $\{w\}$ is \textit{not} a $c_\comp$-fixpoint, because $z$ cannot see a state that cannot be seen from $w$ (note that the arrow between $v$ and $w$ is symmetric). 

\begin{figure}[h]

\begin{center}
\begin{tikzpicture}[->,>=stealth',shorten >=1pt,shorten <=1pt, auto,node
distance=2cm,thick,every loop/.style={<-,shorten <=1pt}]
\tikzstyle{every state}=[fill=gray!20,draw=none,text=black]
\node[label=center:$x$,inner sep=0pt,minimum size=.175cm] at (0,0) (D) {}; 
\node[label=center:$y$,inner sep=0pt,minimum size=.175cm] at (2,0) (F) {}; 
\node[label=center:$z$,inner sep=0pt,minimum size=.175cm] at (4,0) (H) {}; 
\node[label=center:$w$,inner sep=0pt,minimum size=.175cm] at (3,1) (W) {}; 

\path[{Triangle[open]}-{Triangle[open]},draw,thick] (D) to node {{}}  (F);
\path[-{Triangle[open]},draw,thick] (F) to node{{}}  (H);
\path[{Triangle[open]}-{Triangle[open]},draw,thick] (F) to node {{}}  (W);
\path[{Triangle[open]}-{Triangle[open]},draw,thick] (H) to node {{}}  (W);

\path[-, draw=blue, opacity=0.5, thick, rounded corners]  (4, .4) -- (4.4, .4) -- (4.4, -.4) -- (3.6, -.4) -- (3.6, .4) -- (4.4, .4) -- (4, .4); 

\path[-, draw=darkgreen, opacity=0.5, thick, rounded corners] (3.1, 1.5) -- (3.5, 1.5) -- (4.5, .5) -- (4.5, -.5) -- (3.5, -.5) -- (2.5, .5)  -- (2.5, 1.5) -- (3.5, 1.5) -- (3.1, 1.5); 

\path[-, draw=red, opacity=0.5, thick, rounded corners] (0, .4) -- (0.4, .4) -- (0.4, -.4) -- (-.4, -.4) -- (-.4, .4) -- (0.4, .4) -- (0, .4); 

\end{tikzpicture}\qquad\qquad \begin{tikzpicture}[->,>=stealth',shorten >=1pt,shorten <=1pt, auto,node
distance=2cm,thick,every loop/.style={<-,shorten <=1pt}]
\tikzstyle{every state}=[fill=gray!20,draw=none,text=black]
\node[label=center:$x$,inner sep=0pt,minimum size=.175cm] at (0,0) (D) {}; 
\node[label=center:$y$,inner sep=0pt,minimum size=.175cm] at (2,0) (F) {}; 
\node[label=center:$z$,inner sep=0pt,minimum size=.175cm] at (4,0) (H) {}; 
\node[label=center:$w$,inner sep=0pt,minimum size=.175cm] at (3,1) (W) {}; 

\path[{Triangle[open]}-{Triangle[open]},draw,thick] (D) to node {{}}  (F);
\path[-{Triangle[open]},draw,thick] (H) to node{{}}  (F);
\path[{Triangle[open]}-{Triangle[open]},draw,thick] (F) to node {{}}  (W);
\path[{Triangle[open]}-{Triangle[open]},draw,thick] (H) to node {{}}  (W);

\path[-, draw=darkgreen, opacity=0.5, thick, rounded corners]  (-.4, .5) -- (2.5, .5) -- (2.5, -.5) -- (-.5, -.5) -- (-.5, .5) -- (1, .5); 

\path[-, draw=red, opacity=0.5, thick, rounded corners] (3.1, 1.5) -- (3.5, 1.5) -- (4.5, .5) -- (4.5, -.5) -- (3.5, -.5) -- (2.5, .5)  -- (2.5, 1.5) -- (3.5, 1.5) -- (3.1, 1.5); 

\path[-, draw=blue, opacity=0.5, thick, rounded corners] (0, .4) -- (0.4, .4) -- (0.4, -.4) -- (-.4, -.4) -- (-.4, .4) -- (0.4, .4) -- (0, .4); 

\end{tikzpicture}\qquad\qquad\begin{tikzpicture}[->,>=stealth',shorten >=1pt,shorten <=1pt, auto,node
distance=2cm,thick,every loop/.style={<-,shorten <=1pt}]
\tikzstyle{every state}=[fill=gray!20,draw=none,text=black]
\node[label=center:$x$,inner sep=0pt,minimum size=.175cm] at (0,0) (D) {}; 
\node[label=center:$y$,inner sep=0pt,minimum size=.175cm] at (1,1) (F) {}; 
\node[label=center:$z$,inner sep=0pt,minimum size=.175cm] at (2,0) (H) {}; 

\path[-{Triangle[open]},draw,thick] (D) to node {{}}  (F);
\path[-{Triangle[open]},draw,thick] (F) to node {{}}  (H);

\path[-, draw=red, opacity=0.5, thick, rounded corners]  (0, .4) -- (.4, .4) -- (.4, -.4) -- (-.4, -.4) -- (-.4, .4) -- (.4, .4) -- (0, .4); 

\path[-, draw=darkgreen, opacity=0.5, thick, rounded corners] (1.1, 1.5) -- (1.5, 1.5) -- (2.5, .5) -- (2.5, -.5) -- (1.5, -.5) -- (.5, .5)  -- (.5, 1.5) -- (1.5, 1.5) -- (1.1, 1.5); 

\path[-, draw=blue, opacity=0.5, thick, rounded corners] (2, .4) -- (2.4, .4) -- (2.4, -.4) -- (1.6, -.4) -- (1.6, .4) -- (2.4, .4) -- (2, .4); 

\end{tikzpicture}
\end{center}
\caption{Reflexive frame representations of the lattice expansions in Figure \ref{N5Fig}.}\label{FramesForN5Fig}
\end{figure}
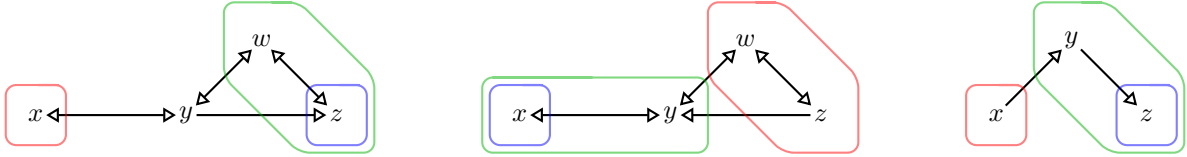

\begin{figure}[ht]

\begin{center}

\begin{tikzpicture}[->,>=stealth',shorten >=1pt,shorten <=1pt, auto,node
distance=2cm,thick,every loop/.style={<-,shorten <=1pt}]
\tikzstyle{every state}=[fill=gray!20,draw=none,text=black]
\node[label=center:$x$,inner sep=0pt,minimum size=.175cm] at (0,0) (A) {}; 
\node[label=center:$y$,inner sep=0pt,minimum size=.175cm] at (0,1.5) (B) {}; 
\node[label=center:$w$,inner sep=0pt,minimum size=.175cm] at (1.5,1.5) (C) {}; 
\node[label=center:$z$,inner sep=0pt,minimum size=.175cm] at (1.5,0) (D) {}; 

\path[{Triangle[open]}-{Triangle[open]},draw,thick] (A) to node {{}}  (B);
\path[{Triangle[open]}-{Triangle[open]},draw,thick] (B) to node {{}}  (C);
\path[{Triangle[open]}-{Triangle[open]},draw,thick] (C) to node {{}}  (D);

\path[-, draw=red, opacity=0.5, thick, rounded corners]  (0, .4) -- (.4, .4) -- (.4, -.4) -- (-.4, -.4) -- (-.4, .4) -- (.4, .4) -- (0, .4); 

\path[-, draw=darkgreen, opacity=0.5, thick, rounded corners]  (0, 2) -- (.5, 2) -- (.5, -.5) -- (-.5, -.5) -- (-.5, 2) -- (.5, 2) -- (0, 2); 

\path[-, draw=blue, opacity=0.5, thick, rounded corners]  (1.5, 2) -- (2, 2) -- (2, -.5) -- (1, -.5) -- (1, 2) -- (2, 2) -- (1.5, 2); 

\path[-, draw=orange, opacity=0.5, thick, rounded corners]  (1.5, .4) -- (1.9, .4) -- (1.9, -.4) -- (1.1, -.4) -- (1.1, .4) -- (1.9, .4) -- (1.5, .4); 

\end{tikzpicture}\qquad\qquad\begin{tikzpicture}[->,>=stealth',shorten >=1pt,shorten <=1pt, auto,node
distance=2cm,thick,every loop/.style={<-,shorten <=1pt}]
\tikzstyle{every state}=[fill=gray!20,draw=none,text=black]
\node[label=center:$x$,inner sep=0pt,minimum size=.175cm] at (0,0) (A) {}; 
\node[label=center:$y$,inner sep=0pt,minimum size=.175cm] at (0,1.5) (B) {}; 
\node[label=center:$v$,inner sep=0pt,minimum size=.175cm] at (1.5,1.5) (C) {}; 
\node[label=center:$u$,inner sep=0pt,minimum size=.175cm] at (1.5,0) (D) {}; 
\node[label=center:$z$,inner sep=0pt,minimum size=.175cm] at (3,1.5) (E) {}; 
\node[label=center:$w$,inner sep=0pt,minimum size=.175cm] at (3,0) (F) {}; 

\path[{Triangle[open]}-{Triangle[open]},draw,thick] (A) to node {{}}  (B);
\path[{Triangle[open]}-{Triangle[open]},draw,thick] (B) to node {{}}  (C);
\path[{Triangle[open]}-{Triangle[open]},draw,thick] (C) to node {{}}  (D);
\path[{Triangle[open]}-{Triangle[open]},draw,thick] (A) to node {{}}  (D);
\path[{Triangle[open]}-{Triangle[open]},draw,thick] (E) to node {{}}  (F);
\path[{Triangle[open]}-{Triangle[open]},draw,thick] (D) to node {{}}  (F);
\path[{Triangle[open]}-{Triangle[open]},draw,thick] (C) to node {{}}  (E);
\path[{Triangle[open]}-{Triangle[open]},draw,thick] (B) to node {{}}  (D);
\path[{Triangle[open]}-{Triangle[open]},draw,thick] (C) to node {{}}  (F);
\path[-{Triangle[open]},draw,thick] (C) to node {{}}  (A);
\path[-{Triangle[open]},draw,thick] (D) to node {{}}  (E);

\path[-, draw=red, opacity=0.5, thick, rounded corners]  (0, .4) -- (.4, .4) -- (.4, -.4) -- (-.4, -.4) -- (-.4, .4) -- (.4, .4) -- (0, .4); 

\path[-, draw=darkgreen, opacity=0.5, thick, rounded corners]  (0, 2) -- (.5, 2) -- (.5, -.5) -- (-.5, -.5) -- (-.5, 2) -- (.5, 2) -- (0, 2); 

\path[-, draw=blue, opacity=0.5, thick, rounded corners]  (3, 2) -- (3.5, 2) -- (3.5, -.5) -- (2.5, -.5) -- (2.5, 2) -- (3.4, 2) -- (3, 2); 

\path[-, draw=orange, opacity=0.5, thick, rounded corners]  (3, 1.9) -- (3.4, 1.9) -- (3.4, 1.1) -- (2.6, 1.1) -- (2.6, 1.9) -- (3.4, 1.9) -- (3, 1.9); 

\end{tikzpicture}

\end{center}

\caption{Reflexive frame representations of the lattice expansions in Figure \ref{O6Fig}.}\label{FramesForO6Fig}
\end{figure}
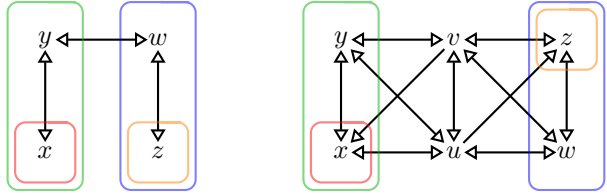

A more efficient procedure for calculating $c_\comp$-fixpoints, using Ganter's \citeyearpar{Ganter2010} algorithm for calculating fixpoints of a closure operator, is implemented in one of the notebooks mentioned in \S~\ref{Intro}.
\end{example}

From this starting point, algebras for intuitionistic logic, orthologic, and classical logic arise from natural constraints on the  relation $\comp$. It has long been known that reflexive frames in which $\comp$ is \textit{symmetric} give rise to ortholattices (\citealt[\S\S~32-4]{Birkhoff1940}), and all complete ortholattices can be so represented (\citealt{MacLaren1964}), which yields a relational semantics for orthologic (\citealt{Goldblatt1974}, cf.~\citealt{Dishkant1972}). To characterize the complete Heyting case, \citealt{Holliday2022} uses the following concepts.\footnote{\citealt[\S~3.4]{Holliday2021} uses the pre-refinement and post-refinement relations to translate from single relation structures, as used in this paper and \citealt{Holliday2022}, to doubly ordered structures, as used in the duality for complete lattices in \citealt{Massas2020}.}

\begin{definition}\label{RefinementDef} Given a relational frame $(X,\comp)$ and $x,y\in X$:
\begin{enumerate}
\item\label{RefinementDef1} $x$ \textit{pre-refines $y$} if for all $z\in X$, $z\comp x$ implies $z\comp y$;
\item $x$ \textit{post-refines $y$} if for all $z\in X$, $x\comp z$ implies $y\comp z$;
\item $x$ \textit{refines} $y$ if $x$ pre-refines and post-refines $y$;
\item $x$ is \textit{compossible with $y$} if there is a non-absurd $w\in X$ that refines $x$ and pre-refines $y$.
\end{enumerate} 
We say that $\comp$ is \textit{compossible} if whenever $x\comp y$, then $x$ is compossible with $y$.\end{definition}

\noindent Note that if $\comp$ is symmetric, then pre-refinement and post-refinement are equivalent, and $x$ is compossible with $y$ just in case they have a common non-absurd refinement.

The following lemma will be useful below.

\begin{lemma}\label{prerefinelem} For any relational frame $(X,\comp)$ and $x,y\in X$, if $x$ pre-refines $y$, then for every $c_\comp$-fixpoint $A$, if $y\in A$, then $x\in A$.
\end{lemma}
\begin{proof} If $x'\comp x$, then since $x$ pre-refines $y$, $x'\comp y$. Then since $y\in A$, there is an $x''\compflip x'$ with $x''\in A$. Hence for any $x'\comp x$ there is an $x''\compflip x'$ with $x''\in A$, which shows $x\in A$.
\end{proof}
\noindent Note that if $x$ post-refines $y$, then for any $A$ that $y$ rejects in the sense of Remark \ref{IntuitivePicture}, $x$ rejects $A$ too. Hence if $x$ refines $y$, then $x$ accepts every proposition that $y$ does and rejects every proposition that $y$ does. 

 Now we can characterize complete Heyting algebras, ortholattices, and Boolean algebras using relational frames as follows. For a proof, see  \citealt[Theorems 2.21 and 3.18]{Holliday2022}. Part \ref{HeytOrthBoole1} also follows from our results concerning lattices with implications in  Appendix \ref{AppendixB}.

\begin{theorem}\label{HeytOrthBoole}$\,$
\begin{enumerate}
\item\label{HeytOrthBoole1} $(L,\neg)$ is a complete Heyting algebra with pseudocomplementation $\neg$  iff $(L,\neg)$ is isomorphic to \\ ${(\lat(X,\comp), \neg_\comp)}$ for a relational frame $(X,\comp)$ in which $\comp$ is reflexive and compossible. 
\item $(L,\neg)$ is a complete ortholattice with orthocomplementation $\neg$ iff $(L,\neg)$ is isomorphic to $(\lat(X,\comp), \neg_\comp)$ for a relational frame $(X,\comp)$ in which $\comp$ is reflexive and symmetric.
\item $(L,\neg)$ is a complete Boolean algebra with Boolean negation $\neg$ iff $(L,\neg)$ is isomorphic to $(\lat(X,\comp), \neg_\comp)$ for a relational frame $(X,\comp)$ in which $\comp$ is reflexive, symmetric, and compossible.
\end{enumerate}
\end{theorem}

Not every pseudocomplemented lattice $(L,\neg)$ is a Heyting algebra, as Heyting algebras require a \textit{relative pseudocomplementation} $\to$ such that for all $a,b,c\in L$, $a\wedge b\leq c$ iff $a\leq b\to c$, which implies that $L$ is distributive. Thus, let us isolate a condition just for pseudocomplementation, which is the conjunction of two conditions: $a\wedge \neg a=0$, and $a\wedge b=0$ implies $a\leq \neg b$. Let us also isolate the condition for  double negation introduction that we want for weak pseudocomplementations, as well as the condition for double negation elimination that turns weak pseudocomplementations into orthocomplementations (Lemma \ref{UsefulLem}.\ref{UsefulLem3}).

\begin{proposition}\label{CorrLemm} For any relational frame $(X,\comp)$, in each of the following pairs, (a) and (b) are equivalent:
\begin{enumerate}
\item \label{CorrLemm1}
\begin{enumerate}
\item for all $c_\comp$-fixpoints $A$, we have $A\cap \neg_\comp A=0$; 
\item for all non-absurd $x\in X$, there is a $z\comp x$ that pre-refines $x$.
\end{enumerate}
\item\label{CorrLemm3}
\begin{enumerate}
\item for all $c_\comp$-fixpoints $A$, we have $A\subseteq\neg_\comp\neg_\comp A$;
\item \textit{pseudosymmetry}: for all $x\in X$ and $y\comp x$, there is a $z\comp y$ that pre-refines $x$.
\end{enumerate}
\item\label{CorrLemm2}
\begin{enumerate}
\item for all $c_\comp$-fixpoints $A,B$, if $A\cap B=0$, then $A\subseteq\neg_\comp B$.
\item \textit{weak compossibility}: for all $x\in X$ and $y\comp x$, there is a non-absurd $z$ that pre-refines $y$ and $x$.
\end{enumerate}
\item\label{CorrLemm4} 
\begin{enumerate}
\item for all $c_\comp$-fixpoints $A$, we have $\neg_\comp\neg_\comp A\subseteq A$;
\item for all $x\in X$ and $y\comp x$, there is a $y'\comp x$ such that for all $z\in X$, if $z\comp y'$ then $y\comp z$.
\end{enumerate}
\end{enumerate}
\end{proposition}
\begin{proof} For part \ref{CorrLemm1}, suppose (b) holds,  $x\in A$, and $x\not\in 0$, so by Lemma \ref{AbsurdLem}.\ref{AbsurdLem1}, $x$ is non-absurd. Then by (b) there is a $z\comp x$ that pre-refines $x$, which with Lemma \ref{prerefinelem} implies $z\in A$ and hence $x\not\in \neg_\comp A$. This proves $A\cap\neg_\comp A\subseteq 0$. Conversely, suppose (b) does not hold, so there is a non-absurd $x$ that is not pre-refined by any state open to $x$. First, we claim $x\in \neg_\comp c_\comp (\{x\})$. For suppose $y\comp x$. Since $y$ does not pre-refine $x$, there is a $z\comp y$ such that $z\not\comp x$. This shows $y\not\in c_\comp (\{x\})$, so $x\in \neg_\comp c_\comp(\{x\})$ and hence $x\in c_\comp(\{x\})\cap \neg_\comp c_\comp(\{x\})$. Then since $x$ is non-absurd, we have $c_\comp(\{x\})\cap \neg_\comp c_\comp(\{x\}) \neq 0$.

For part \ref{CorrLemm3}, suppose (b) holds, $x\in A$, and $y\comp x$. Then by pseudosymmetry, there is a $z\comp y$ that pre-refines $x$. Since $x\in A$, it follows by Lemma \ref{prerefinelem} that $z\in A$, which with $z\comp y$ implies $y\not\in \neg_\comp A$. Thus, we have $x\in \neg_\comp \neg_\comp A$, so $A\subseteq \neg_\comp \neg_\comp A$. Conversely, suppose (b)  does not hold, so there are $x,y\in X$ with $y\comp x$ such that for all $z\comp y$, there is some $w\comp z$ with $w\not\comp x$, which implies $z\not\in c_\comp(\{x\})$. Hence $y\in \neg c_\comp(\{x\})$, which with $y\comp x$ implies $x\not\in \neg_\comp\neg c_\comp(\{x\})$. Yet $x\in c_\comp(\{x\})$, so $c_\comp(\{x\})\not\subseteq \neg_\comp\neg c_\comp(\{x\})$.

For part \ref{CorrLemm2}, suppose (b) holds, $A\cap B=0$, $x\in A$, but $x\not\in \neg_\comp B$, so there is a $y\comp x$ with $y\in B$. Then by weak compossibility, there is a non-absurd $z$ that pre-refines $y$ and $x$. Hence $z\in A\cap B$ by Lemma \ref{prerefinelem}. Since $z$ is non-absurd, it follows that $A\cap B\neq 0$ by Lemma \ref{AbsurdLem}.\ref{AbsurdLem1}. Conversely,  suppose (b) does not hold, so there are $x,y\in X$ with $y\comp x$ but there is no non-absurd $z$ that pre-refines $y$ and $x$. It follows that $c_\comp (\{y\})\cap c_\comp(\{x\})=0$. But since $y\comp x$, we have $x\not\in \neg_\comp c_\comp(\{y\})$, so $c_\comp (\{x\})\not\subseteq \neg_\comp c_\comp(\{y\})$.

For part \ref{CorrLemm4}, suppose (b) holds and $x\not\in A$, so there is a $y\comp x$ such that for all $w\compflip y$, $w\not\in A$. By (b), there is a $y'\comp x$ such that for all $z\in X$,  $z\comp y'$ implies $y\comp z$ and hence $z\not\in A$ by the previous sentence. Thus, $y'\in \neg_\comp A$, which with $y'\comp x$ implies $x\not\in \neg_\comp \neg_\comp A$. Conversely, suppose (b) does not hold, so there is some $y\comp x$ such that (i) for all $y'\comp x$, there is a $z\comp y'$ such that  $y\not \comp z$. Let $A=\{w\in X\mid y\not\comp w\}$. Then $A$ is a $c_\comp$-fixpoint, for if $v\not\in A$, then $y\comp v$ and for all $u\compflip y$, $u\not\in A$. Moreover, $x\in \neg_\comp\neg_\comp A$ by (i), but $x\not\in A$.\end{proof}

\begin{remark}Note the relation between the (b) conditions in parts \ref{CorrLemm1} and \ref{CorrLemm3} of Lemma \ref{CorrLemm}: the first says that if $y\comp x$, then there is a pre-refinement of $x$ that is open to $x$, while the second says that if $y\comp x$, then there is a pre-refinement of $x$ that is open to $y$. In Appendix \ref{AppendixB}, we consider a pair of analogous conditions for an implication $\to_\comp$ in place of the negation $\neg_\comp$ (Lemma \ref{ImpCorr}).\end{remark}

\noindent Concerning part \ref{CorrLemm1}, it turns out (Theorem \ref{NegThm}.\ref{NegThm2}) that for the purposes of representing protocomplementations, we can strengthen the condition in \ref{CorrLemm1}(b) to reflexivity without loss of generality. Concerning part \ref{CorrLemm3},  pseudosymmetry  is a weakening of the symmetry property that yields ortholattices. Pseudosymmetry says that if $y$ is open to $x$, then  while $x$ might not be open to $y$, some pre-refinement of $x$ is open to $y$. In the terms of Remark \ref{IntuitivePicture}, pseudosymmetry corresponds to the condition that for any proposition $A$ and $y\in X$, \[\mbox{if $y$ accepts $\neg A$, then $y$ rejects $A$.}\] For assume pseudosymmetry and that $y$ does not reject $A$, so there is an $x\compflip y$ with $x\in A$; then taking $z$ as in the statement of pseudosymmetry, we have $z\in A$ by Lemma \ref{prerefinelem}, so $z\comp y$ implies that $y$ does not accept $\neg A$. Conversely, if pseudosymmetry fails, then $y$ does not reject $c_\comp(\{x\})$ but does accept $\neg_\comp c_\comp(\{x\})$.

\begin{remark}\label{DunnDifference} In Dunn's setting with triples $(X,\comp,\leq)$ referenced in \S~\ref{FrameToLatSection}, $A\subseteq\neg\neg A$ corresponds to the symmetry of $\comp$ (\citealt[Thm.~2.10]{Dunn2005}, \citealt[Thm.~11.41]{Restall2000}), which in our setting overshoots and makes $\neg$ an orthocomplementation.
\end{remark}

We will also consider the following strengthening of pseudosymmetry.

\begin{definition}\label{StrongPseudo} A relational frame $(X,\comp)$ is \textit{strongly pseudosymmetric} if for all $x\in X$ and $y\comp x$, there is a $z\comp y$ such that $z$ pre-refines $x$ and $x$ pre-refines $z$.
\end{definition}
\noindent Note that if $z$ pre-refines $x$ and vice versa, then $x$ and $z$ belong to exactly the same propositions, i.e., $c_\comp$-fixpoints, by Lemma \ref{prerefinelem} (though they may \textit{reject} different propositions).

We will see (Theorem \ref{NegThm}.\ref{NegThm4}) that lattices with weak pseudocomplementations can be represented using  pseudosymmetric reflexive frames---or even strongly pseudosymmetric ones at the expense of a bigger frame.

\begin{example} In Figure \ref{FramesForN5Fig}, the reflexive frame on the left is pseudosymmetric but not strongly pseudosymmetric; the frame in the middle is strongly pseudosymmetric but not symmetric; and the frame on the right is not pseudosymmetric. In Figure \ref{FramesForO6Fig}, the reflexive frame on the left is symmetric while the one on the right is strongly pseudosymmetric but not symmetric.
\end{example}

Finally, let us turn from lattices to our formal language $\mathcal{L}$. Proposition \ref{FrameToLat} leads immediately to the following relational semantics for  $\mathcal{L}$.

\begin{definition}\label{RelModel} A \textit{relational model} is a triple $\mathcal{M}=(X,\comp,V)$ where $(X,\comp)$ is a relational frame and $V$ maps each $p\in\mathsf{Prop}$ to a $c_\comp$-fixpoint $V(p)\subseteq X$. We define a forcing relation between states in $\mathcal{M}$ and formulas of $\mathcal{L}$ as follows:
\begin{enumerate}
\item $\mathcal{M},x\Vdash p$ iff $x\in V(p)$;
\item $\mathcal{M},x\Vdash\neg\varphi$ iff for all $x'\comp x$, $\mathcal{M},x'\nVdash\varphi$;
\item $\mathcal{M},x\Vdash\varphi\wedge\psi$ iff $\mathcal{M},x\Vdash\varphi$ and $\mathcal{M},x\Vdash\psi$;
\item $\mathcal{M},x\Vdash\varphi\vee\psi$ iff $\forall x'\comp x$ $\exists x''\compflip x'$: $\mathcal{M},x''\Vdash\varphi$ or $\mathcal{M},x''\Vdash\psi$.
\end{enumerate}
Given a class $\mathbb{C}$ of relational frames, we define $\varphi\vDash_\mathbb{C}\psi$ if for all $(X,\comp)\in \mathbb{C}$, all models $\mathcal{M}$ based on $(X,\comp)$,  and all $x\in X$, if $\mathcal{M},x\Vdash \varphi$, then $\mathcal{M},x\Vdash\psi$.
\end{definition}

Where $\llbracket\varphi\rrbracket^\mathcal{M}=\{x\in X\mid \mathcal{M},x\Vdash\varphi\}$, an easy induction shows the following.
\begin{lemma} For any relational model $\mathcal{M}=(X,\comp,V)$ and $\varphi\in\mathcal{L}$, $\llbracket\varphi\rrbracket^\mathcal{M}$ is a $c_\comp$-fixpoint.
\end{lemma}

\begin{figure}[h]
\begin{center}
\begin{tikzpicture}[->,>=stealth',shorten >=1pt,shorten <=1pt, auto,node
distance=2cm,thick,every loop/.style={<-,shorten <=1pt}]
\tikzstyle{every state}=[fill=gray!20,draw=none,text=black]
\node[label=center:$x$,inner sep=0pt,minimum size=.175cm] at (0,0) (D) {}; 
\node[label=center:$y$,inner sep=0pt,minimum size=.175cm] at (2,0) (F) {}; 
\node[label=center:$z$,inner sep=0pt,minimum size=.175cm] at (4,0) (H) {}; 
\node[label=center:$w$,inner sep=0pt,minimum size=.175cm] at (3,1) (W) {}; 

\node at (.6,-.3) (q) {$q$};

\node at (2.7,-.3) (q) {$p$};

\node at (4.7,-.3) (q) {$r$};

\path[{Triangle[open]}-{Triangle[open]},draw,thick] (D) to node {{}}  (F);
\path[-{Triangle[open]},draw,thick] (H) to node{{}}  (F);
\path[{Triangle[open]}-{Triangle[open]},draw,thick] (F) to node {{}}  (W);
\path[{Triangle[open]}-{Triangle[open]},draw,thick] (H) to node {{}}  (W);

\path[-, draw=darkgreen, opacity=0.5, thick, rounded corners]  (-.4, .5) -- (2.5, .5) -- (2.5, -.5) -- (-.5, -.5) -- (-.5, .5) -- (1, .5); 

\path[-, draw=red, opacity=0.5, thick, rounded corners] (3.1, 1.5) -- (3.5, 1.5) -- (4.5, .5) -- (4.5, -.5) -- (3.5, -.5) -- (2.5, .5)  -- (2.5, 1.5) -- (3.5, 1.5) -- (3.1, 1.5); 

\path[-, draw=blue, opacity=0.5, thick, rounded corners] (0, .4) -- (0.4, .4) -- (0.4, -.4) -- (-.4, -.4) -- (-.4, .4) -- (0.4, .4) -- (0, .4); 

\end{tikzpicture}
\end{center}
\caption{A valuation on the reflexive frame from the middle of Figure \ref{FramesForN5Fig}.}\label{ValFig}
\end{figure}
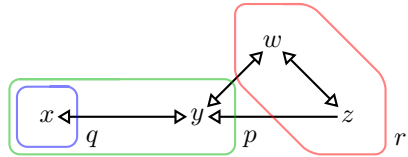

\begin{example} Consider a valuation $V$ on the reflexive frame in Figure \ref{ValFig} that sets $V(p)=\{x,y\}$,  $V(q)=\{x\}$, and $V(r)=\{w,z\}$. Then observe that $\mathcal{M},y\Vdash q\vee r$, even though $\mathcal{M},y\nVdash q$ and $\mathcal{M},y\nVdash r$. Thus, $\mathcal{M},y\Vdash p\wedge (q\vee r)$. However, $\mathcal{M},y\nVdash (p\wedge q)\vee (p\wedge r)$, since $y$ can see $w$, but $w$ cannot be seen from a state forcing $p\wedge q$ (namely from $x$) or a state forcing $p\wedge r$ (since there are no such states). Thus, this model provides a counterexample to the distributive law. Also observe that no state  forces $\neg p$, so $\mathcal{M},z\Vdash \neg\neg p$, yet $\mathcal{M},z\nVdash p$. Thus, this model provides a counterexample to double negation elimination. Similar calculations can be done upon evaluating propositional variables as other $c_\comp$-fixpoints in Figures \ref{FramesForN5Fig} or \ref{FramesForO6Fig}.\end{example}

\subsection{Discrete representation of lattices with negation}\label{DiscreteRep}

Having seen how to go from a relational frame to a lattice with negation, let us now consider the converse direction: given a lattice with negation, we build a relational frame into whose lattice of $c_\comp$-fixpoints the given lattice embeds. The following definition and result are from \citealt{Holliday2022} with some details expanded.

\begin{definition}\label{Good} Let $L$ be a lattice and $P$ a set of pairs of elements of $L$. Define a binary relation $\comp$ on $P$ by 
$(a,b)\comp (c,d)$ if $c\not\leq b$. Then we say $P$ is \textit{separating} if for all $a,b\in L$:
\begin{enumerate}
\item\label{Good2} if $a\not\leq b$, then there is a $(c,d)\in P$ with $c\leq a$ and $c\not\leq b$;
\item\label{Good3} for all $(c,d)\in P$, if $c\not\leq b$, then there is a $(c',d')\comp (c,d)$ such that for all $(c'',d'')\compflip (c',d')$, we have $c''\not\leq b$.
\end{enumerate}
\end{definition}
\noindent One can interpret the pairs in $P$ intuitively as in Remark \ref{IntuitivePicture}: the state $(a,b)$ accepts everything entailed by proposition~$a$ and rejects everything that entails proposition $b$; and $(a,b)$ is open to $(c,d)$ if $(a,b)$ does not reject anything that $(c,d)$ accepts, i.e., $c\not\leq b$.

A \textit{complete embedding} of a lattice $L$ into a lattice $L'$ is an injective map $f:L\longrightarrow L'$ that preserves all existing meets and joins of $L$. A complete embedding of lattice expansions $(L,\neg)$ is defined in the same way but also requiring the preservation of $\neg$.

\begin{proposition}\label{CompRep} Let $L$ be a lattice and $P$ a separating set of pairs of elements of $L$. For $a\in L$, define $ f(a)= \{(x,y) \in P\mid x\leq a\}$. Then:
\begin{enumerate}
\item\label{CompRep1} $f$ is a complete embedding of $L$ into $\lat(P,\comp)$;
\item\label{CompRep2} if $L$ is complete, then $f$ is an isomorphism from $L$ to $\lat(P,\comp)$.
\end{enumerate}
\end{proposition}

\begin{proof} For part \ref{CompRep1}, condition \ref{Good3} of Definition \ref{Good} implies that $ f(b)$ is a $c_\comp$-fixpoint for each $b\in L$. Clearly $f$ preserves all existing meets: \[ f(\underset{a\in A}{\bigwedge} a)=\{(x,y)\in P\mid x\leq \underset{a\in A}{\bigwedge} a\}=\underset{a\in A}{\bigcap}\{(x,y)\in P\mid x\leq a\}= \underset{a\in A}{\bigcap} f(a).\] For joins, to see that $ f(\bigvee A )\subseteq  \bigvee \{ f(a)\mid a\in A\}$, suppose that $(x,y)\in  f(\bigvee A )$ and $(x',y')\comp (x,y)$. Hence $x\leq \bigvee A $ but $x\not\leq y'$, so $\bigvee A  \not\leq y'$, which implies $a\not\leq y'$ for some $a\in A$. Then part \ref{Good2} of Definition \ref{Good} yields an $(x'',y'')\in  f(a)$ with $(x',y')\comp (x'',y'')$. This proves that $(x,y)\in \bigvee \{ f(a)\mid a\in A\}$. The converse inclusion $ \bigvee \{ f(a)\mid a\in A\}\subseteq f(\bigvee A )$ follows from order preservation, which follows from meet preservation. Finally, part~\ref{Good2} of Definition \ref{Good} ensures that $f$ is injective.

For part \ref{CompRep2}, we claim $f$ is surjective. Given a $c_{\comp}$-fixpoint $A$, define $a=\bigvee\{a_i\mid \exists b_i:(a_i,b_i)\in A\}$. We claim $A= f(a)$. For $A\subseteq f(a)$, suppose $(a_i,b_i)\in A$. Then by definition of $a$,  $a_i\leq a$, so $(a_i,b_i)\in   f(a)$. For $A\supseteq f(a)$, suppose $(c,d)\in  f(a)$, so $c\leq a$. Since $A$ is a $c_{\comp}$-fixpoint, to show $(c,d)\in A$, it suffices to show that for every $(c,'d')\comp (c,d)$ there is a $(c'',d'')\compflip (c',d')$ with $(c'',d'')\in A$. Suppose $(c,'d')\comp (c,d)$, so $c\not\leq d'$, which with $c\leq a$ implies $a\not\leq d'$. Then for some $(a_i,b_i)\in A$, we have $a_i\not\leq d'$.  Setting $(c'',d'')=(a_i,b_i)$, from $a_i\not\leq d'$ we have $(c',d')\comp (c'',d'')$, and $(c'',d'')\in A$, so we are done.
\end{proof}

Different choices of a separating set $P$ of pairs can lead to more or less efficient representations of different types of lattices. Cases where $L$ is an arbitrary lattice, ortholattice, or Heyting algebra are covered in \citealt[Prop.~3.16]{Holliday2022}. In the case of bounded lattices with $\neg$, we choose the pairs with the $\neg$ operation in mind. But the following theorem applies to bounded lattices in general, given the point in \S~\ref{AlgSection} that any bounded lattice can be equipped with a weak pseudocomplementation. In Section \ref{Conditionals} and Appendix \ref{AppendixB}, we prove analogous theorems for bounded lattices with implications. Recall that a set of elements in a lattice $L$ is \textit{join-dense} (resp.~\textit{meet-dense}) if every element of $L$ is a join (resp.~meet) of a (possibly infinite) set of elements of $L$. E.g., the set of all elements of $L$ is trivially join- (and meet-) dense in $L$.  

\begin{theorem}\label{NegThm} Let $L$ be a bounded lattice, $\mathrm{V}$ a join-dense set of elements of $L$, and $\Lambda$ a meet-dense set of elements of $L$. Given a set $P$ of pairs of elements of $L$, define $\comp$ on $P$ by $(a, b)\comp (c, d)$ if $c\not\leq b$.  
\begin{enumerate}
\item\label{NegThm1} If $\neg$ is a precomplementation on $L$, then where
\[P = \{(a,\neg a)\mid a\in L\} \cup \{(1,b)\mid b\in \Lambda\},\]
there is a complete embedding of $(L,\neg)$ into $(\lat(P,\comp),\neg_\comp)$.
\item\label{NegThm2} If $\neg$ is a protocomplementation on $L$, then where
\[P = \{(a,\neg a)\mid a\in L, a\neq 0\} \cup \{(1,b)\mid b\in \Lambda, b\neq 1\},\]
there is a complete embedding of $(L,\neg)$ into $(\lat(P,\comp),\neg_\comp)$, and $\comp$ is reflexive.
\item\label{NegThm3} If $\neg$ is an ultraweak pseudocomplementation on $L$, then  where
\[P = \{(a,\neg a)\mid a\in \mathrm{V}\} \cup \{(1,b)\mid b\in \Lambda\},\]
 there is a complete embedding of $(L,\neg)$ into $(\lat(P,\comp),\neg_\comp)$, and $\comp$ is  pseudosymmetric (and strongly pseudosymmetric if $\mathrm{V}=L$).
\item\label{NegThm4} If $\neg$ is a weak pseudocomplementation on $L$, then where
\[P = \{(a,\neg a)\mid a\in \mathrm{V}, a\neq 0\} \cup \{(1,b)\mid b\in \Lambda, b\neq 1\},\]
there is a complete embedding of $(L,\neg)$ into $(\lat(P,\comp),\neg_\comp)$, and $\comp$ is reflexive and  pseudosymmetric (and strongly pseudosymmetric if $\mathrm{V}=L$). Moreover, if $\neg$ is a pseudocomplementation, then $\comp$ is weakly compossible.
\end{enumerate}
In each case, if $L$ is complete, then the embedding is an isomorphism.
\end{theorem}

\begin{proof} Note first that (i) for all parts of the theorem, for $(a,b)\in P$, we have $\neg a\leq b$, using that $\neg 1=0$.

First we claim that in each part,  $P$ is separating in the sense of Definition~\ref{Good}.  To prove part \ref{Good2} of Definition~\ref{Good}, suppose $a\not\leq b$. In parts \ref{NegThm1} and \ref{NegThm2} of the  theorem, we  take $(c,d)=(a,\neg a)$. Since $a\neq 0$, we have $(a,\neg a)\in P$. In parts \ref{NegThm3} and \ref{NegThm4} of the theorem, from $a\not\leq b$ we obtain a nonzero $a'\in\mathrm{V}$ such that $a'\leq a$ and $a'\not\leq b$, and we set $(c,d)=(a', \neg a')$.  To prove part \ref{Good3} of Definition \ref{Good}, suppose $(c,d)\in P$ and $c\not\leq b$. Hence there is some $b'\in\Lambda$ such that $c\not\leq b'$ and $b\leq b'$.  Let $(c',d')=(1,b')$. Since $c\not\leq b'$, we have $b'\neq 1$ and hence $(c',d')\in P$, and also  $(c',d')\comp (c,d)$. Now consider any $(c'',d'')\in P$ with $(c',d')\comp (c'',d'')$. Then $c''\not\leq d'=b'$, so $c''\not\leq b$. Hence part \ref{Good3} of Definition \ref{Good} holds. Thus, by Proposition \ref{CompRep}, $f$ is a complete embedding of $L$ into $\mathfrak{L}(P,\comp)$, which is a lattice isomorphism if $L$ is complete.  

Next we claim that for each part, $ f(\neg a)=\neg_\comp f(a)$. Suppose $(x,y)\in  f(\neg a)$, so $x\leq\neg a$, and $(x',y')\comp (x,y)$. If $x'\leq a$, then $\neg a\leq\neg x'$, which with $x\leq\neg a$ implies $x\leq \neg x'$, which with $\neg x'\leq y'$ from (i) implies $x\leq y'$, contradicting $(x',y')\comp (x,y)$. Thus, $x'\not\leq a$, so $(x',y')\not\in  f(a)$. Hence $(x,y)\in \neg_\comp  f(a)$. Conversely, let $(x,y)\in P\setminus f(\neg a)$, so $x\not\leq \neg a$. In part \ref{NegThm1}, we immediately have $(a,\neg a)\in P$, and $(a,\neg a)\comp (x,y)$, so $(x,y)\not\in \neg_\comp f(a)$. For part \ref{NegThm2}, we use that  $\neg 0=1$, so from $x\not\leq \neg a$ we have  $a\neq 0$, so $(a,\neg a)\in P$. For part \ref{NegThm3}, we have that $x\not\leq \neg a$ implies $a\not\leq\neg x$ (Lemma \ref{UsefulLem}.\ref{UsefulLem2}), so there is some $a'\in\mathrm{V}$ such that $a'\leq a$ but $a'\not\leq\neg x$,  so $x\not\leq \neg a'$. Hence $(a',\neg a')\in P$ and $(a',\neg a')\comp (x,y)$,  which with $a'\leq a$ yields $(x,y)\not\in \neg_\comp f(a)$. For part \ref{NegThm4}, we again use that $\neg 0=1$, so from $x\not\leq \neg a'$ we have $a'\neq 0$, so $(a',\neg a')\in P$.

For parts \ref{NegThm2} and \ref{NegThm4}, that $\comp$ is reflexive follows from the anti-inflationary property of semicomplementations (Lemma \ref{UsefulLem}.\ref{UsefulLem1}). For parts \ref{NegThm3} and \ref{NegThm4}, we prove  pseudosymmetry. Suppose $(c,d)\comp (a,b)$, which implies $\neg c\leq d$ by (i) and $a\not\leq d$. Hence $a\not\leq \neg c$, so there is a nonzero $a'\in\mathrm{V}$ such that $a'\leq a$ but $a'\not\leq \neg c$, which implies $c\not\leq \neg a'$  (Lemma \ref{UsefulLem}.\ref{UsefulLem2}). Hence  $(a',\neg a')\comp (c,d)$, and since $a'\leq a$, $(a',\neg a')$ pre-refines $(a,b)$. If $\mathrm{V}=L$, then we can take $a'=a$, in which case $(a,\neg a)$ pre-refines $(a,b)$ and vice versa. Finally, for the claim about pseudocomplementations in part \ref{NegThm4}, if $(a,b)\comp (c,d)$, then $a\wedge c\neq 0$, for otherwise $c\leq \neg a$, and $\neg a \leq b$ by (i), so $c\leq b$, contradicting $(a,b)\comp (c,d)$. Hence there is a nonzero $e\in \mathrm{V}$ with $e\leq a\wedge c$. Then $(e,\neg e)\in P$, and since $e\leq a$ and $e\leq c$, we have that $(e,\neg e)$ pre-refines $(a,b)$ and $(c,d)$. Hence $\comp$ is weakly compossible.\end{proof}

\begin{example} As an illustration of part \ref{NegThm4} of Theorem \ref{NegThm}, consider the lattice with weak pseudocomplementation shown on the left of Figure \ref{NegThmFig}. Setting $\mathrm{V}=\Lambda = \{2,3\}$, we have \[P = \{(a,\neg a)\mid a\in \mathrm{V}, a\neq 0\} \cup \{(1,b)\mid b\in \Lambda, b\neq 1\} = \{(2,0),(3,0)\}\cup \{(1,2),(1,3)\}.\]
Then the definition of $\comp$ by $(a, b)\comp (c, d)$ if $c\not\leq b$ yields the relational frame on the right of Figure \ref{NegThmFig}.

\begin{figure}[h]
\begin{center}
\begin{minipage}{2in}
\begin{tikzpicture}[->,>=stealth',shorten >=1pt,shorten <=1pt, auto,node
distance=2cm,thick,every loop/.style={<-,shorten <=1pt}]
\tikzstyle{every state}=[fill=gray!20,draw=none,text=black]
\node[label=right:$$,inner sep=0pt,minimum size=.175cm] (1) at (0,0) {{$1$}};
\node[label=left:$$,inner sep=0pt,minimum size=.175cm] (x) at (-1,-1) {{\textcolor{red}{$2$}}};

\node[label=right:$$,inner sep=0pt,minimum size=.175cm] (z) at (1,-1) {{\textcolor{darkgreen}{$3$}}};
\node[label=right:$$,inner sep=0pt,minimum size=.175cm] (0) at (0,-2) {{$0$}};

\path (1) edge[-] node {{}} (x);
\path (1) edge[-] node {{}} (z);
\path (x) edge[-] node {{}} (0);
\path (z) edge[-] node {{}} (0);

\path (x) edge[->,dashed,bend right,gray] node {{}} (0);
\path (z) edge[->,dashed,bend left,gray] node {{}} (0);

\end{tikzpicture}
\end{minipage}\begin{minipage}{2in}
\begin{tikzpicture}[->,>=stealth',shorten >=1pt,shorten <=1pt, auto,node
distance=2cm,thick,every loop/.style={<-,shorten <=1pt}]
\tikzstyle{every state}=[fill=gray!20,draw=none,text=black]
\node[inner sep=0pt,minimum size=.175cm] at (0,0) (A) {$(2,0)$}; 
\node[inner sep=0pt,minimum size=.175cm] at (0,1.5) (B) {$(1,2)$}; 
\node[inner sep=0pt,minimum size=.175cm] at (2,1.5) (C) {$(1,3)$}; 
\node[inner sep=0pt,minimum size=.175cm] at (2,0) (D) {$(3,0)$}; 

\path[{Triangle[open]}-,draw,thick] (A) to node {{}}  (B);
\path[{Triangle[open]}-{Triangle[open]},draw,thick] (B) to node {{}}  (C);
\path[-{Triangle[open]},draw,thick] (C) to node {{}}  (D);
\path[{Triangle[open]}-{Triangle[open]},draw,thick] (A) to node {{}}  (D);
\path[{Triangle[open]}-{Triangle[open]},draw,thick] (B) to node {{}}  (D);
\path[{Triangle[open]}-{Triangle[open]},draw,thick] (C) to node {{}}  (A);

\path[-, draw=red, opacity=0.5, thick, rounded corners]  (0, .5) -- (.5, .5) -- (.5, -.5) -- (-.5, -.5) -- (-.5, .5) -- (.5, .5) -- (0, .5); 

\path[-, draw=darkgreen, opacity=0.5, thick, rounded corners]  (2.1, .5) -- (2.5, .5) -- (2.5, -.5) -- (1.5, -.5) -- (1.5, .5) -- (2.5, .5) -- (2.1, .5); 

\end{tikzpicture}

\end{minipage}
\end{center}
\caption{A lattice with weak pseudocomplementation (left) represented by a pseudosymmetric reflexive frame (right, with reflexive loops assumed but not shown) as in Theorem \ref{NegThm}.\ref{NegThm4}.}\label{NegThmFig}
\end{figure}
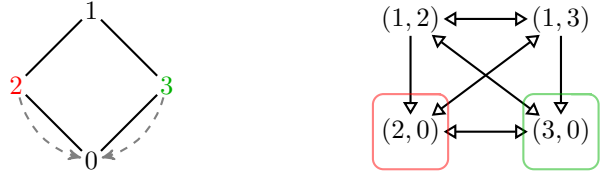
\end{example}

\begin{remark} Less economical choices of $P$ than in Theorem \ref{NegThm} are possible, e.g., setting $P=\{(a,b)\mid a,b\in L, \neg a\leq b\}$ in parts \ref{NegThm1} and \ref{NegThm3} and $P=\{(a,b)\mid a,b\in L, a\not\leq b, \neg a\leq b\}$ in parts \ref{NegThm2} and \ref{NegThm4}, as in \citealt[Thm.~3.19]{Holliday2022}. Note that if we equip $L$ with the weak pseudocomplementation defined by $\neg 0=1$ and $\neg a=0$ for $a\neq 0$, then the latter choice of $P$ reduces to $\{(a,b)\mid a,b\in L, a\not\leq b\}$, which is used as the underlying set of the reflexive frame dual to a complete lattice in \citealt[Thm.~2.11]{Holliday2021}.\end{remark}

Theorem \ref{NegThm} yields five completeness theorems, as two come from part \ref{NegThm4}. Define a \textit{prelogic} in the same way as an intro-elim logic in Definition \ref{BinaryLogic} but dropping part 6 ($\varphi\vdash\neg\neg\varphi$) and part 7 ($\varphi\wedge\neg\varphi\vdash\psi$).\footnote{Note that in this setting, `$\bot$' and `$\top$' are arguably no longer appropriate symbols to abbreviate $p\wedge\neg p$ and $\neg (p\wedge\neg p)$.} Let $\vdash_\mathsf{pre}$ be the weakest prelogic. Define a \textit{protologic} in the same way as an intro-elim logic in Definition \ref{BinaryLogic} but with part 6  replaced by $\psi\vdash \neg(\varphi\wedge\neg\varphi)$. Let $\vdash_\mathsf{pro}$ be the weakest protologic. Define a \textit{paraconsistent intro-elim logic} in the same way as an intro-elim logic in Definition \ref{BinaryLogic} but dropping part 7. Let  $\vdash_\mathsf{para}$ be the weakest paraconsistent intro-elim logic, which can be equivalently defined using our Fitch-style proof system for $\vdash_\mathsf{F}$ but without the $\neg$E rule. Finally, define a \textit{pseudocomplementary logic} in the same way as an intro-elim logic in Definition \ref{BinaryLogic} but with the added principle that if $\varphi\wedge\psi\vdash\bot$, then $\varphi\vdash\neg\psi$. Let $\vdash_\mathsf{psu}$ be the weakest pseudocomplementary logic. 
 
\begin{theorem}\label{MainCompleteness} Let $\mathbb{K}$ be the class of all relational frames, $\mathbb{R}$ the class of reflexive frames, $\mathbb{P}$ (resp.~$\mathbb{S}$)  the class of pseudosymmetric (resp.~strongly pseudosymmetric) frames, $\mathbb{PR}$ (resp.~$\mathbb{SR}$)  the class of pseudosymmetric (resp.~strongly pseudosymmetric) reflexive frames, and $\mathbb{WCR}$ the class of weakly compossible reflexive frames. Then for any formulas $\varphi,\psi\in\mathcal{L}$:
\begin{enumerate}
\item\label{MainCompleteness1} $\varphi\vdash_\mathsf{pre}\psi$ if and only if $\varphi\vDash_\mathbb{K}\psi$;
\item\label{MainCompleteness2} $\varphi\vdash_\mathsf{pro}\psi$ if and only if $\varphi\vDash_\mathbb{R}\psi$;
\item\label{MainCompleteness3} $\varphi\vdash_\mathsf{para}\psi$ if and only if $\varphi\vDash_\mathbb{P}\psi$ (resp.~$\varphi\vDash_\mathbb{S}\psi$);
\item\label{MainCompleteness4} $\varphi\vdash_\mathsf{F}\psi$ if and only if $\varphi\vDash_\mathbb{PR}\psi$ (resp.~$\varphi\vDash_\mathbb{SR}\psi$);
\item\label{MainCompleteness5} $\varphi\vdash_\mathsf{psu}\psi$ if and only if $\varphi\vDash_\mathbb{WCR}\psi$.
\end{enumerate}
\end{theorem}

\begin{proof} Soundness follows from Propositions \ref{FrameToLat} and \ref{CorrLemm}.

For completeness, we first prove parts \ref{MainCompleteness2}, \ref{MainCompleteness4}, and \ref{MainCompleteness5}. The proof is structurally the same in each case. Given $\varphi\nvdash_\mathsf{F}\psi$, where $\theta$ is the valuation on the Lindenbaum-Tarski algebra $(L,\neg)$ of $\vdash_\mathsf{F}$ for which $\tilde{\theta}(\varphi)\not\leq \tilde{\theta}(\psi)$, and $f$ is the embedding of $(L,\neg)$ into $(\lat(P,\comp), \neg_{\comp})$ from Theorem \ref{NegThm}.\ref{NegThm4}, define a valuation $V$ on $\mathfrak{L}(P,\comp)$ by $V(p)= f(\theta(p))$, yielding a model $\mathcal{M}=(P,\comp,V)$. An easy induction shows that for any $\chi\in\mathcal{L}$, $\llbracket \chi\rrbracket^\mathcal{M}= f(\tilde{\theta}(\chi))$. Then from $\tilde{\theta}(\varphi)\not\leq \tilde{\theta}(\psi)$ we have $f(\tilde{\theta}(\varphi))\not\subseteq f(\tilde{\theta}(\psi))$, so $\llbracket\varphi\rrbracket^\mathcal{M}\not\subseteq\llbracket\psi\rrbracket^\mathcal{M}$, so  $\varphi\nvDash_\mathbb{SR}\psi$.

For parts \ref{MainCompleteness1} and \ref{MainCompleteness3}, the Lindenbaum-Tarski algebra of $\vdash_\mathsf{pre}$ (resp.~$\vdash_\mathsf{para}$) is not bounded; but we can embed it into a bounded lattice by adjoining a new minimum $0$ and maximum $1$ to the lattice and setting $\neg 0=1$ and $\neg 1=0$.\footnote{This shows that $\vdash_\mathsf{pre}$ is complete with respect to bounded lattices with precomplementations satisfying $\neg 0 =1$. This depends on the fact that we do not have primitive symbols $\bot$ and $\top$ interpreted as $0$ and $1$ in our language. If we had such symbols in a language $\mathcal{L}_{\bot,\top}$ with corresponding rules $\bot\vdash\varphi$ and $\varphi\vdash\top$ in the definition of $\vdash_{\mathsf{pre}_{\bot,\top}}$, then $\vdash_{\mathsf{pre}_{\bot,\top}}$ would not be complete with respect to lattices with precomplementations satisfying $\neg 0=1$, and the Lindenbaum-Tarski algebra of $\vdash_{\mathsf{pre}_{\bot,\top}}$ would be bounded in the first place.}  Then the rest of the proof is the same as above, using Theorem \ref{NegThm}.\ref{NegThm1} (resp.~\ref{NegThm}.\ref{NegThm3}).\end{proof}

 \noindent Compare part \ref{MainCompleteness2} of Theorem \ref{MainCompleteness} to Theorems 2 and 3 of \citealt{Zhong2021}, which axiomatize the logic of the class $\mathbb{R}$ of reflexive frames in the $\{\neg,\wedge\}$-fragment of~$\mathcal{L}$.

One of the appealing aspects of this relational semantics is how it allows us to apply reasoning that is very familiar from the intuitionistic setting to our non-distributive setting. For example, consider the following proof of the disjunction property for $\vdash_\mathsf{F}$ that takes the disjoint union of two models and adds a new root as in the standard intuitionistic proof. Essentially the same proof applies to the other logics in Theorem \ref{MainCompleteness}.

\begin{proposition}\label{DisjunctionProp} For any $\varphi,\psi\in\mathcal{L}$, if $\top\vdash_\mathsf{F}\varphi\vee\psi$, then $\top\vdash_\mathsf{F}\varphi$ or $\top\vdash_\mathsf{F}\psi$.
\end{proposition}

\begin{proof} Suppose $\top\nvdash_\mathsf{F}\varphi$ and $\top\nvdash_\mathsf{F}\psi$, so by the completeness direction of Theorem \ref{MainCompleteness}.\ref{MainCompleteness4}, there are models $\mathcal{M}_1=(X_1,\comp_1,V_1)$ and $\mathcal{M}_2=(X_2,\comp_2,V_2)$ based on  pseudosymmetric reflexive frames, $x_1\in X_1$, and $x_2\in X_2$ such that $\mathcal{M}_1,x_1\nVdash\varphi$ and $\mathcal{M}_2,x_2\nVdash\psi$. Without loss of generality,  assume $X_1\cap X_2=\varnothing$. Define the disjoint union $\mathcal{M}=(X,\comp,V)$ by $X=X_1\cup X_2$, $\comp \,=\, \comp_1\cup \comp_2$, and $V(p)=V_1(p)\cup V_2(p)$ for $p\in\mathsf{Prop}$. Clearly $(X,\comp)$ is a  pseudosymmetric reflexive frame, $V(p)$ is a $c_\comp$-fixpoint, and  $\mathcal{M},x_1\nVdash\varphi$ and $\mathcal{M},x_2\nVdash\psi$.  

Fixing some $r\not\in X$, define $\mathcal{M}'=(X',\comp',V')$ by $X'=X\cup \{r\}$, $\mathord{\comp'}=\mathord{\comp} \cup \{(x,r) \mid x\in X'\}$, and $V'(p)=V(p)$ for $p\in\mathsf{Prop}$. Then $\comp$ is clearly reflexive. For pseudosymmetry, for $x,y\in X$, suppose $y\comp' x$. If $x\neq r$, then $y\comp x$, so pseudosymmetry of $\comp$ implies there is a $z\comp y$ that pre-refines $x$ with respect to $\comp$. From $z\comp y$ we have $z\comp'y$, and we claim that $z$ pre-refines $x$ with respect to $\comp'$. For suppose $w\comp' z$. Then $w\neq r$, so $w\comp z$, which implies $w\comp x$ since $z$ pre-refines $x$ with respect to $\comp$, so $w\comp ' x$. On the other hand, if $x=r$, then set $z=y$. Hence $z\comp y$, and clearly $z$ pre-refines $r$, since $v\comp r$ for all $v\in X'$. Thus, $\comp'$ is pseudosymmetric. It is also easy to see that $V'(p)$ is a $c_{\comp'}$-fixpoint, so $\mathcal{M}'$ is a model.

Now we claim that for all $\chi\in\mathcal{L}$ and $x\in X$, $\mathcal{M},x\Vdash\chi$ iff $\mathcal{M}',x\Vdash\chi$. The proof is by induction on $\chi$. The base case for $p$ is immediate from the definition of $V'$; the $\wedge$ case is immediate from the inductive hypothesis; and the $\neg$ case and the implication from $\mathcal{M},x\Vdash\chi_1\vee\chi_2$ to  $\mathcal{M}',x\Vdash\chi_1\vee\chi_2$  follow from the inductive hypothesis and the fact that $r\not\comp x$.  Finally, suppose $\mathcal{M}',x\Vdash\chi_1\vee\chi_2$ and $x'\comp x$, so $x'\comp' x$.  Hence there is some $x''\compflip' x'$ such that $\mathcal{M}',x''\Vdash \chi_i$ for some $i\in\{1,2\}$. If $x''\in X$, then $x''\compflip x'$, and by the inductive hypothesis, $\mathcal{M}, x''\Vdash\chi_i$. If $x''=r$, then since $x'$ pre-refines $r$, we have $\mathcal{M},x'\Vdash\chi_i$ by Lemma \ref{prerefinelem}. In either case, we have shown that for all $x'\comp x$ there is a $y\compflip x'$ such that $\mathcal{M},y\Vdash \chi_i$ for some $i\in \{1,2\}$. Thus, $\mathcal{M},x\Vdash \chi_1\vee\chi_2$.

By the previous paragraph, $\mathcal{M}',x_1\nVdash\varphi$ and $\mathcal{M}',x_2\nVdash\psi$. Then since $x_1$ and $x_2$ pre-refine $r$, $\mathcal{M}',r \nVdash \varphi$ and $\mathcal{M}',r \nVdash \psi$ by Lemma \ref{prerefinelem}. Then since $r$ can see a state, namely itself, that cannot be seen by any state forcing $\varphi$ or  $\psi$, we have $\mathcal{M}',r\nVdash\varphi\vee\psi$. Hence $\top\nvdash_\mathsf{F}\varphi\vee\psi$ by the soundness part of Theorem \ref{MainCompleteness}.\ref{MainCompleteness4}.
\end{proof}

We  conclude this section by briefly following up on the idea from Remarks \ref{NCremark} and \ref{MoreGeneralNeg} of representing lattices with negations that do not necessarily satisfy $\neg 1 =0$. We prove an analogue of Theorem \ref{NegThm}.\ref{NegThm1} for such negations; analogues of the other parts of Theorem \ref{NegThm} can be similarly obtained.

\begin{theorem}\label{NegThmAntitone} Let $L$ be a bounded lattice, $\Lambda$ a meet-dense set of elements of $L$, and $\neg$ an antitone unary operation on $L$. Define $P =  \{(a,\neg a)\mid a\in L\} \cup \{(1,b)\mid b\in \Lambda\}$, ${(a, b)\comp (c, d)}$  if  $c\not\leq b$, and $F =  c_\comp(\{(\neg 1, \neg\neg 1)\})$. Then there is a complete embedding of $(L,\neg)$ into $(\lat(P,\comp),\neg_{\comp,F})$ with $\neg_{\comp,F}$  defined as in Remark \ref{MoreGeneralNeg}, which is an isomorphism if $L$ is complete.\end{theorem}

\begin{proof} The proof that the map $f$ in Proposition \ref{CompRep} is a complete lattice embedding of $L$ into $\lat(P,\comp)$, which is an isomorphism if $L$ is complete, is exactly as in the proof of Theorem \ref{NegThm}.\ref{NegThm1}. It only remains to verify that $ f(\neg a)=\neg_{\comp,F} f(a)$. 

Suppose $(x,y)\in  f(\neg a)$, so $x\leq\neg a$. Further suppose $(x',y')\comp (x,y)$ and $(x',y')\in f(a)$, so $x'\leq a$. Then $\neg a\leq\neg x'$, which with $x\leq\neg a$ implies $x\leq \neg x'$. Now if $(x',y')\in  \{(a,\neg a)\mid a\in L\}$, then from $x\leq \neg x'$ we have $x\leq y'$, contradicting $(x',y')\comp (x,y)$. Thus, we have $(x',y')\in \{(1,b)\mid b\in\Lambda\}$. Then from $x\leq\neg x'$, we have $x\leq \neg 1$, in which case we claim $(x,y)\in F$. For if $(x^*,y^*)\comp (x,y)$, so $x\not\leq y^*$, then $\neg 1\not\leq y^*$, so $(x^*,y^*)\comp (\neg 1,\neg\neg 1)$, which shows $(x,y)\in c_\comp(\{(\neg 1,\neg\neg 1)\}$. Thus, for  all $(x',y')\comp (x,y)$, if $(x',y')\in f(a)$, then $(x,y)\in F$. It follows that $(x,y)\in \neg_{\comp,F}  f(a)$.  Conversely, let $(x,y)\in P\setminus f(\neg a)$, so $x\not\leq \neg a$. Then $(a,\neg a)\comp (x,y)$. Moreover, since $a\leq 1$, we have $\neg 1\leq \neg a$, so $(a,\neg a)\not\comp (\neg 1,\neg\neg 1)$, which implies there is no $(z,w)\compflip (a,\neg a)$ with $(z,w)\in F$. It follows that $(x,y)\not\in \neg_{\comp,F} f(a)$. \end{proof}

\subsection{Topological representation of lattices with negations}\label{TopRep}

Topological representations of bounded lattices using reflexive frames endowed with a topology were developed in \citealt{Ploscica1995} and \citealt{Craig2013}, building on \citealt{Urquhart1978} and \citealt{Allwein1993}. In \citealt{Holliday2022}, we considered a variant of the approach of \citealt{Craig2013} using disjoint filter-ideal pairs but with a different topology in the spirit of the choice-free Stone duality of \citealt{BH2020}. In this section, we briefly show how the filter-ideal representation can be adapted to bounded lattices equipped with protocomplementations and hence in particular weak pseudocomplementations. For topological representations of ortholattices in particular, using symmetric and reflexive frames of proper filters equipped with a topology, see \citealt{Goldblatt1975} and \citealt{McDonald2021}, and for associated categorical dualities, see \citealt{Bimbo2007}, \citealt{Dmitrieva2021}, and \citealt{McDonald2021}.

Given a bounded  lattice $L$ and a protocomplementation $\neg$, define $\mathsf{FI}(L,\neg)=(X,\comp)$ as follows: $X$ is the set of all pairs $(F,I)$ such that $F$ is a  filter in $L$, $I$ is a ideal in $L$, $F\cap I=\varnothing$, and ${\{\neg a\mid a\in F\}\subseteq I}$. One can interpret the states in $X$ intuitively as in Remark \ref{IntuitivePicture}: the state $(F,I)$ accepts the propositions in $F$ and rejects the propositions in $I$. Then define  $(F,I)\comp (F',I')$ iff $I\cap F'=\varnothing$. Note that since $\neg$ is a protocomplementation, $\comp$ is reflexive; but if we are interested in negations that are not semicomplementations, we can drop the condition that $F\cap I=\varnothing$ (see the end of Appendix \ref{AppendixB} and compare the odd vs.~even parts of Theorem~\ref{NegThm}). Given $a\in L$, let $\widehat{a}=\{(F,I)\in X\mid a\in F\}$. Finally, let $\mathsf{S}(L)$ be $\mathsf{FI}(L,\neg )$ endowed with the topology generated by $\{\widehat{a}\mid a\in L\}$.

\begin{theorem}\label{EmbedThm} For any bounded lattice $L$ and protocomplementation $\neg$ on $L$, the map $a\mapsto\widehat{a}$ is 
\begin{enumerate}
\item\label{EmbedThm1} an embedding  of $(L,\neg)$ into $(\lat(\mathsf{FI}(L,\neg)),\neg_\comp)$ and 
\item\label{EmbedThm2} an isomorphism from $L$ to the subalgebra of $(\lat(\mathsf{FI}(L,\neg )),\neg_\comp)$ consisting of $c_\comp$-fixpoints that are compact open in the space $\mathsf{S}(L)$.
\end{enumerate}
\end{theorem}

\begin{proof} Given $a\in L$, let $\mathord{\uparrow}a$ and $\mathord{\downarrow}a$ be the filter and ideal, respectively, generated by $a$. 

First observe that for any $a\in L$,  $\widehat{a}$ is a $c_\comp$-fixpoint. It suffices to show that if $(F,I)\not\in \widehat{a}$, then there is an $(F',I')\comp (F,I)$ such that for all $(F'',I'')\compflip (F',F')$, we have $(F'',I'')\not\in \widehat{a}$. Suppose $(F,I)\not\in\widehat{a}$, so $a\not\in F$ and hence $a\neq 1$. Let $F'=\mathord{\uparrow}1$ and $I'=\mathord{\downarrow}a$. Then $(F',I')\in X$. Now consider any $(F'',I'')$ such that $(F',I')\comp (F'',I'')$, so $I'\cap F''=\varnothing$. Then since $a\in I'$, we have $a\not\in F''$, so $(F'',I'')\not\in \widehat{a}$, as desired.

Next, the map $a\mapsto\widehat{a}$ is clearly injective:  if $a\not\leq b$, then $(\mathord{\uparrow}a, \mathord{\downarrow}\neg a)\in X$, $(\mathord{\uparrow}a, \mathord{\downarrow}\neg a)\in \widehat{a}$, and $(\mathord{\uparrow}a, \mathord{\downarrow}\neg a)\not\in\widehat{b}$. Obviously $\widehat{1}=X$ and $\widehat{0}=\varnothing$. The map also preserves $\wedge$:  $\widehat{a\wedge b}=\{(F,I)\in X\mid a\wedge b\in F\}=\{(F,I)\in X\mid a,b\in F\}=\{(F,I)\in X\mid a\in F\}\cap \{(F,I)\in X\mid b\in F\}=\widehat{a}\cap \widehat{b}=\widehat{a}\wedge\widehat{b}$. 

Next we show $\widehat{a\vee b}\subseteq\widehat{a}\vee\widehat{b}$, as the converse inclusion follows from meet preservation. Recall from Proposition \ref{ClosureLattice} that ${\widehat{a}\vee\widehat{b}}=c_\comp(\widehat{a}\cup\widehat{b})$. Suppose $(F,I)\in \widehat{a\vee b}$, so $a\vee b\in F$. Consider any ${(F',I')\comp (F,I)}$, so $I'\cap F=\varnothing$ and hence $a\vee b\not\in I'$. Then since $I'$ is an ideal,  $a\not\in I'$ or ${b\not\in I'}$. Without loss of generality, suppose $a\not \in I'$, so $a\neq 0$. Then setting $F''=\mathord{\uparrow}a$ and $I''=\mathord{\downarrow}\neg a$, we have $(F'',I'')\in X$ and $I'\cap F''=\varnothing$, so ${(F',I')\comp (F'',I'')}$, and $(F'',I'')\in\widehat{a}$. Thus, we have shown that for any $(F',I')\comp (F,I)$ there is an $(F'',I'')\compflip (F',I')$ with $(F'',I'')\in\widehat{a}\cup \widehat{b}$. Hence $(F,I)\in\widehat{a}\vee\widehat{b}$.  Finally, we show that $\widehat{\neg a}=\neg_\comp\widehat{a}$. First suppose $(F,I)\in \widehat{\neg a}$ and $(F',I')\comp (F,I)$. Since $(F,I)\in \widehat{\neg a}$, we have $\neg a\in F$, which with $(F',I')\comp (F,I)$ implies $\neg a\not\in I'$, which with the definition of $X$ implies $a\not\in F'$, so $(F',I')\not\in \widehat{a}$. Hence $(F,I)\in \neg_\comp\widehat{a}$. Conversely, if $(F,I)\not\in \widehat{\neg a}$, so $\neg a\not\in F$, then $(\mathord{\uparrow}a,\mathord{\downarrow}\neg a)\comp (F,I)$ and $(\mathord{\uparrow}a,\mathord{\downarrow}\neg a)\in\widehat{a}$, so $(F,I)\not\in \neg_\comp \widehat{a}$.

For part \ref{EmbedThm2}, we first show that $\widehat{a}$ is compact open. Since the $\widehat{b}$'s form a basis, we need only show that if $\widehat{a}\subseteq \bigcup \{\widehat{b}_k\mid k\in K\}$, then there is a finite subcover. Indeed, since $(\mathord{\uparrow}a,\mathord{\downarrow}\neg a)\in \widehat{a}$, we have $(\mathord{\uparrow}a,\mathord{\downarrow}\neg a)\in \widehat{b_k}$ for some $k\in K$, which implies $a\leq b_k$, so $\widehat{a}\subseteq\widehat{b_k}$. Finally, we show that $a\mapsto\widehat{a}$  is onto the set of compact open $c_\comp$-fixpoints. Suppose $U$ is compact open, so $U=\widehat{a_1}\cup\dots\cup\widehat{a_n}$ for some $a_1,\dots,a_n\in L$. Further suppose $U$ is a $c_\comp$-fixpoint, so $c_\comp(U)=U$. Where $d=a_1\vee\dots\vee a_n$, an obvious induction using part \ref{EmbedThm1} and the fact that $c_\comp(c_\comp(A)\cup B)=c_\comp(A\cup B)$ for any $A,B\subseteq X$ yields $\widehat{d}= c_\comp(\widehat{a_1}\cup\dots\cup\widehat{a_n})$, so $\widehat{d}=c_\comp (U)= U$.
\end{proof}
\noindent In Appendix \ref{AppendixB}, we prove an analogue of Theorem \ref{EmbedThm} for bounded lattices with implications.

\begin{remark} The difference between the embedding part of Theorem \ref{NegThm} and the embedding part of Theorem \ref{EmbedThm} is that in the former  we are embedding $(L,\neg)$ into its \textit{MacNeille completion} (see \citealt[Thm.~2.2]{Gehrke2005}) whereas in the latter  we are embedding $(L,\neg)$ into its \textit{canonical extension} (see \citealt{Gehrke2001}, \citealt{Craig2014}).
\end{remark}

Finally, consider the case where $\neg$ is a weak pseudocomplementation in line with our logic $\vdash_\mathsf{F}$.

\begin{proposition}\label{FIprop} If $\neg$ is a weak pseudocomplementation on $L$, then $\comp$ in $\mathsf{FI}(L,\neg)$ is strongly pseudosymmetric.
\end{proposition}
\begin{proof} Suppose $(F',I')\comp (F,I)$. Where $I''$ is the ideal generated by $\{\neg a\mid a\in F\}$, we claim that $F\cap I''=\varnothing$. Otherwise there are $a_1,\dots,a_n,b\in F$ such that $b\leq \neg a_1\vee\dots\vee \neg a_n$. Then where $a=a_1\wedge\dots\wedge a_n$, we have $a\in F$ and $b\leq \neg a$, so $\neg a\in F$, which implies  $a\wedge\neg a\in F$ and hence $0\in F$, contradicting the fact that $F$ is a proper filter. Hence $(F, I'')\in X$. Now we claim that $(F, I'')\comp (F',I')$. For otherwise there is some $b\in F'$ and $a_1,\dots,a_n\in F$ such that $b\leq \neg a_1\vee\dots\vee \neg a_n$, so where $a=a_1\wedge\dots\wedge a_n$, we have $a\in F$ and $b\leq \neg a$, so $\neg a\in F'$ and hence $\neg\neg a\in I'$, which implies $a\in I'$, which contradicts $(F',I')\comp (F,I)$.  Finally, since $(F,I'')$ and $(F,I)$ have the same first coordinate, $(F, I'')$ pre-refines $(F,I)$ and vice versa.\end{proof}

Thus, by analogy with modal logic, we may say that our propositional logic $\vdash_\mathsf{F}$ is \textit{canonical} in the sense that it is validated by its canonical frame, whether one considers that to be the relational frame built from the Lindenbaum-Tarski algebra of the logic by  Theorem \ref{EmbedThm} or by Theorem \ref{NegThm}.\ref{NegThm4}. 

\subsection{Modal translations}\label{ModalTranslations}

Relational semantics for non-classical propositional logics immediately raise the possibility of translating such logics into modal logics on a classical base, as in  G\"{o}del's  translation of intuitionistic logic into the normal modal logic \textbf{S4} (\citealt{Godel1933b}, \citealt{McKinsey1948}), the modal logic of reflexive and transitive frames. In a similar spirit, Goldblatt \citeyearpar{Goldblatt1974} gave a full and faithful embedding of orthologic into the normal modal logic \textbf{KTB}, the modal logic of reflexive and symmetric frames. Below we will give a full and faithful embedding of our logic $\vdash_\mathsf{F}$ into the extension of the minimal temporal logic $\mathbf{K}_t$ (\citealt[Def.~4.33]{Blackburn2001}) with the reflexivity axiom $Hq\to q$ and the pseudosymmetry axiom $Hq\to HPHq$ (or $FHq \to PHq$), based on viewing $\comp$ in our frames as the temporal relation. We call this logic $\mathbf{K}_t\mathbf{TP}$.  The pseudosymmetry axiom $\mathbf{P}$ is Sahlqvist and hence canonical (\citealt[Thm.~4.42]{Blackburn2001}), so $\mathbf{K}_t\mathbf{TP}$ is complete with respect to the class of pseudosymmetric reflexive frames.  In fact, the canonical frame for $\mathbf{K}_t\mathbf{TP}$ (\citealt[Def.~4.34]{Blackburn2001}) is strongly pseudosymmetric. For where $\Gamma$ and $\Sigma$ are maximally consistent sets and $R$ is the canonical relation, we claim that if $\Gamma R \Sigma$, then \[\Delta_0=\{\varphi\mid H\varphi \in\Gamma\}\cup \{H\psi \mid H\psi\in \Sigma\}\] is consistent.  If not, then for  $H\varphi_1,\dots,H\varphi_n\in \Gamma$ and $H\psi_1,\dots,H\psi_m\in \Sigma$, we have \[\varphi_1\wedge\dots\wedge\varphi_n\vdash \neg (H\psi_1\wedge\dots\wedge H\psi_m)\vdash \neg H\chi\] where $\chi=\psi_1\wedge\dots\wedge\psi_m$, which implies  $H\varphi_1\wedge\dots\wedge H\varphi_n\vdash H\neg H\chi$, so $H\neg H\chi \in \Gamma$. But  $H\chi\in \Sigma$, so we have $HPH\chi\in\Sigma$ by the $\mathbf{P}$ axiom, which with $\Gamma R \Sigma$ implies $PH\chi\in\Gamma$, contradicting $H\neg H\chi \in \Gamma$. Extending $\Delta_0$ to a maximally consistent set provides the desired witness for strong pseudosymmetry, as $\Delta R\Gamma$ and $\Delta$ and $\Sigma$ have the same temporal predecessors. 
 
 The translation $t$ from our language $\mathcal{L}$ to the temporal language is given by: \[\mbox{$t(p)=HF p$, $t(\neg\varphi)=H\neg t(\varphi)$, $t(\varphi\wedge\psi)=(t(\varphi)\wedge t(\psi))$, and  $t(\varphi\vee\psi)= HF(t(\varphi)\vee t(\psi))$.}\]
 Then the following is easy to prove using completeness for both logics (where $\alpha\vdash_{\mathbf{K}_t\mathbf{TP}}\beta$ means that $\alpha\to \beta$ is a theorem of $\mathbf{K}_t\mathbf{TP}$), transferring countermodels on one side to countermodels on the other side.
 
 \begin{proposition} For all $\varphi,\psi\in\mathcal{L}$, we have $\varphi\vdash_\mathsf{F}\psi$ iff $t(\varphi)\vdash_{\mathbf{K}_t\mathbf{TP}}t(\psi)$.
 \end{proposition}
 \noindent Similarly, the other logics in Theorem \ref{MainCompleteness} embed via $t$ into corresponding temporal logics; e.g., $\vdash_\mathsf{pre}$ embeds into $\mathbf{K}_t$, so we obtain the decidability of the former from the known decidability of the latter.

A referee asked whether if we restrict attention to the $\{\wedge,\neg\}$-fragment of $\mathcal{L}$, denoted $\mathcal{L}_{\wedge,\neg}$, then we obtain a full and faithful embedding of $\vdash_\mathsf{F}$ into $\mathbf{KTB}$ by modifying Goldblatt's \citeyearpar{Goldblatt1974} modal translation as follows: \[\mbox{$m(p)=p$ (instead of $\Box\Diamond p$), $m(\neg\varphi)=\Box\neg m(\varphi)$, and $m(\varphi\wedge\psi)=(m(\varphi)\wedge m(\psi))$.}\] 
Recall that $\mathbf{KTB}$ is the smallest normal modal logic containing the axioms $\Box p\to p$ and $p\to\Box\Diamond p$, and let $\alpha\vdash_\mathbf{KTB}\beta$ mean that $\alpha\to \beta$ is a theorem of $\mathbf{KTB}$. Under the $m$ translation,  $p\vdash_\mathsf{F}\neg\neg p$ corresponds to  $p\vdash_\mathbf{KTB}\Box\Diamond p$, while $\neg\neg p\nvdash_\mathsf{F} p$ corresponds to $\Box\Diamond p\nvdash_\mathbf{KTB}p$. More generally, we prove the following.

\begin{proposition} For all $\varphi,\psi\in\mathcal{L}_{\wedge,\neg}$, we have $\varphi\vdash_\mathsf{F}\psi$ iff $m(\varphi)\vdash_\mathbf{KTB}m(\psi)$.
\end{proposition}
\begin{proof} Let an \textit{intro-elim logic for $\mathcal{L}_{\wedge,\neg}$} be defined as in Definition \ref{BinaryLogic} but without the conditions involving $\vee$. It is easy to check that the relation $\vdash$ defined on $\mathcal{L}_{\wedge,\neg}$ by $\varphi\vdash\psi$ iff $m(\varphi)\vdash_\mathbf{KTB}m(\psi)$ is an intro-elim logic for $\mathcal{L}_{\wedge,\neg}$. Now where  $\vdash_\mathsf{F}^{\wedge,\neg}$ is the smallest intro-elim logic for $\mathcal{L}_{\wedge,\neg}$, we claim that  $\varphi\vdash_\mathsf{F}\psi$ implies $\varphi\vdash_\mathsf{F}^{\wedge,\neg}\psi$ for $\varphi,\psi\in\mathcal{L}_{\wedge,\neg}$. For if $\varphi\nvdash_\mathsf{F}^{\wedge,\neg}\psi$, then the Lindenbaum-Tarski algebra of $\vdash_\mathsf{F}^{\wedge,\neg}$ is a meet semilattice with $0$ and $1$ equipped with a weak pseudocomplementation, denoted $(M,\neg)$, that refutes the entailment from $\varphi$ to $\psi$. Now the proof of Theorem \ref{NegThm}.\ref{NegThm4}, replacing $\mathrm{V}$ with $L$, works for meet semilattices with $0$ and $1$  equipped with a weak pseudocomplementation, delivering a $(\wedge,\neg)$-embedding of $(M,\neg)$ into  a complete lattice with weak pseudocomplementation, $(\mathfrak{L}(X,\comp),\neg_\comp)$, that also refutes the entailment from $\varphi$ to $\psi$. Hence $\varphi\nvdash_\mathsf{F}\psi$ by Proposition \ref{SoundProp}. Thus, $\varphi\vdash_\mathsf{F}\psi$  implies $\varphi\vdash_\mathsf{F}^{\wedge,\neg}\psi$ and therefore $m(\varphi)\vdash_\mathbf{KTB}m(\psi)$.

 Conversely, if $\varphi\nvdash_\mathsf{F}\psi$, then by Theorem~\ref{MainCompleteness}.\ref{MainCompleteness4}, there is a model $\mathcal{M}=(X,\comp, V)$ based on a pseudosymmetric reflexive frame and $w\in X$ such that $\mathcal{M},w\Vdash\varphi$ and $\mathcal{M},w\nVdash\psi$. Let $\mathcal{M}^s=(X,\comp^s,V)$ be the model for the unimodal language where $\comp^s$ is the symmetric closure of $\comp$. Although $V(p)$ may not be a $c_{\comp^s}$-fixpoint,  this is not required for a modal model. Now we prove by induction on the structure of formulas $\varphi\in\mathcal{L}_{\wedge,\neg}$ that for all $x\in X$, $\mathcal{M},x\Vdash \varphi$ iff $\mathcal{M}^s,x\vDash m(\varphi)$, where $\vDash$ is the usual modal satisfaction relation with $\compflip$ as the accessibility relation for $\Box$. The base case and $\wedge$ case are obvious. For the $\neg$ case, if $\mathcal{M},x\nVdash \neg\varphi$, then there is a $y\comp x$ with $\mathcal{M},y\Vdash\varphi$, which implies $y\comp^s x$ and $\mathcal{M}^s,y\vDash m(\varphi)$ by the inductive hypothesis, so $\mathcal{M}^s,x\nvDash \Box\neg m(\varphi)$. Conversely, suppose $\mathcal{M}^s,x\nvDash \Box\neg m(\varphi)$, so there is some $y\comp^s x$ with $\mathcal{M}^s,y\vDash m(\varphi)$ and hence $\mathcal{M},y\Vdash \varphi$ by the inductive hypothesis. Given $y\comp^s x$, we have either $y\comp x$ or $x\comp y$. If $y\comp x$, then $\mathcal{M},x\nVdash \neg\varphi$. If $x\comp y$, then by pseudosymmetry, there is a $z\comp x$ that pre-refines $y$. Then from $\mathcal{M},y\Vdash \varphi$ we obtain $\mathcal{M},z\Vdash \varphi$ by Lemma \ref{prerefinelem}, so again $\mathcal{M},x\nVdash \neg\varphi$. Thus, we conclude $\mathcal{M}^s,w\vDash m(\varphi)$ and $\mathcal{M}^s,w\nvDash m(\psi)$, so $m(\varphi)\nvdash_\mathbf{KTB}m(\psi)$ by the soundness of $\mathbf{KTB}$ with respect to reflexive and symmetric frames.\end{proof}
 
 Note that if we compose the $m$ translation above with the $g$ translation from orthologic to $\vdash_\mathsf{F}$ in \S~\ref{FitchSection}, then we obtain Goldblatt's translation of orthologic into $\mathbf{KTB}$.

\section{Quantification}\label{QuantSection}

In this section, we extend the logic $\vdash_\mathsf{F}$ with rules for the universal and existential quantifiers. For simplicitly, we consider a first-order language $\mathcal{LQ}$ with no function symbols, no constants, and no identity symbol. Atomic formulas are of the form $P(v_1,\dots,v_n)$ where $P$ is an $n$-ary predicate and $v_1,\dots,v_n$ belong to a countably infinite set $\mathsf{Var}$ of variables. Thus, formulas are given by the grammar
\[\varphi::= P(v_1,\dots,v_n)\mid \neg\varphi\mid (\varphi\wedge\varphi)\mid (\varphi\vee\varphi)\mid \forall v\varphi \mid \exists v\varphi\]
where $v_1,\dots,v_n,v\in\mathsf{Var}$. We assume familiarity with the notions of free variables and of one variable being substitutable for another in $\varphi$ (see, e.g., \citealt[p.~113]{Enderton2001}); $\varphi^v_u$ is the result of substituting $u$ for $v$ in $\varphi$. 

We define proofs for $\vdash_\mathsf{FQ}$, \textit{fundamental first-order logic}, as for $\vdash_\mathsf{F}$ in \S~\ref{FitchSection} but with the following additional clauses, represented diagrammatically in Figure \ref{QuantRules}, where $1\leq i\leq n$:
\begin{itemize}
\item If $\langle \sigma_1,\dots,\sigma_n\rangle$ is a proof, $\sigma_i$ is a formula $\varphi$, and $v$ does not occur free in $\sigma_1$, then $\langle \sigma_1,\dots,\sigma_n,\forall v\varphi\rangle$ is a proof ($\forall$I).
\item If $\langle \sigma_1,\dots,\sigma_n\rangle$ is a proof, $\sigma_i$ is a formula of the form $\forall v\varphi$, and $u$ is substitutable for $v$ in $\varphi$, then $\langle \sigma_1,\dots,\sigma_n,\varphi^v_u\rangle$ is a proof ($\forall$E).
\item If $\langle \sigma_1,\dots,\sigma_n\rangle$ is a proof, $\sigma_i$ is a formula of the form $\varphi^v_u$, and $u$ is substitutable for $v$ in $\varphi$, then $\langle \sigma_1,\dots,\sigma_n,\exists v\varphi\rangle$ is a proof ($\exists$I). 
\item If $\langle \sigma_1,\dots,\sigma_n\rangle$ is a proof, $\sigma_i$ is a formula of the form $\exists v\varphi$, $\sigma_n$ is a proof beginning with $\varphi$ and ending with $\psi$, and $v$ does not occur free in $\psi$, then $\langle \sigma_1,\dots,\sigma_n,\psi\rangle$ is a proof  ($\exists$E).
\end{itemize}
As in the propositional case, by adding RAA we obtain first-order orthologic; by adding Reiteration\footnote{When defining a proof \textit{given a set $R$ of reiterables} as in Appendix \ref{AppendixA}, $\forall$I states that if  $\langle \sigma_1,\dots,\sigma_n\rangle$ is a proof given $R$, $\sigma_i$ is a formula $\varphi$, and $v$ does not occur free in $\sigma_1$ or in any formula in $R$, then $\langle \sigma_1,\dots,\sigma_n,\forall v\varphi\rangle$ is a proof given $R$.} we obtain intuitionistic first-order logic; and by adding both we obtain classical first-order logic. Moreover, the negative translation from orthologic to $\vdash_\mathsf{F}$ in \S~\ref{FitchSection} also extends to a translation from first-order orthologic to $\vdash_\mathsf{FQ}$ by setting $g(\forall v\varphi)=\forall vg(\varphi)$ and $g(\exists v\varphi)=\neg\forall v\neg g(\varphi)$.

\begin{figure}[h]
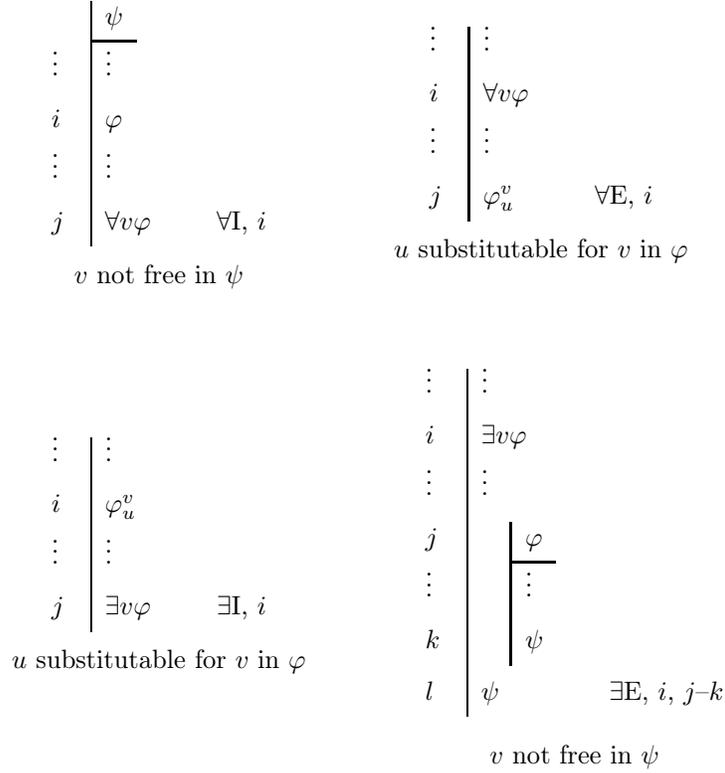

\begin{center}
\begin{minipage}{2in}
\[\begin{nd}
\hypo [\,] {} {\psi}
\have [\vdots] {0} {\vdots}
\have [i] {3}   {\varphi}
\have [\vdots] {4}   {\vdots}
\have [j]{5} {\forall v\varphi} \Ai{3}
\end{nd}
\]\[v\mbox{ not free in }\psi\]\end{minipage}\begin{minipage}{2in}
\[\begin{nd}
\have [\vdots] {0} {\vdots}
\have [i] {3}   {\forall v\varphi}
\have [\vdots] {4}   {\vdots}
\have [j]{5} {\varphi^v_u} \Ae{3}
\end{nd}
\]\[u\mbox{ substitutable for }v\mbox{ in }\varphi\]\end{minipage}\vspace{.2in}

\begin{minipage}{2.35in}
\[\begin{nd}
\have [\vdots] {0} {\vdots}
\have [i] {3}   {\varphi^v_u}
\have [\vdots] {4}   {\vdots}
\have [j]{5} {\exists v\varphi} \Ei{3}
\end{nd}
\]\[u\mbox{ substitutable for }v\mbox{ in }\varphi\]\end{minipage}\begin{minipage}{2in}
\[\begin{nd}
\have [\vdots] {0} {\vdots}
\have [i] {1} {\exists v\varphi}
\have [\vdots] {2} {\vdots}
\open
\hypo [j] {3}   {\varphi}
\have [\vdots] {4}   {\vdots}
\have [k] {6}   {\psi}
\close
\have [l]{7} {\psi} \Ee{1,3-6}
\end{nd}
\]\[v\mbox{ not free in }\psi\]\end{minipage}
\end{center}
\caption{Fitch-style rules for the logic with quantifiers.}\label{QuantRules}
\end{figure}

 By the kind of sequent calculus analysis mentioned at the end of \S~\ref{FitchSection}, Aguilera and Byd\u{z}ovsk\'y \citeyearpar{Aguilera2022} have shown that in striking contrast to intuitionistic or classical first-order logic, fundamental first-order logic is  decidable. Thus, just the addition of Reiteration takes us from decidability to undecidability.

\begin{theorem}[Aguilera and Byd\u{z}ovsk\'y] It is decidable in double exponential time whether $\varphi\vdash_\mathsf{FQ}\psi$.
\end{theorem}

Turning to semantics, relational frames for $\mathcal{LQ}$ are triples $(X,\comp, D)$ where $(X,\comp)$ is a relational frame and $D$ is a nonempty set disjoint from $X$. A relational model $(X,\comp, D,V)$ adds a function  $V$ assigning to each $n$-ary predicate $P$ and $n$-tuple of objects $d_1,\dots,d_n$ from $D$ a $c_\comp$-fixpoint $V(P,d_1,\dots,d_n)\subseteq X$. Given $v\in\mathsf{Var}$ and variable assignments $g,h\in D^\mathsf{Var}$, let $h\sim_v g$ mean that $h$ and $g$ differ at most at $v$. Then the forcing clauses are:
\begin{itemize}
\item $\mathcal{M},x\Vdash_g P(v_1,\dots,v_n)$ iff $x\in V(P,g(v_1),\dots,g(v_n))$;
\item clauses for $\neg$, $\wedge$, and $\vee$ as before;
\item $\mathcal{M},x\Vdash_g \forall v\varphi$ iff $\forall h\sim_v g$, $\mathcal{M},x\Vdash_h \varphi$;
\item $\mathcal{M},x\Vdash_g \exists v\varphi$ iff $\forall x'\comp x$ $\exists x''\compflip x'$ $\exists h\sim_v g$: $\mathcal{M},x''\Vdash_h\varphi$.
\end{itemize}
Where $\llbracket \varphi\rrbracket^\mathcal{M}_g=\{x\in X\mid \mathcal{M},x\Vdash_g\varphi\}$, an easy induction shows that $\llbracket \varphi\rrbracket^\mathcal{M}_g$ is a $c_\comp$-fixpoint, and 
\begin{eqnarray*}
\llbracket \forall v \varphi\rrbracket^\mathcal{M}_g &=& \bigwedge \{\llbracket\varphi\rrbracket_h^\mathcal{M}\mid h\sim_v g\}\\
\llbracket \exists v\varphi\rrbracket^\mathcal{M}_g &=& \bigvee \{\llbracket\varphi\rrbracket_h^\mathcal{M}\mid h\sim_v g\}.
\end{eqnarray*}
Given a class $\mathbb{C}$ of relational frames for $\mathcal{LQ}$, we define $\varphi\vDash_\mathbb{C}\psi$ if for all $(X,\comp,D)\in \mathbb{C}$, all models $\mathcal{M}={(X,\comp,D, V)}$ based on $(X,\comp, D)$, and all variable assignments $g\in D^\mathsf{Var}$, if $\mathcal{M},x\Vdash_g\varphi$, then $\mathcal{M},x\Vdash_g\psi$.

Let $\mathbb{PRQ}$ be the class of pseudosymmetric reflexive frames for $\mathcal{LQ}$. We can use Theorem \ref{NegThm}.\ref{NegThm4} to prove completeness of $\vdash_\mathsf{FQ}$ with respect to $\mathbb{PRQ}$. The Lindenbaum-Tarski algebra of $\vdash_\mathsf{FQ}$ is defined as usual.

\begin{lemma}\label{QuantLem} In the Lindenbaum-Tarski algebra of $\vdash_\mathsf{FQ}$, for all $\varphi\in\mathcal{L}$ and $v\in\mathsf{Var}$:
\begin{eqnarray*}
[\forall v\varphi]&=& \bigwedge \{ [\varphi^v_u] \mid u\in\mathsf{Var}\mbox{ and substitutable for }v\mbox{ in }\varphi \}\\
{[}\exists v\varphi{]} &=& \bigvee \{ [\varphi^v_u] \mid u\in\mathsf{Var}\mbox{ and substitutable for }v\mbox{ in }\varphi \}.
\end{eqnarray*}
\end{lemma}
\begin{proof} A standard exercise using the introduction and elimination rules for the quantifiers.
\end{proof}

\begin{theorem}\label{QuantComplete} For all formulas $\varphi,\psi\in\mathcal{LQ}$, we have $\varphi\vdash_\mathsf{FQ}\psi$ if and only if $\varphi\vDash_\mathbb{PRQ}\psi$.
\end{theorem}
\begin{proof} Soundness is straightforward (cf.~Proposition \ref{SoundProp}). For completeness, suppose that $\varphi\nvdash_\mathsf{FQ}\psi$, so in the Lindenbaum-Tarski algebra $(L,\neg)$ for $\vdash_\mathsf{FQ}$, we have $[\varphi]\not\leq[\psi]$. By Theorem \ref{NegThm}.\ref{NegThm4}, there is a complete embedding $f$ of $(L,\neg)$ into $(\lat(X,\comp),\neg_\comp)$ for a pseudosymmetric reflexive frame $(X,\comp)$. We turn $(X,\comp)$ into a model $\mathcal{M}=(X,\comp,D,V)$ for $\mathcal{LQ}$ by setting $D=\mathsf{Var}$ and $V(P,v_1,\dots,v_n)=f([P(v_1,\dots,v_n)])$. Let the variable assignment $g$ be the identity function on $\mathsf{Var}$. Given Lemma \ref{QuantLem} and the fact that $f$ is a \textit{complete} embedding, it is easy to show that for all formulas $\varphi\in\mathcal{LQ}$, $\llbracket \varphi\rrbracket^\mathcal{M}_g= f ([\varphi])$. Then from $[\varphi]\not\leq[\psi]$, we have $f([\varphi])\not\leq f([\psi])$,  so $\llbracket \varphi\rrbracket^\mathcal{M}_g \not\subseteq \llbracket \psi\rrbracket^\mathcal{M}_g$ and hence $\varphi\nvDash_\mathbb{PRQ}\psi$.
\end{proof}
Clearly the same strategy also works for quantified versions of other logics we have discussed.

\section{Comments on conditionals}\label{Conditionals}

So far we have said nothing about ``the'' conditional. But there are many kinds of conditionals, especially when moving out of the classical or intuitionistic world and into the orthological world or beyond. Indeed, there are at least three paths we could pursue when adding a conditional to our language: add the traditional introduction and elimination rules for $\to$ to $\vdash_\mathsf{F}$; add rules meant to capture properties of the indicative conditional `if\dots then' of natural language, which might differ from the traditional rules for $\to$; or consider how the relational semantics of \S~\ref{RelationalSection} might be extended to treat conditionals. In this section, we consider these three paths in roughly reverse order.  We will mention options without making definitive choices.

Semantically, where $\Phi(y,A,B)$ is a condition on a state $y$ and subsets $A,B$ of a frame $(X,\comp)$ such that $y$ is the only free state variable in $\Phi(y,A,B)$, the set
\[A\to_\Phi B =\{x\in X\mid \forall y\comp x\,(y\in A\Rightarrow\Phi(y,A,B))\}\]
is a $c_\comp$-fixpoint and hence a candidate for a kind of conditional proposition. Examples of $\Phi(y,A,B)$ include:
\begin{enumerate}
\item $y\in B$;
\item $\exists z\comp y$: $z\in B$;
\item $\exists z\compflip y$: $z\in B$;
\item $\exists z\comp y$: $z\in A\cap B$;
\item $\exists z\compflip y$: $z\in A\cap B$.
\end{enumerate}
Let us consider these options from a technical point of view and a natural language point of view. On the technical side, option 1 has been considered a kind of ``strict'' implication (cf.~\citealt[p.~150]{Chiara2002}, \citealt{Chen2022}, \citealt{Kawano2022}) in the context of quantum logic. Options 3 and 5 both determine the Heyting implication in compossible reflexive frames representing Heyting algebras\footnote{\label{HeytingNote}Recall  Theorem \ref{HeytOrthBoole}.\ref{HeytOrthBoole1}. In compossible reflexive frames, a definition used in  \citealt[Thm.~2.21(i)]{Holliday2022}  that is equivalent to options 3 and 5 is that $x\in A\to B$ iff for every $y$ that pre-refines $x$, if $y\in A$, then $y\in B$. Toward proving  the equivalence, first a lemma about Modus Ponens under option 3: if $x\in A$ and $x\in A\to B$, then $x\in B$. For if $y\comp x$, then by compossibility, there is a $z$ that refines $y$  and pre-refines $x$; since $\comp$ is reflexive and $z$ pre-refines $x$, we have $z\comp x$ and $z\in A$ by Lemma \ref{prerefinelem}.  Given $x\in A\to B$, $z\comp x$, and $z\in A$, there is a $w\in B$ with $z\comp w$. Then since $z$ post-refines $y$, we have $y\comp w$. Thus, we have shown that $\forall y\comp x$ $\exists w\compflip y$: $w\in B$, so $x\in B$. Now for the equivalence, suppose $x\in A\to B$ according to option 3. Further suppose that $y$ pre-refines $x$, and $y\in A$. Then $y\in A\to B$ by Lemma \ref{prerefinelem}, so $y\in B$ by the Modus Ponens lemma, so $x\in A\to B$ according to the definition from \citealt{Holliday2022}. Conversely, suppose $x\in A\to B$ according to that definition, which obviously validates Modus Ponens. Further suppose $y\comp x$ and $y\in A$. Then by compossibility, there is a $z$ that refines $y$ and pre-refines $x$, and by reflexivity, $z\comp z$. Hence $y\comp z$, $z\in A$, and $z\in A\to B$, so $z\in A\cap B$ by Modus Ponens, so $x\in A\to B$ according to options 3 and 5.} and are equivalent to options 2 and 4, respectively, in symmetric frames for ortholattices. Our representation theorem for negation, Theorem \ref{NegThm}, smoothly generalizes to an implication of type 3 or 5, as we show for 5 in Theorem~\ref{CombImp} below and for 3 in Appendix \ref{AppendixB}.

From the point of view of `if...then' in natural language, options 1-3 are all problematic, as they validate antecedent strengthening, i.e., if $A\subseteq C$, then $C\to B\subseteq A\to B$, which is apparently invalid for `if...then' in natural language (\citealt{Stalnaker1968}). For example,  `if it's cloudy, then it might be raining' ($c\to\Diamond r$) clearly does not entail `if it's cloudy and it's not raining, then it might be raining' ($(c\wedge\neg r)\to \Diamond r$). The problem for option 4 (resp.~2) is that it renders $A\to B= \neg (A\cap \neg (A\cap B))$ (resp.~$A\to B=\neg (A\cap\neg B)$); but the right-to-left inclusion is rejected by semanticists for `if...then' in natural language (see, e.g., \citealt{Edgington1995}) and by intuitionists even for `if...then' in mathematical proofs. Option 5 does not appear to enforce any problematic principles if we restrict attention to the operations $\wedge,\vee,\to$, as shown by Theorem \ref{CombImp} below. However, like all the other options, option 5 leads to $\neg A\subseteq A\to 0$;  yet we can assign high probability to `It's not raining' and yet almost no probability to `If it is raining, then a tsunami is flattening Manhattan', which shows that $\neg p$ should not entail $p\to \bot$ under an understanding of entailment with respect to which probability is monotonic (as it must be if we are to have anything like standard probability~theory). 

If we temporarily set aside the interaction of $\to$ and $\neg$, then the basic properties of the option 5 conditional, which we will write as 
\[ A\twoheadrightarrow_\comp B = \{x\in X\mid \forall y\comp x\;(y\in A\Rightarrow \exists z\compflip y:z\in A\cap B)\},\]
appear quite plausible, as listed in the following definition. Note that in the terms of Remark \ref{IntuitivePicture}, we have $x \in A\twoheadrightarrow_\comp B$ iff no $y$ open to $x$ accepts $A$ but rejects $A\cap B$. Also note that our closure operator $c_\comp$ is definable from $\twoheadrightarrow_\comp$ by $c_\comp(A)=X \twoheadrightarrow_\comp A$. 

\begin{definition}\label{ImpAlg} Given a bounded lattice $L$, a \textit{preconditional} on $L$ is a binary operation $\to$  on $L$ satisfying the following for all $a,b,c\in L$:
\begin{enumerate}
\item\label{ImpAlg1} $1\to a\leq a$;
\item\label{ImpAlg2} $a\wedge b\leq a\to b$;
\item\label{ImpAlg4} $ a\to b \leq a\to (a\wedge b)$;
\item\label{ImpAlg3} if $b\leq a$, then $a\to (b\to c)\leq b\to c$;
\item\label{ImpAlg5} if $b\leq c$, then $a\to b\leq a\to c$.
\end{enumerate}
\end{definition}
\noindent Any bounded lattice can be equipped with a preconditional defined by: if $a\leq b$, then $a\to b=1$; otherwise $a\to b=a\wedge b$. Moreover, in any bounded lattice with a precomplementation $\neg$, the operation $\to$ defined by $a\to b=\neg a\vee (a\wedge b)$ is a preconditional with $\neg a=a\to 0$.  In Appendix \ref{AppendixB}, we give an axiomatization of \textit{preimplications} that differs from that of preconditionals by replacing the ability to combine antecedent and consequent as in axiom \ref{ImpAlg4} above with the ability to strengthen the antecedent. In a Heyting algebra, the relative pseudocomplementation $\to$ is both a preconditional and a preimplication.\footnote{By contrast, we note that the implication in algebras for Visser's \citeyearpar{Visser1981} \textit{basic propositional logic} is not necessarily a preconditional or preimplication, since it can violate $1\to a\leq a$.} 

\begin{proposition} For any relational frame $(X,\comp)$, the operation $\twoheadrightarrow_\comp$ is a preconditional on $\lat(X,\comp)$.
\end{proposition}
\begin{proof} Part \ref{ImpAlg1} of Definition \ref{ImpAlg} follows from the observation that $c_\comp(A)=X\twoheadrightarrow_\comp A$, so if $A$ is a $c_\comp$-fixpoint, then $A=X\twoheadrightarrow_\comp A$.  Parts \ref{ImpAlg2}, \ref{ImpAlg4}, and \ref{ImpAlg5} are immediate from the definition of $\twoheadrightarrow_\comp$. For part \ref{ImpAlg3}, suppose $B\subseteq A$, $x\in A\twoheadrightarrow_\comp (B\twoheadrightarrow_\comp C)$,  $x'\comp x$, and $x'\in B$, so $x'\in A$. It follows that there is a $y\compflip x'$ such that $y\in B\twoheadrightarrow_\comp C$, which with $x'\in B$ implies there is an $x''\compflip x'$ with $x''\in B\cap C$. This shows that $x\in B \twoheadrightarrow_\comp C$.\end{proof}

We now show that any preconditional can be represented as $\twoheadrightarrow_\comp$ in a relational frame.

\begin{theorem}\label{CombImp}
 Let $L$ be a bounded lattice and $\to $ preconditional on $L$. Then where
\begin{eqnarray*}
P=\{(x,x\to y)\mid x,y\in L\}\mbox{ and }(a, b)\comp (c, d)\mbox{ if }c\not\leq b,
\end{eqnarray*}
there is a complete embedding of $(L,\to)$ into $(\lat(P,\comp),\twoheadrightarrow_\comp)$, which is an isomorphism if $L$ is complete.\end{theorem}
\begin{proof} First we claim that $P$ is separating as in Definition~\ref{Good}. For part~\ref{Good2} of Definition~\ref{Good}, given $a\not\leq b$, set $(c,d)=(a,a\to 0)$. For part \ref{Good3} of Definition \ref{Good},  suppose $(c,d)\in P$ and $c\not\leq b$. Then set $(c',d')=(1,1\to b)$. From $c\not\leq b$, we have $c\not\leq 1\to b$ by Definition \ref{ImpAlg}.\ref{ImpAlg1}, so $(c',d')\comp (c,d)$. Now consider any $(c'',d'')\in P$ with $(c',d')\comp (c'',d'')$. Then $c''\not\leq d'=1\to b$ and hence $c''\not\leq 1\wedge b = b$ by Definition \ref{ImpAlg}.\ref{ImpAlg2}, so $c''\not\leq b$. Hence parts \ref{Good2} and \ref{Good3} of Definition \ref{Good} hold, so by Proposition \ref{CompRep}, $f$ is a complete embedding of $L$ into $\mathfrak{L}(P,\comp)$, which is a lattice isomorphism if $L$ is complete.

Next we claim that $ f(a\to b)= f(a)\twoheadrightarrow_\comp f(b)$. First suppose $(x,x\to y)\in  f(a\to b)$, so $x\leq a\to b$. Further suppose that $(x',x'\to y')\comp (x,x\to y)$ and $(x',x'\to y')\in f(a)$, so  $x'\leq a$. From $(x',x'\to y')\comp (x,x\to y)$,  we have $x\not\leq x'\to y'$. Now we claim that $a\wedge b\not\leq x'\to y'$.  For if $a\wedge b\leq x'\to y'$, then by Definition \ref{ImpAlg}.\ref{ImpAlg4}, \ref{ImpAlg}.\ref{ImpAlg5}, and \ref{ImpAlg}.\ref{ImpAlg3} (given $x'\leq a$), we have
\[x\leq a\to b\leq a\to (a\wedge b)\leq a\to (x'\to y') \leq x'\to y',\]
contradicting $x\not\leq x'\to y'$. Let $(x'',x''\to y'')=(a\wedge b,(a\wedge b)\to 0)$. Then $(x'',x''\to y'')\in P$, $(x',x'\to y')\comp (x'',x''\to y'')$, and $(x'',x''\to y'')\in f(b)$. Hence $(x,x\to y)\in f(a)\twoheadrightarrow_\comp f(b)$. 
 
 Conversely, suppose $(x,x\to y)\in P\setminus f(a\to b)$, so $x\not\leq a\to b$. Let  $(x',x'\to y')= (a,a\to b)$, so $(x',x'\to y')\comp (x,x\to y)$. Now suppose $(x',x'\to y')\comp (x'',x''\to y'')$, so $x''\not\leq x'\to y'=a\to b$. It follows by Definition~\ref{ImpAlg}.\ref{ImpAlg2} that $x''\not\leq a\wedge b$, so $(x'',y'')\not\in f(a)\cap f(b)$. Hence $(x,x\to y)\not\in f(a)\twoheadrightarrow_\comp f(b)$. \end{proof}
 
 \noindent A completeness theorem for a \textit{preconditional logic} with a connective $\to$ obeying principles matching those of Definition \ref{ImpAlg} can easily be obtained from Theorem \ref{CombImp}, just as we obtained completeness theorems for logics with $\neg$ from Theorem \ref{NegThm}. It is also straightforward to add the quantifiers $\forall$ and $\exists$ (recall Theorem \ref{QuantComplete}) to such a logic. The next step is to consider reasonable axioms to add to those of preconditionals and to characterize the corresponding classes of relational frames, though we will not do so here. (Some of the correspondence facts for $\twoheadrightarrow_\comp$ are the same as for $\to_\comp$ in Appendix \ref{AppendixB}, such as Lemma \ref{ImpCorr}.\ref{ImpCorr0} and \ref{ImpCorr}.\ref{ImpCorr3}.)
 
There are multiple ways to deal with the problem that $\neg_\comp A\subseteq A\twoheadrightarrow_\comp 0$. One is to work with frames ${(X,\comp,\blacktriangleleft)}$ with two relations $\comp$ and $\blacktriangleleft$, interpreting negation as $\neg_\blacktriangleleft$ and the conditional as $\twoheadrightarrow_\comp$, with an interaction condition between $\comp$ and $\blacktriangleleft$ equivalent to the condition that $\neg_\blacktriangleleft$ maps $c_\comp$-fixpoints to $c_\comp$-fixpoints.\footnote{The equivalent condition is that if $y\blacktriangleleft x$, then $\exists x'\comp x$ $\forall x''\compflip x'$ $\exists z\blacktriangleleft x''$: $z$ pre-refines $y$. To see this is sufficient, suppose $x\not\in \neg_\blacktriangleleft A$, so there is a $y\blacktriangleleft x$ with $y\in A$. Then by the condition, $\exists x'\comp x$ $\forall x''\compflip x'$ $\exists z\blacktriangleleft x''$: $z$ pre-refines $y$. Since $z$ pre-refines $y$ and $A$ is a $c_\comp$-fixpoint, $z\in A$ by Lemma \ref{prerefinelem}, so $x''\not\in \neg_\blacktriangleleft A$. Thus, assuming $x\not\in \neg_\blacktriangleleft A$, we have  $\exists x'\comp x$ $\forall x''\compflip x'$, $x''\not\in \neg_\blacktriangleleft A$, which shows that $\neg_\blacktriangleleft A$ is a $c_\comp$-fixpoint. For necessity, suppose the condition does not hold. Let $A=c_\comp(\{y\})$, which is the set of states that pre-refine $y$. Then $x\not\in \neg_\blacktriangleleft A$ but $\forall x'\comp x$ $\exists x''\compflip x'$: $x''\in \neg_\blacktriangleleft A$, so $\neg_\blacktriangleleft A$ is not a $c_\comp$-fixpoint.} Since presumably we want $A\twoheadrightarrow_\comp 0\subseteq\neg_\blacktriangleleft A$, we require $\blacktriangleleft\,\subseteq \,\comp$. Then we extend Theorem \ref{CombImp} as follows.

\begin{theorem}\label{CombImpNeg} Let $L$ be a bounded lattice, $\to $ preconditional on $L$, and $\neg$ an antitone operation on $L$ such that for all $a\in L$, $a\to 0\leq \neg a$. Then where $P$ and $\comp$ are defined as in Theorem \ref{CombImp}, and $\blacktriangleleft$ is defined by
\[(x',x'\to y')\blacktriangleleft (x,x\to y)\mbox{ iff } (x',x'\to y')\comp (x,x\to y)\mbox{ and for all }a\in L, x\leq \neg a\mbox{ implies }x'\not\leq a,\]
there is a complete embedding of $(L,\to,\neg)$ into $(\lat(P,\comp),\twoheadrightarrow_\comp,\neg_\blacktriangleleft)$, which is an isomorphism if $L$ is complete.
\end{theorem}
\begin{proof} We need only add to the proof of Theorem \ref{CombImp} that $f(\neg a)=\neg_\blacktriangleleft f(a)$. Suppose $(x,x\to y)\in f(\neg a)$, so $x\leq \neg a$. Then for all $(x',x'\to y')\blacktriangleleft (x,x\to y)$, we have $x'\not\leq a$ and hence $(x',x'\to y')\not\in f(a)$, so $(x,x\to y)\in  \neg_\blacktriangleleft f(a)$. Conversely, suppose $(x,x\to y)\in f(\neg a)$, so $x\not\leq \neg a$. Then $x\not\leq a\to 0$, so $(a,a\to 0)\comp (x,x\to y)$. Moreover, for all $b\in L$, if $x\leq \neg b$, then given $x\not\leq \neg a$ we have $\neg b\not\leq \neg a$, so $a\not\leq b$ by the antitonicity of $\neg$. Thus, $(a,a\to 0)\blacktriangleleft (x,x\to y)$, which shows $(x,x\to y)\not\in  \neg_\blacktriangleleft f(a)$.\end{proof}
\noindent One can then impose additional conditions on $\blacktriangleleft$ to validate additional principles for $\neg_\blacktriangleleft$. Moreover, an analogue of the topological representation of lattices with $\neg$ in Theorem \ref{TopRep} can be given for lattices with $\to$ and $\neg$ based on the idea of Theorem \ref{CombImpNeg} (cf.~Theorem \ref{EmbedThmImp} in Appendix \ref{AppendixB}). We leave for future work the systematic investigation of this approach to handling negation and conditionals.

A different semantic approach continues to represent $(L,\neg)$ as $(\lat(X,\comp),\neg_\comp)$  but  treats $A\to(\cdot)$ as a normal modal operation interpreted by an accessibility relation $R_A$ on $X$, as in ``set-selection function'' semantics (\citealt[\S~2.7]{Lewis1973}), such that $xR_A y$ implies $y\in A$. Thus, $xR_Ay$ means that $y$ is one of the relevant $A$-possibilities at which $B$ must hold in order for $A\to B$ to hold at $x$. \citealt[\S~4]{Holliday2022} includes representation theorems for bounded lattices equipped with both a negation $\neg$ and a normal modal $\Box$, using triples $(X,\comp, R)$ where $R$ is a binary relation on $X$ satisfying an interaction condition with $\comp$ that guarantees that the $\Box_R$ operation defined by $\Box_R B = \{x\in X\mid \forall y\in X (xRy\Rightarrow y\in B)\}$ sends $c_\comp$-fixpoints to $c_\comp$-fixpoints. The same approach can be applied to conditionals, only we now represent each normal modal operation $A\to(\cdot)$ by a binary relation $R_A$. In the filter-ideal space of $(L,\neg,\to)$ as in \S~\ref{TopRep},\footnote{Applying the discrete representation of \S~\ref{DiscreteRep} to complete lattices with modalities raises additional issues, such as the requirement that $\Box$ (resp.~$A\to(\cdot)$) be \textit{completely multiplicative} (see \citealt[\S~4]{Holliday2022}).} one defines \[\mbox{$(F,I)R_{\widehat{a}}(F',I')$ iff for all $b\in L$, $a\to b\in F$ implies $b\in F'$,}\]
and then the modal operation $a\to (\cdot)$ is represented by $\Box_{R_{\widehat{a}}}$ (cf.~\citealt[Prop.~4.10]{Holliday2022}). Assuming $(L,\neg,\to)$ satisfies $a\to a=1$ for all $a\in L$, then $R_{\widehat{a}}$ satisfies the constraint that $R_{\widehat{a}}$-successors belong to $\widehat{a}$.

Treating $A\to (\cdot)$ as a normal modality matches a natural proof-theoretic approach to $\to$ based on Fitch-style proofs for modal logic (\citealt{Fitch1966}). Fitch distinguishes between ordinary subproofs, used for $\neg$I and $\vee$E, and box subproofs (his terms is `strict column'), used for $\Box$I. Similarly, we  distinguish between ordinary subproofs, used for $\neg$I and $\vee$E, and arrow subproofs, used for $\to$I. This slightly complicates the rigorous inductive definition of proofs, but the basic idea is straightforward. Just as Fitch indicates his box subproofs with a $\Box$ symbol to the left of the vertical subproof line, we will indicate our arrow subproofs with a $\to$ symbol to the left of the vertical subproof line. A more important difference is that since Fitch \citeyearpar{Fitch1966} dealt only with a unary modal $\Box$, rather than our binary or indexed operators, his box subproofs have no assumptions, whereas our arrow subproofs will. Our  $\to$I rule says that if a proof contains an arrow subproof beginning with $\varphi$ and ending with $\psi$, then one can add $\varphi\to\psi$ on the next line of the proof. The $\to$E rules says that if a proof contains $\varphi\to \psi$ and ends with an arrow subproof whose  assumption is $\varphi$, then that arrow subproof can be extended with $\psi$. The rules are shown diagramatically in Figure \ref{ToRules}. 

\begin{figure}[h]
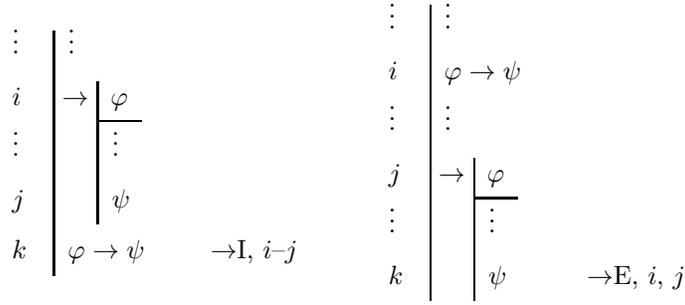

\begin{center}
\begin{minipage}{2in}
\[\begin{nd}
\have [\vdots] {0} {\vdots}
\open
\hypo [i] {3}   {\hspace{-.26in}\to\;\;\varphi}
\have [\vdots] {4}   {\vdots}
\have [j] {6}   {\psi}
\close
\have [k]{7} {\varphi\to\psi} \ii{3-6}
\end{nd}
\]\end{minipage}\begin{minipage}{2in}\[\begin{nd}
\have [\vdots] {} {\vdots}
\have [i] {0} {\varphi\to\psi}
\have [\vdots] {1} {\vdots}
\open
\hypo [j] {3}   {\hspace{-.26in}\to\;\;\varphi}
\have [\vdots] {4}   {\vdots}
\have [k] {6}   {\psi} \ie{0,3}
\end{nd}
\]
\end{minipage}\end{center}
\caption{Introduction and elimination rules for $\to$.}\label{ToRules}
\end{figure}

One might argue for adding to the $\to$E rule that if a proof contains $\varphi$ and $\varphi\to\psi$, then one can extend the proof with $\psi$, per Modus Ponens, the traditional $\to$ elimination rule, as shown on the left of Figure \ref{ToNegRule}. But McGee \citeyearpar{McGee1985} has famously argued that one can assign higher probability to $p\wedge (p\to (q\to r))$ than to $q\to r$, so the former does not entail the latter (cf.~\citealt[\S~4.2]{Santorio2022}). On the other hand, if instead of trying to capture a notion of entailment with respect to which probability is monotonic, we try to capture  \textit{preservation of probability~$1$}, then Modus Ponens seems unimpeachable: if the probability of $\varphi$ is 1 and the probability of $\varphi\to\psi$ is 1, then the probability of $\psi$ should be 1 as well. Now if we simply extend $\vdash_\mathsf{F}$ with the traditional introduction and elimination rules for $\to$ (left of Figure \ref{ToRules}, left of Figure \ref{ToNegRule}), but without $\to$E from Figure \ref{ToRules} (in which case there is no real difference between arrow subproofs and ordinary subproofs), then we obtain a logic whose algebraic semantics is given by bounded lattices equipped with a weak pseudocomplementation and a binary operation $\to$ satisfying the properties that if $a\leq b$, then $a\to b=1$ (for $\to$I), and $a\wedge (a\to b)\leq b$ (for MP). But it would seem that if Modus Ponens is acceptable, then so is the $\to$E rule of Figure \ref{ToRules}, so we should have both. Under the interpretation of $A\to(\cdot)$ as $\Box_{R_A}$ above, which matches the rules in Figure \ref{ToRules}, to validate Modus Ponens it suffices to assume   \textit{weak centering} (\citealt{Lewis1973}): if $x\in A$, then $xR_Ax$.

One might also argue for strengthening  $\to$E  so that if a proof contains $\varphi\to\psi$ and an \textit{ordinary} subproof beginning with $\varphi$, then one can extend that ordinary subproof with $\psi$. But applying this to ordinary subproofs for $\neg$I yields the Modus Tollens inference, $ \neg\psi \wedge (\varphi\to\psi)\vdash \neg\varphi$, which Veltman \citeyearpar[p.~3]{Veltman1985} has argued is invalid using examples in which $\psi$ contains conditionals and Yalcin \citeyearpar{Yalcin2012} has argued is invalid using examples in which $\psi$ contains epistemic modals; e.g., from `The card might not be diamonds or hearts, but if it is red, then it must be diamonds or hearts', it does not follow that  `The card is not red' (we assume that `might not' entails `not must'). The idea of applying $\to$E to ordinary subproofs for $\vee$E yields $(\varphi\vee\psi)\wedge (\varphi\to\chi)\wedge (\psi\to\chi)\vdash \chi$, which arguably also admits counterexamples with epistemic modals: from `the card is red or black; if it's red, it must be diamonds or hearts; and if it's black, it must be clubs or spades', it does not follow that  `it must be diamonds or hearts, or it must be clubs or spades', since surely it might not be diamonds or hearts, and it might not be clubs or spades (cf.~\citealt{Kolodny2010} for examples with deontic modals). The basic problem is that an ordinary subproof beginning with $\varphi$ corresponds to considering a possibility where $\varphi$ is merely true, whereas the natural language uses of `if' above seem to involve a hypothetical update of a body of information to a new body of information in which the antecedent is ``known.'' In contrast to ordinary subproofs, one can interpret an arrow subproof beginning with $\varphi$ as corresponding to a hypothetical update of that kind. See \citealt{HM2022} for further discussion of the logic of conditionals and epistemic~modals.

In the context of the proof system with the rules of Figure \ref{ToRules}, one might argue for another way of introducing negation: if a hypothetical update with $\varphi$ leads to a contradiction, then conclude $\neg\varphi$. Supposing we now have $\bot$ as a primitive symbol in our language (interpreted as $0$ in our lattices), this rule is shown on the right of Figure \ref{ToNegRule}. Algebraically, this is just $a\to 0\leq\neg a$ again. With respect to our pseudosymmetric reflexive frames equipped with accessibility relations $R_A$ to define $A\to (\cdot)$, the condition that $A\to 0\subseteq\neg A$ corresponds to the condition that if $x\compflip y\in A$, then $R_A(x)\neq\varnothing$.

\begin{figure}[h]
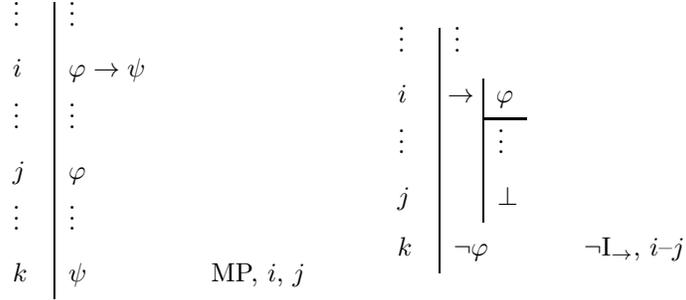

\begin{center}
\begin{minipage}{2in}
\[\begin{nd}
\have [\vdots] {} {\vdots}
\have [i] {0} {\varphi\to\psi}
\have [\vdots] {1} {\vdots}
\have [j] {3}   {\varphi}
\have [\vdots] {4}   {\vdots}
\have [k] {6}   {\psi} \MP{0,3}
\end{nd}
\]
\end{minipage}\begin{minipage}{2in}
\[\begin{nd}
\have [\vdots] {0} {\vdots}
\open
\hypo [i] {3}   {\hspace{-.26in}\to\;\;\varphi}
\have [\vdots] {4}   {\vdots}
\have [j] {6}   {\bot}
\close
\have [k]{7} {\neg\varphi} \nii{3-6}
\end{nd}
\]\end{minipage}
\end{center}
\caption{Modus Ponens (left) and another way to introduce negation (right).}\label{ToNegRule}
\end{figure}

None of the above is meant to suggest that the proof system for $\wedge$, $\vee$, $\neg$, and $\to$ consisting of the I and E rules in Figures \ref{FitchRules}, \ref{ToRules}, and \ref{ToNegRule} (and perhaps Modus Ponens, depending on one's target notion of entailment) is the strongest reasonable system. But our goal has not been to formulate as strong a logic as ultimately reasonable but rather to identify a fundamental starting point based on introduction and elimination rules.

\section{Conclusion}\label{ConclusionSection}

We have presented a logic in the signature with conjunction, disjunction, negation, and the universal and existential quantifiers that is based purely on the introduction and elimination rules for the logical constants. The corresponding algebraic semantics is based on bounded lattices with weak pseudocomplementations. We have seen that such lattice expansions admit representation theorems using  pseudosymmetric reflexive frames, furnishing an elegant relational semantics for the logic. From this starting point, intuitionistic logic, orthologic, and classical logic can be obtained either proof-theoretically---by adding to our Fitch-style proof system Reiteration, Reductio ad Absurdum, or both---or semantically---by adding to our relational frames the properties of compossibility, symmetry, or both. We also sketched options for adding a conditional to our logic, though it remains to fill out this sketch in future work.

The relational semantic approach we have developed applies far beyond the fundamental logic, both to stronger and weaker logics that can be systematically investigated in the style of investigations of  logics intermediate between intuitionistic and classical logic. Moreover, adding to our relational frames a second relation of accessibility yields semantics for modal versions of these logics (\citealt{Holliday2022}). In richer modal languages, we may be able to delineate those fragments of a language for which classical, intuitionistic, or orthological reasoning is safe from those that call for the caution of the fundamental logic.

\subsection*{Acknowledgements}

I thank Juan P. Aguilera, Johan van Benthem, Ahmee Christensen, Yifeng Ding, Cian Dorr, Kit Fine, Bas van Fraassen, Peter Fritz, Valentin Goranko, Dominic Hughes, John MacFarlane, Paolo Mancosu, Matthew Mandelkern, Guillaume Massas, Eric Pacuit, Francesca Poggiolesi, Daniel Villalon, James Walsh, Dag Westerst{\aa}hl, and the anonymous referees for helpful comments. I am also grateful to audiences at Advances in Modal Logic 2022, Colloquium Logicum 2022, the University of Pennsylvania Logic Seminar and UC Berkeley Logic Colloquium in October 2022, and the NYU Philosophy Colloquium in December 2022.

\appendix

\section{Appendix}\label{AppendixA}

In this appendix, we give a definition of Fitch-style proofs for intuitionistic logic in the $\wedge,\vee,\neg$ fragment.\footnote{The introduction and elimination rules for the intuitionistic implication $\to$ can obviously be added in the same style.} In particular, we define the notion of a \textit{proof $\sigma$ given a set $R$ of reiterables}, where reiterables are formulas. Then $\varphi\vdash \psi$ if there is a proof given the empty set of reiterables that begins with $\varphi$ and ends with $\psi$. 

For each set $R$ of formulas, the set of \textit{proofs given $R$} is the smallest set containing for each formula $\varphi$ the sequence  $\langle \varphi\rangle$ and satisfying the following closure conditions for $1\leq i,j\leq n$:
\begin{itemize}
\item If $\langle \sigma_1,\dots,\sigma_n\rangle$ is a proof given $R$ and $\tau$ is a proof given $R\,\cup\, \{\sigma_i\mid \sigma_i\mbox{ a formula}\}$, then $\langle \sigma_1,\dots,\sigma_n,\tau\rangle$ is a proof given $R$.
\item If $\langle \sigma_1,\dots,\sigma_n\rangle$ is a proof given $R$ and $\varphi\in R$, then $\langle \sigma_1,\dots,\sigma_n,\varphi\rangle$ is a proof given $R$ (Reiteration).
\item closure conditions for $\wedge$I, $\wedge$E, $\vee$I, $\vee$E, $\neg$I, and $\neg$E as in \S~\ref{FitchSection} with `proof' replaced by `proof given $R$'.
\end{itemize}
The Reiteration rule is illustrated in Figures \ref{FirstFigure} and \ref{PseudoFig}. In Figure \ref{FirstFigure}, the first subproof is a proof given $\{\Diamond\neg p\}$ as the set of reiterables, while in Figure \ref{PseudoFig}, the sole subproof is a proof given $\{\psi\}$ as the set of reiterables; in both cases, the proof as a whole, containing the subproof(s), is a proof given the empty set of reiterables. For an example in the sequential format, consider the following proof of distributivity:
\[\big\langle p\wedge (q\vee r), p, q\vee r, \big\langle q, p, (p\wedge q), (p\wedge q)\vee (p\wedge r) \big\rangle,\big\langle r,p,(p\wedge r), (p\wedge q)\vee (p\wedge r) \big\rangle, (p\wedge q)\vee (p\wedge r) \big\rangle.\]
Here $\langle p\wedge (q \vee r), p,q\vee r\rangle$ is a proof given the empty set of reiterables. Then since $p$ appears in that proof, and $\langle q, p, (p\wedge q), (p\wedge q)\vee (p\wedge r) \big\rangle$ is a proof given $\{p\}$ as the set of reiterables, we obtain that $ \big\langle p\wedge (q\vee r), p, q\vee r, \big\langle q, p, (p\wedge q), (p\wedge q)\vee (p\wedge r) \big\rangle\big\rangle$ is a proof given the empty set of reiterables, and so on.

Note that if we drop the second bullet point for Reiteration from the definition above, then the notion of \textit{proof given $R$} coincides with our original notion of proof for $\vdash_\mathsf{F}$ in \S~\ref{FitchSection}. Thus, the only gap between $\vdash_\mathsf{F}$ and intuitionistic logic is indeed the Reiteration rule.

For classical logic, we simply add the following to the definition above:
\begin{itemize}
\item If $\langle \sigma_1,\dots,\sigma_n\rangle$ is a  proof given $R$, $\sigma_i$ is a formula of the form $\psi$, and $\sigma_n$ is a sequence beginning with $\neg\varphi$ and ending with $\neg\psi$, then $\langle \sigma_1,\dots,\sigma_n,\varphi\rangle$ is a proof given $R$ (RAA).
\end{itemize} 

\section{Appendix}\label{AppendixB}

In this appendix, we extend the relational representation of lattices with negations from \S~\ref{RelationalSection} to certain kinds of implications. Given a relational frame $(X,\comp)$, we define a binary operation $\to_\comp$ on $\lat(X,\comp)$ by
\[A\to_\comp B=\{x\in X\mid \forall x'\comp x\;(x'\in A\Rightarrow \exists x''\compflip x':x''\in B)\}.\]
The operation $\twoheadrightarrow_\comp$ from \S~\ref{Conditionals} is then definable by
\[A\twoheadrightarrow_\comp B = A\to_\comp (A\cap B),\]
and the closure operator $c_\comp$ and negation $\neg_\comp$ from Proposition \ref{IsClosure} are definable by
\begin{eqnarray*}
c_\comp(A)&=& X\to_\comp A  \\
\neg_\comp A&=& A\to_\comp 0 ,
\end{eqnarray*}
using Lemma \ref{AbsurdLem}.\ref{AbsurdLem1} for the second equation.\footnote{Returning to the issue of morphisms broached in Footnote \ref{MorphismsNote}, a candidate notion of morphism between relational frames that also preserves $\to_\comp$ is a map $f$ that satisfies (i) and (ii) from Footnote \ref{MorphismsNote} plus two extra conditions for $\to_\comp$. First recall (iii) from Footnote \ref{MorphismsNote}, expressed in the language of Definition \ref{RefinementDef}: if $y'\comp 'f(x)$, then $\exists y\comp x$: $f(y)$ pre-refines $y'$. This ensures $\neg_\comp f^{-1}[A']\subseteq f^{-1}[\neg_{\comp'}A']$. To ensure $f^{-1}[A']\to_\comp f^{-1}[B'] \subseteq f^{-1}[A'\to_{\comp'} B']$, we strengthen (iii) to (iii$^+$): if $y'\comp 'f(x)$, then $\exists y\comp x$: $f(y)$ \textit{refines} $y'$. For suppose $x\in f^{-1}[A']\to_\comp f^{-1}[B']$. To show $f(x)\in A'\to_{\comp'} B'$, suppose $y'\comp' f(x)$ and $y'\in A'$. Then picking $y$ as in (iii$^+$), since $f(y)$ pre-refines $y'$, we have $f(y)\in A'$ by Lemma \ref{prerefinelem}. Hence $y\in f^{-1}[A']$, which with $y\comp x$ and $x\in f^{-1}[A']\to_\comp f^{-1}[B']$ implies there is a $z\compflip y$ with $z\in f^{-1}[B']$, so $f(z)\in B'$. Then from $z\compflip y$ we have $f(z)\compflip' f(y)$ by (i), and then since $f(y)$ post-refines $y'$, we have $f(z)\compflip' y'$. Thus, we have shown that for all $ y'\comp' f(x)$ with $y'\in A'$, there is a $ z'\compflip' y'$ with $z'\in B'$, so $f(x)\in A'\to_\comp B'$. Finally, to ensure $f^{-1}[A'\to_{\comp'} B']\subseteq f^{-1}[A']\to_\comp f^{-1}[B']$, consider (iv) (and compare it with (iii)): if $y'\compflip' f(x)$, then $\exists y\compflip x$: $f(y)$ pre-refines $y'$. We will apply (iv) with a change of variables: if $z'\compflip' f(y)$, then $\exists z\compflip y$: $f(z)$ pre-refines $z'$. Now suppose $f(x)\in A'\to_{\comp'} B'$. To show  $x\in   f^{-1}[A']\to_\comp f^{-1}[B']$, suppose $y\comp x$ and $y\in f^{-1}[A']$, so $f(a)\in A'$. By (i), we have $f(y)\comp' f(x)$. Then since $f(x)\in A'\to_{\comp'} B'$, there is a $z'\compflip' f(y)$ such that $z'\in B'$. Then taking $z$ as in (iv), we have $f(z)\in B'$ by Lemma \ref{prerefinelem}, so $z\in f^{-1}[B']$. Thus, we have shown that for all $y\comp x$ with $y\in f^{-1}[A']$, there is a $z\compflip y$ with $z\in f^{-1}[B']$, so $x\in   f^{-1}[A']\to_\comp f^{-1}[B']$.}

Just as we identified  conditions on $\comp$ corresponding to axioms on $\neg_\comp$ (Lemma \ref{CorrLemm}), we can do the same for $\to_\comp$. We give only a brief sample in the following. For axioms on an implication $\to$ on a lattice $L$,  we consider relativizing earlier axioms involving $0$ to an arbitrary $b\in L$:\footnote{A referee informed me that this idea is what led Meyer and Slaney \citeyearpar{Meyer1989} to their Abelian Logic by generalizing the classical axiom $\neg\neg a\to a$ to $((a\to b)\to b)\to a$.}
\begin{itemize}
\item $\neg 0=1$ turns into $b\to b=1$;
\item $a\wedge \neg a\leq 0$ turns into $a\wedge (a\to b)\leq b$;
\item $a\leq\neg\neg a$ turns into $a\leq (a\to b)\to b$;
\item $a\wedge c\leq 0\Rightarrow a\leq \neg c$ turns into $a\wedge c\leq b\Rightarrow a\leq c\to b$.
\end{itemize}
Note by contrast that $\neg\neg a\leq a$ does not turn into a classically valid law when replacing $0$ with $b$.

\begin{lemma}\label{ImpCorr} For any relational frame $(X,\comp)$, in each of the following pairs, (a) and (b) are equivalent:
\begin{enumerate}
\item\label{ImpCorr0}\begin{enumerate}
\item for all $c_\comp$-fixpoints $B$, we have $B\to_\comp B=1$;
\item for all $x\in X$ and $y\comp x$, there is a $z\compflip y$ that pre-refines $y$.
\end{enumerate}
\item\label{ImpCorr1}\begin{enumerate}
\item for all $c_\comp$-fixpoints $A,B$, we have $A\cap (A\to_\comp B)\subseteq B$;
\item \textit{right pre-interpolation}: for all $x\in X$ and $y\comp x$, there is a $z\comp x$ that post-refines $y$ and pre-refines~$x$.  
\end{enumerate}
\item\label{ImpCorr2} \begin{enumerate}
\item for all $c_\comp$-fixpoints $A,B$, we have $A\subseteq (A\to_\comp B)\to_\comp B$;
\item \textit{left pre-interpolation}: for all $x\in X$ and $y\comp x$, there is a $z\comp y$ that post-refines $y$ and pre-refines~$x$.  
\end{enumerate}
\item\label{ImpCorr3}
\begin{enumerate}
\item for all $c_\comp$-fixpoints $A,B,C$, if $A\cap C\subseteq B$, then $A\subseteq C\to_\comp B$;
\item \textit{left post-extendability}: for all $x\in X$ and $y\comp x$, there is a $z\compflip y$ that pre-refines $y$ and~$x$.
\end{enumerate}
\end{enumerate}
\end{lemma}

\begin{proof} For part \ref{ImpCorr0}, suppose (b) holds, $y\comp x$, and $y\in B$. Hence there is a $z\compflip y$ that pre-refines $y$, so $z\in B$ by Lemma \ref{prerefinelem}. This shows $x\in B\to_\comp B$. Conversely, suppose (b) does not hold, so there are $y\comp x$ for which no $z\compflip y$ belongs to $c_\comp(\{y\})$. Then since $y\in c_\comp(\{y\})$ and $y\comp x$, we have $x\not\in c_\comp(\{y\})\to_\comp c_\comp(\{y\})$.

For part \ref{ImpCorr1}, suppose (b) holds, $x\in A\cap (A\to_\comp B)$, and $y\comp x$. Let $z$ be as in right pre-interpolation. Since $z$ pre-refines $x$, we have $z\in A$, and then since $z\comp x$ and $x\in A\to_\comp B$, there is a $w\compflip z$ with $w\in B$. Since $z$ post-refines $y$, we have $w\compflip y$. Thus, we have shown that $\forall y\comp x$ $\exists w\compflip y$: $y\in B$, so $x\in B$. Conversely, suppose (b) does not hold, so there are $y\comp x$ such that (i) no $z\comp x$ that  pre-refines $x$ post-refines $y$. Let $A$ be the set of states that pre-refine $x$, i.e., $A=c_\comp(\{x\})$, and $B=\{w\in X \mid y\not\comp w\}$. Then $A$ and $B$ are $c_\comp$-fixpoints, and by (i), $x\in A\to_\comp B$, and yet $x\not\in B$. 

For part \ref{ImpCorr2}, suppose (b) holds, $x\in A$, $y\comp x$, and $y\in A\to_\comp B$. Let $z$ be as in left pre-interpolation. Since $z$ pre-refines $x$, we have $z\in A$, and then since $z\comp y$ and $y\in A\to_\comp B$, there is a $w\compflip z$ with $w\in B$. Since $z$ post-refines $y$, we have $w\compflip y$. Thus, we have shown that for all $y\comp x$ with $y\in A\to_\comp B$, there is a $w\compflip y$ with $w\in B$, so $x\in (A\to_\comp B)\to_\comp B$. Conversely, suppose (b) does not hold, so there are $y\comp x$ such that (i) no $z\comp y$ that pre-refines $x$ post-refines $y$. Let $A$ be the set of states that pre-refine $x$ and $B=\{w\in X \mid y\not\comp w\}$. Then $A$ and $B$ are $c_\comp$-fixpoints, and by (i), $y\in A\to_\comp B$, yet there is no $w\compflip y$ with $w\in B$, which with $y\comp x$ implies $x\not\in (A\to_\comp B)\to_\comp B$, and yet $x\in A$.

For part \ref{ImpCorr3}, suppose (b) holds, $A\cap C\subseteq B$, $x\in A$, $y\comp x$, and $y\in C$. Let $z$ be as in left post-extendability. Then since $z$ pre-refines $x$ and $y$, we have $z\in A\cap C$ and hence $z\in B$. Thus, we have shown that for all $y\comp x$, if $y\in C$, then there is a $z\compflip y$ with $z\in B$, which shows $x\in C\to_\comp B$. Conversely, suppose (b) does not hold,  so (i) there are $y\comp x$ such that no $z\compflip y$ pre-refines both $x$ and $y$. Let $A$ be the set of states that pre-refine $x$, $C$ the set of states that pre-refine $y$, and $B=A\cap C$. Then $A$, $B$, and $C$ are $c_\comp$-fixpoints, and by (i), $x\not\in C\to_\comp B$, and yet $x\in A$.\end{proof}

We now identify the implications on lattices that we will be able to represent using the $\to_\comp$ operation (compare the \textit{preconditionals} of Definition \ref{ImpAlg} representable using $\twoheadrightarrow_\comp$).

\begin{definition}\label{MatAlg} Given a bounded lattice $L$, a \textit{preimplication} on $L$ is a binary operation $\to$  on $L$ satisfying the following for all $a,b,c\in L$:
\begin{enumerate}
\item\label{MatAlg1} $a=1\to a$;
\item\label{MatAlg3} $a\to (a\to b)\leq a\to b$;
\item\label{MatAlg4} if $a\leq b$, then $b\to c\leq a\to c$;
\item\label{MatAlg5} if $a\leq b$, then $c\to a\leq c\to b$.
\end{enumerate}
From $\to$ we define a unary operation $\neg$ by $\neg a=a\to 0$. 
\end{definition}
\noindent Any bounded lattice can be equipped with a preimplication defined by: if $a\leq b$, then $a\to b=1$; otherwise $a\to b=b$. In a Heyting algebra, the relative pseudocomplementation $\to$ is clearly a preimplication. In an otholattice with orthocomplementation $\neg$, the operation $\to$ defined by $a\to b=\neg (a\wedge\neg b)$ is a preimplication from which we recover the orthocomplementation by $\neg a=a\to 0$.  More generally, in a bounded lattice with a precomplementation $\neg$, the operation $\to$ defined by $a\to b=\neg a\vee b$ is a preimplication with $\neg a=a\to 0$.

\begin{lemma} For any relational frame $(X,\comp)$, the operation $\to_\comp$ is a preimplication on $\lat(X,\comp)$.
\end{lemma}
\begin{proof} Part \ref{MatAlg1} follows from the observation that $c_\comp(A)=X\to_\comp A$, so if $A$ is a $c_\comp$-fixpoint, then $A=X\to_\comp A$.  For part \ref{MatAlg3}, suppose $x\in A\to_\comp (A\to_\comp B)$,  $x'\comp x$, and $x'\in A$. Then there is a $y\compflip x'$ such that $y\in A\to_\comp B$, which with $x'\in A$ implies there is an $x''\compflip x'$ with $x''\in B$. This shows that $x\in A \to_\comp B$.  Parts \ref{MatAlg4} and \ref{MatAlg5} are immediate from the definition of $\to_\comp$. 
\end{proof}

Next we introduce terminology for preimplications satisfying axioms considered in Lemma \ref{ImpCorr}. 

\begin{definition}\label{ImpDefs} A \textit{protoimplication} is a preimplication satisfying \[b\to b=1\mbox{ and }a\wedge (a\to b)\leq b\] for all $a,b\in L$; an \textit{ultraweak pseudoimplication} (resp.~\textit{weak pseudoimplication}) is a preimplication (resp.~protoimplication) satisfying \[a\leq (a\to b)\to b\] for all $a,b\in L$; and a \textit{relative pseudocomplementation} is a protoimplication satisfying
\[a\wedge c\leq b\Rightarrow a\leq c\to b\]
for all $a,b,c\in L$.
\end{definition}

\noindent The preimplication we defined above on any bounded lattice is in fact a weak pseudoimplication. Concerning the axiom for ultraweak pseudoimplications, we note the following analogue of Lemma \ref{UsefulLem}.\ref{UsefulLem2}.

\begin{lemma}\label{ImpUsefulLem} For a preimplication $\to$, the following are equivalent:
\begin{enumerate}
\item\label{ImpUsefulLem1} for all $a,b\in L$, $a\leq (a\to b)\to b$; 
\item\label{ImpUsefulLem2} for all $a,b,c\in L$, if $a\leq c\to b$, then $c\leq a \to b$.
\end{enumerate}
\end{lemma}
\begin{proof} By \ref{ImpUsefulLem1}, we have $c\leq (c\to b)\to b$, which with $a\leq c\to b$ yields $c\leq a\to b$ by Definition \ref{MatAlg}.\ref{MatAlg4}. Conversely, by \ref{ImpUsefulLem2},  $a\to b\leq a\to b$ implies $a\leq (a\to b)\to b$.
\end{proof}
\noindent Note as a corollary that if $\to$ is an ultraweak pseudoimplication, then from $b\leq 1\to b$, we have $1\leq b\to b$.

We now prove the representation theorem for bounded lattices with preimplications.

\begin{theorem}\label{ImpThm} Let $L$ be a bounded lattice, $\mathrm{V}$ a join dense set of elements of $L$, and $\Lambda$ a meet dense set of elements of $L$. Given a set $P$ of pairs of elements of $L$, define $\comp$ on $P$ by $(a, b)\comp (c, d)$ if $c\not\leq b$. 
\begin{enumerate}
\item\label{ImpThm1} If $\to$ is a preimplication on $L$, then where
\[P=\{(a,a\to b)\mid a,b\in L\},\]
there is a complete embedding of $(L,\to)$ into $(\lat(P,\comp),\to_\comp)$. Moreover, if $\neg$ is an ultraweak pseudocomplementation, then $\neg$ is strongly pseudosymmetric (recall Definition \ref{StrongPseudo}). 
\item\label{ImpThm2} If $\to$ is a protoimplication on $L$, then where
\[P=\{(a,a\to b)\mid a,b\in L,  a\not\leq b\},\]
there is a complete embedding of $(L,\to)$ into $(\lat(P,\comp),\to_\comp)$, and $\comp$ is reflexive and satisfies right pre-interpolation. Moreover, if $\neg$ is a weak pseudocomplementation, then $\comp$ is strongly pseudosymmetric.
\item\label{ImpThm3} If $\to$ is an ultraweak pseudoimplication on $L$, then where
\[P=\{(a,a\to b)\mid a\in \mathrm{V}, b\in L\}\cup \{(1,1\to b)\mid b\in\Lambda\},\]
there is a complete embedding of $(L,\to)$ into $(\lat(P,\comp),\to_\comp)$, and $\comp$ satisfies left pre-interpolation.  
\item\label{ImpThm4} If $\to$ is a weak pseudoimplication on $L$, then where 
\[P=\{(a,a\to b)\mid a\in \mathrm{V}, b\in L, a\not\leq  b\},\]
there is a complete embedding of $(L,\to)$ into $(\lat(P,\comp),\to_\comp)$, and $\comp$ is reflexive and satisfies right pre-interpolation and left pre-interpolation. Moreover, if $\neg$ is a pseudocomplementation, then $\comp$ is weakly compossible (recall Proposition \ref{CorrLemm}.\ref{CorrLemm2}).
\item\label{ImpThm5} If $\to$ is a relative pseudocomplementation on $L$, then where
\[P=\{(a,a\to b)\mid a\in \mathrm{V}, b\in \Lambda, a\not\leq  b\},\]
there is a complete embedding of $(L,\to)$ into $(\lat(P,\comp),\to_\comp)$, and  $\comp$ is reflexive and compossible (recall Definition \ref{RefinementDef}).
\end{enumerate}
In each case, if $L$ is complete, then the embedding is an isomorphism.
\end{theorem}

\begin{proof} First we claim that in each part, $P$ is separating in the sense of Definition~\ref{Good}.  The proof that $P$ is separating in part \ref{ImpThm5} is already in \citealt[Prop.~3.16(iii)]{Holliday2022}, so we give the other cases.  To prove part~\ref{Good2} of Definition~\ref{Good}, assume $a\not\leq b$. For parts \ref{ImpThm1} and \ref{ImpThm2} of the theorem, we set $(c,d)=(a,a\to 0)$, so $(c,d)\in P$ since $a\neq 0$. For parts \ref{ImpThm3} and \ref{ImpThm4} of the theorem, from $a\not\leq b$ we obtain a nonzero $a'\in\mathrm{V}$ such that $a'\leq a$ but $a'\not\leq b$, and we set $(c,d)=(a',a'\to 0)$.  To prove part \ref{Good3} of Definition \ref{Good},  suppose $(c,d)\in P$ and $c\not\leq b$. Hence there is some $b'\in\Lambda$ such that $c\not\leq b'$ and $b\leq b'$.  For parts \ref{ImpThm1} and \ref{ImpThm3} of the theorem, we set $(c',d')=(1,1\to b')$, so $(c',d')\in P$. From $c\not\leq b'$ we also have $c\not\leq 1\to b'$ by the right-to-left inequality in Definition \ref{MatAlg}.\ref{MatAlg1}, so   $(c',d')\comp (c,d)$. For parts \ref{ImpThm2} and \ref{ImpThm4}, we set $(c',d')=(c,c\to b')$. Since $c\not\leq b'$, we have $(c',d')\in P$, and since $\to$ is a protoimplication, $c\not\leq c\to b'$,  so $(c',d')\comp (c,d)$. Now consider any $(c'',d'')\in P$ with $(c',d')\comp (c'',d'')$. For parts \ref{ImpThm1} and \ref{ImpThm3}, $c''\not\leq d'=1\to b'$ and hence $c''\not\leq b'$ by the left-to-right inequality in Definition \ref{MatAlg}.\ref{MatAlg1}, so $c''\not\leq b$; similarly, for  parts \ref{ImpThm2} and \ref{ImpThm4}, $c''\not\leq d'=c\to b'$ and hence $c''\not\leq b'$ by Definition \ref{MatAlg}.\ref{MatAlg1} and Definition \ref{MatAlg}.\ref{MatAlg4}, so $c''\not\leq b$.  Hence part \ref{Good3} of Definition \ref{Good} holds. Thus, by Proposition \ref{CompRep}, $f$ is a complete embedding of $L$ into $\mathfrak{L}(P,\comp)$, which is a lattice isomorphism if $L$ is complete.

Next we claim that in each part, $ f(a\to b)= f(a)\to_\comp f(b)$. Suppose $(x,x\to y)\in  f(a\to b)$, so $x\leq a\to b$. Further suppose that $(x',x'\to y')\comp (x,x\to y)$ and $(x',x'\to y')\in f(a)$, so  $x'\leq a$. From $(x',x'\to y')\comp (x,x\to y)$,  we have $x\not\leq x'\to y'$. Now we claim that $b\not\leq x'\to y'$.  For if $b\leq x'\to y'$, then by Definition \ref{MatAlg}.\ref{MatAlg5}, \ref{MatAlg}.\ref{MatAlg4}, and \ref{MatAlg}.\ref{MatAlg3}, we have
\[x\leq a\to b\leq a\to (x'\to y')\leq x'\to (x'\to y')\leq x'\to y',\]
contradicting $x\not\leq x'\to y'$. For parts \ref{ImpThm1} and \ref{ImpThm2} of the theorem, we set $(x'',x''\to y'')=(b,b\to 0)$, so $(x'',x''\to y'')\in P$. For parts \ref{ImpThm3} and \ref{ImpThm4}, from $b\not\leq x'\to y'$, we obtain a nonzero $b'\in\mathrm{V}$ such that $b'\leq b$ and $b'\not\leq x'\to y'$, and we set $(x'',x''\to y'')=(b',b'\to 0)$. For part \ref{ImpThm5}, from $b\not\leq x'\to y'$, we have $b\wedge x' \not\leq y'$, so we obtain a $b'\in\mathrm{V}$ and $c'\in\Lambda$ such that   $b'\leq b\wedge x'$,  $y'\leq c'$, and  $b'\not\leq c'$, which together imply $b'\not\leq x'\to y'$. In this case, we set $(x'',x''\to y'')=(b',b'\to c')$. In each case, we have $(x'',x''\to y'')\in P$, $(x',x'\to y')\comp (x'',x''\to y'')$, and $(x'',x''\to y'')\in f(b)$. Hence $(x,x\to y)\in f(a)\to_\comp f(b)$. 
 
 Conversely, suppose $(x,x\to y)\in P\setminus f(a\to b)$, so $x\not\leq a\to b$. For parts \ref{ImpThm1} and \ref{ImpThm2} of the theorem, we set  $(x',x'\to y')= (a,a\to b)$, which immediately belongs to $P$ in part \ref{ImpThm1} and also belongs to $P$ in part \ref{ImpThm2} since if $a\leq b$, then $1\leq b\to b \leq a\to b$ using Definition \ref{MatAlg}.\ref{MatAlg4}, contradicting $x\not\leq a\to b$. For parts \ref{ImpThm3} and \ref{ImpThm4}, from $x\not\leq a\to b$, we have $a\not\leq x\to b$ by Lemma \ref{ImpUsefulLem}, so there is a nonzero $a'\in\mathrm{V}$ such that $a'\leq a$ but $a'\not\leq x\to b$, so $x\not\leq a'\to b$ by Lemma \ref{ImpUsefulLem}, and we set $(x',x'\to y')=(a',a'\to b)$.  For part \ref{ImpThm4}, we also have $a'\not\leq b$, for otherwise $a'\leq b\leq 1\to b\leq x\to b$ using Definition \ref{MatAlg}.\ref{MatAlg1} and \ref{MatAlg}.\ref{MatAlg4}, which contradicts what we derived above. Thus, in parts \ref{ImpThm1}-\ref{ImpThm4}, $(x',x'\to y')\in P$ and $(x',x'\to y')\comp (x,x\to y)$. Now suppose $(x',x'\to y')\comp (x'',x''\to y'')$, so $x''\not\leq x'\to y'=x'\to b$. It follows that $x''\not\leq 1\to b$ by Definition \ref{MatAlg}.\ref{MatAlg4} and then $x''\not\leq b$ by Definition \ref{MatAlg}.\ref{MatAlg1}, so $(x'',y'')\not\in f(b)$. Hence $(x,x\to y)\not\in f(a)\to_\comp f(b)$. For part \ref{ImpThm5}, from $x\not\leq a\to b$ we have $a\not\leq x\to b$ by Lemma \ref{ImpUsefulLem} and then $a\wedge x\not\leq b$, so there are $a'\in\mathrm{V}$ and $b'\in\Lambda$ such that (i) $a'\leq a\wedge x$, (ii)  $b\leq b'$, and  (iii) $a'\not\leq b'$; hence $(a',a'\to b')\in P$, and (i) and (iii) imply $x\not\leq a'\to b'$ and therefore $(a',a'\to b')\comp (x,x\to y)$. We set  $(x',x'\to y')=(a',a'\to b')$.  Then if $(x',x'\to y')\comp (x'',x''\to y'')$, so $x''\not\leq a'\to b'$, then $x''\not\leq b'$ as above and hence $x''\not\leq b$ by (ii), so $(x'',y'')\not\in f(b)$. Thus, $(x,x\to y)\not\in f(a)\to_\comp f(b)$.

Now for parts \ref{ImpThm1} and \ref{ImpThm2}, we show that if $\neg$ is an ultraweak pseudocomplementation, then $\comp$ is strongly pseudosymmetric.  Suppose  $(c,c\to d)\comp (a,a\to b)$, so $a\not\leq c\to d$. Hence $a\neq 0$, so $(a,a\to 0)\in P$, and $a\not\leq c\to 0$ by Definition \ref{MatAlg}.\ref{MatAlg5}, so $c\not\leq a\to 0$ by Lemma \ref{UsefulLem}.\ref{UsefulLem2}.   Thus, $(a, a\to 0)\comp (c,c\to d)$. Since $(a, a\to 0)$ and $(a,a\to b)$ have the same first coordinate, $(a, a\to 0)$ pre-refines $(a,a\to b)$ and~vice~versa. 

For parts \ref{ImpThm2}, \ref{ImpThm4}, and \ref{ImpThm5}, that $\comp$ is reflexive follows from the fact that if $\to$ is a protoimplication, then $a\not\leq b$ implies $a\not\leq a\to b$. For parts \ref{ImpThm2} and \ref{ImpThm4}, we also show that $\comp$ satisfies right pre-interpolation. Suppose $(x', x'\to y')\comp (x,x\to y)$, so $x\not\leq x'\to y'$. For part \ref{ImpThm2}, we let $z=x$. For part \ref{ImpThm4}, from $x\not\leq x'\to y'$, we obtain a nonzero $a\in \mathrm{V}$ such that $a\leq x$ and $a\not\leq x'\to y'$, and we let $z=a$. In either case, since $\to$ is a protoimplication,  $z\not\leq x'\to y'$ implies $z\not\leq z\to (x'\to y')$; then given $z\leq x$, we have $x\not\leq  z\to (x'\to y')$ as well.  Thus,  $(z, z\to (x'\to y'))\in P$ and $(z, z\to (x'\to y'))\comp (x,x\to y)$. Moreover, $(z, z\to (x'\to y'))$ post-refines $(x', x'\to y')$, for if $w\leq x'\to y'$, then $w\leq 1\to (x'\to y')\leq z\to (x'\to y')$ by Definition \ref{MatAlg}.\ref{MatAlg1} and \ref{MatAlg}.\ref{MatAlg4}; and since $z\leq x$,  $(z, z\to (x'\to y'))$ pre-refines $(x,x\to y)$.\footnote{For parts \ref{ImpThm2} and \ref{ImpThm4} when $\mathrm{V}=L$, we can take $z=x$, in which case $(z, z\to (x'\to y'))$ pre-refines $(x,x\to y)$ and vice versa, so a \textit{strong} right pre-interpolation property holds.}

For parts \ref{ImpThm3} and \ref{ImpThm4}, we show that $\comp$ satisfies left pre-interpolation. Suppose $(x', x'\to y')\comp (x,x\to y)$, so $x\not\leq x'\to y'$. Hence there is a nonzero $z\in\mathrm{V}$ such that $z\leq x$ but $z\not\leq x'\to y'$, so $(z,z\to (x'\to y'))\in P$. Moreover, from $z\not\leq x'\to y'$ it follows that $x'\not\leq z\to (x'\to y')$, for otherwise $z\leq x'\to (x'\to y')\leq x'\to y'$ by Lemma \ref{ImpUsefulLem} and Definition \ref{MatAlg}.\ref{MatAlg3}. Thus,  $(z,z\to (x'\to y'))\comp (x',x'\to y')$. Moreover, $(z,z\to (x'\to y'))$ post-refines $(x',x'\to y')$ and pre-refines $(x,x\to y)$ as in the previous paragraph.\footnote{For parts \ref{ImpThm3} and \ref{ImpThm4} when $\mathrm{V}=L$, we can take $z=x$, in which case $(z, z\to (x'\to y'))$ pre-refines $(x,x\to y)$ and vice versa, so a \textit{strong} left pre-interpolation property holds.} 

For part \ref{ImpThm4}, we show that if $\neg$ is a pseudocomplementation, then $\comp$ is weakly compossible. Suppose $(a,a\to b)\comp (c,c\to d)$, so $c\not\leq a\to b$ and hence $c\not\leq a\to 0$ by Definition \ref{MatAlg}.\ref{MatAlg5}, so $a\wedge c\neq 0$ since $\neg$ is pseudocomplementation. Hence there is a nonzero $e\in \mathrm{V}$ with $e\leq a\wedge c$. Then $(e, e\to 0)\in P$, and since $e\leq a$ and $e\leq c$, we have that $(e, e\to 0)$ pre-refines $(a,b)$ and $(c,d)$. Hence $\comp$ is weakly compossible.

Finally, for part \ref{ImpThm5}, that $\comp$ is compossible is proved in \citealt[Prop.~3.17(iii)]{Holliday2022}.\end{proof}

\noindent For part \ref{ImpThm5}, recall the equivalent definition of $\to_\comp$ in compossible reflexive frames from Footnote \ref{HeytingNote}.

Completeness theorems for \textit{preimplication logics} with a connective $\to$ obeying principles matching those of Definition \ref{MatAlg} can easily be obtained from Theorem \ref{ImpThm}, just as we obtained completeness theorems for logics with $\neg$ from Theorem \ref{NegThm}. It is also straightforward to add the quantifiers $\forall$ and $\exists$ (recall Theorem \ref{QuantComplete}) to such logics. One could attempt a systematic study of preimplicational logics (or the preconditional logics of \S~\ref{Conditionals}) analogous to the study of superintuitionistic logics (see \citealt{BH2019} and references therein), which can be seen as preimplicational (or preconditional) logics.

Finally, let us adapt the topological representation of \S~\ref{TopRep} to lattices with preimplications. Given a bounded lattice $L$ and a preimplication $\to$, define $\mathsf{FI}(L,\to)=(X,\comp)$ as follows: $X$ is the set of all pairs $(F,I)$ such that $F$ is a filter in $L$, $I$ is an ideal in $L$, and for all $a,b\in L$:
\[\mbox{if $a\in F$ and $b\in I$, then $a\to b\in I$.}\]
Then define  $(F,I)\comp (F',I')$ iff $I\cap F'=\varnothing$. When dealing with protoimplications, one can impose the additional condition on $X$ that $F\cap I=\varnothing$ (recall \S~\ref{TopRep}), thereby making $\comp$ reflexive.  Finally, given $a\in L$, let $\widehat{a}=\{(F,I)\in X\mid a\in F\}$, and let $\mathsf{S}(L)$ be $\mathsf{FI}(L,\to)$ endowed with the topology generated by $\{\widehat{a}\mid a\in L\}$.

\begin{theorem}\label{EmbedThmImp} For any bounded lattice $L$ and preimplication $\to$ on $L$, the map $a\mapsto\widehat{a}$ is 
\begin{enumerate}
\item\label{EmbedThmImp1} an embedding  of $(L,\to)$ into $(\lat(\mathsf{FI}(L,\to)),\to_\comp)$ and 
\item\label{EmbedThmImp2} an isomorphism from $L$ to the subalgebra of $(\lat(\mathsf{FI}(L,\to )),\to_\comp)$ consisting of $c_\comp$-fixpoints that are compact open in the space $\mathsf{S}(L)$.
\end{enumerate}
\end{theorem}

\begin{proof} First, we claim that for any $a,b\in L$, $(\mathord{\uparrow}a,\mathord{\downarrow}a\to b)\in X$. For suppose $c\in\mathord{\uparrow}a$ and $d\in \mathord{\downarrow}a\to b$, so $a\leq c$ and $d\leq a\to b$. Then by Definition \ref{MatAlg}.\ref{MatAlg4},  \ref{MatAlg}.\ref{MatAlg5}, and \ref{MatAlg}.\ref{MatAlg3}, we have
\[c\to d\leq a\to d\leq a\to (a\to b)\leq a\to b,\]
so $c\to d\in \mathord{\downarrow}a\to b$. Since by Definition \ref{MatAlg}.\ref{MatAlg1},  $a=1\to a$, it follows that $(\mathord{\uparrow}1,\mathord{\downarrow}a)\in X$ as well.

Now the proof that $\widehat{a}$ is a $c_\comp$-fixpoint and that $a\mapsto\widehat{a}$ is injective and preserves $\wedge$ and $\vee$ is the same as in the proof of Theorem \ref{EmbedThm}.  Obviously $\widehat{1}=X$ and $\widehat{0}$ is the set of all $(F,I)\in X$ such that $F$ is an improper filter, which is the set of absurd states (Definition \ref{AbsurdDef}); clearly if $F$ is improper, then $(F,I)$ is absurd, and conversely, if there is some element $a$ of $L$ not in $F$, so $a\neq 1$, then $(\mathord{\uparrow}1, \mathord{\downarrow}a)\comp (F,I)$, so $(F,I)$ is not absurd. 

Next we show that $\widehat{a\to b}=\widehat{a}\to_\comp \widehat{b}$. First suppose $(F,I)\in \widehat{a\to b}$, $(F',I')\comp (F,I)$, and $(F',I')\in\widehat{a}$, so $a\in F'$. Since $(F,I)\in \widehat{ a\to b}$, we have $a\to b\in F$, which with $(F',I')\comp (F,I)$ implies $ a\to b\not\in I'$, which with $a\in F'$ and the definition of $X$ implies $b\not\in I'$. Now let $F''=\mathord{\uparrow}b$ and $I''=\mathord{\downarrow} b\to 0$. Then $(F'',I'')\in X$, $(F',I')\comp (F'',I'')$, and $(F'',I'')\in \widehat{b}$. Thus, $(F,I)\in \widehat{a}\to_\comp\widehat{b}$. Conversely, if $(F,I)\not\in \widehat{a\to b}$, so $a\to b\not\in F$, then setting $(F',I')=(\mathord{\uparrow}a,\mathord{\downarrow}a\to b)$, we have $(F',I')\in X$ and $(F',I')\comp (F,I)$. Now consider any $(F'',I'')$ such that $(F',I')\comp (F'',I'')$, so $a\to b\not\in F''$. Then since $b=1\to b\leq a\to b$ by Definition \ref{MatAlg}.\ref{MatAlg1} and \ref{MatAlg}.\ref{MatAlg4}, we have $b\not\in F''$, so $(F'',I'')\not\in\widehat{b}$. Thus, $(F,I)\not\in \widehat{a}\to_\comp\widehat{b}$.

The proof of part \ref{EmbedThmImp2} is the same as the proof of Theorem \ref{EmbedThm}.\ref{EmbedThm2}.\end{proof}

Under the assumption that $\to$ satisfies stronger axioms as in Definition \ref{ImpDefs}, one can prove that $\mathsf{FI}(X,\to)$ satisfies corresponding properties in Lemma \ref{ImpCorr} (cf.~Proposition \ref{FIprop}).

\bibliographystyle{plainnat}
\bibliography{fundamental}

\end{document}